\documentclass[reqno]{amsart}
\usepackage{amsfonts}
\usepackage{amsmath}
\usepackage{amssymb}
\usepackage{mathrsfs}
\usepackage{enumitem}
\usepackage{slashed}
\usepackage{mathtools}
\usepackage{graphicx, subcaption, float}
\usepackage[dvipsnames]{xcolor}
\usepackage[colorlinks=true,citecolor=blue]{hyperref}
\usepackage{comment}

\usepackage{cite}

\newtheorem{theorem}{Theorem}[section]
\newtheorem{proposition}{Proposition}[section]
\newtheorem{lemma}[theorem]{Lemma}

\theoremstyle{definition}

\theoremstyle{remark}
\newtheorem{remark}[theorem]{Remark}
\numberwithin{equation}{section}
\newcommand{\proj}{\mathscr{P}_{\rm c}}
\newcommand{\projpm}{\mathscr{P}}
\newcommand{\projz}{\mathscr{P}_{0}}

\newcommand{\R}{\mathbb{R}}
\newcommand{\Z}{\mathbb{Z}}

\newcommand{\C}{\mathbb{C}}
\newcommand{\D}{\partial}
\newcommand{\vD}{{v}_{\rm D}}
\newcommand{\sD}{\slashed{D}_0}
\newcommand{\sDT}{\slashed{D}(T)}
\newcommand{\dD}{\mathcal{D}}
\newcommand{\Tper}{T_{\rm per}}
\newcommand{\sumj}{\sum\limits_{j=1,2}}

\newcommand\numberthis{\addtocounter{equation}{1}\tag{\theequation}}
\graphicspath{{Figures/}}

\begin{document}

\title{Radiative decay of edge states in Floquet media}

% Author
\author[]{Sameh N.\ Hameedi}
\address{Department of Applied Physics and Applied Mathematics, Columbia University, New York, NY 10027, USA}
\email{sh3982@columbia.edu}
\author[]{Amir Sagiv}
\address{Department of Applied Physics and Applied Mathematics, Columbia University, New York, NY 10027, USA}
\email{asagiv88@gmail.com}

\author[]{Michael I.\ Weinstein}
\address{Department of Applied Physics and Applied Mathematics and Department of Mathematics, Columbia University, New York, NY 10027, USA}
\email{miw2103@columbia.edu}

%\thanks{}

%    \subjclass is required.
%\subjclass[2010]{Primary }

%\date{\today}
%\dedicatory{}

% Abstract 
\begin{abstract}
We consider the effect of time-periodic forcing on a one-dimensional Schr{\"o}dinger equation with a topologically protected defect (edge) mode. The unforced system models a domain-wall or dislocation defect in a periodic structure, and it supports a defect mode which bifurcates from the Dirac point (linear band crossing) of the underlying bulk medium. We study the robustness of this state against time-periodic forcing of the type that arises in the study of Floquet Topological Insulators in condensed matter, photonics, and cold-atoms systems. Our numerical simulations demonstrate that under time-periodic forcing of sufficiently high frequency, the defect state undergoes radiative leakage of its energy away from the interface into the bulk; the time-decay is exponential on a time-scale proportional to the inverse square of the forcing amplitude. The envelope dynamics of our Floquet system are approximately governed, on long time scales, by an effective (homogenized) periodically-forced Dirac equation. Multiple scale analysis of the effective envelope dynamics yields an expansion of the radiating solution, which shows excellent agreement with our numerical simulations. 
\end{abstract}

\maketitle

%\tableofcontents

\section{Introduction} \label{introduction}

Consider wave propagation in a periodic and non-dissipative medium. In such a medium, spatially localized  defects or extended interfaces often give rise to {\it defect modes}, states in which energy can concentrate. These are described by time-harmonic solutions of the underlying wave equations.  In the field of topological insulators, a class of defect modes of particular interest are those whose existence is tied to a linear or conical touching of spectral bands ({\it Dirac points}) of an underlying bulk periodic medium, 
and whose robustness
 against perturbations is owed to the breaking of time reversal symmetry and the associated
non-triviality of a topological invariant %(in 2D periodic systems, the Chern number)
 associated with the band structure. Such robust states are of broad interest in applications to storage and transmission of information.

Recently there has been great interest in {\it Floquet materials}, in which time-periodic forcing is applied to spatially periodic media in settings like those above \cite{cayssol2013floquet, ozawa2019topological}.
As with other perturbations which break time-reversal symmetry, Floquet systems give rise to chirality (uni-directionality) of modes and topological stability against defects. However, Floquet materials have the advantage of a larger design-space,
allowing one to tune a system to convert between being topologically trivial and non-trivial \cite{dal2015floquet}, and even be more conducive
 to nonlinearly localized modes \cite{ablowitz2015adiabatic, ablowitz2017tight, ivanov2020edge,  mukherjee2020observation}.
Floquet systems also display a  greater variety of topological properties than their static counterparts; these systems are periodic in one additional dimension and hence have a larger family of topological invariants defined on their high dimensional Brillouin zone \cite{roy2017periodic, sadel2017topological}.

The topological robustness of defect states in such systems has been explored  in experiments \cite{fleury2016floquet, ozawa2019topological, plotnik2013observation, wang2013observation}, and theoretically, mainly in the context of tight-binding (discrete) models \cite{asboth2014chiral, dal2015floquet, graf2018bulk, rudner2013anomalous}.
Simulations of one- and two-dimensional optical systems show Floquet edge modes that persist over long propagation distances, but eventually disperse and decay \cite{guglielmon2018photonic, ivanov2021topological, rechtsman2013photonic}. In a class of hierarchical replica models, edge currents have been shown to exist for arbitrarily long, but finite, time scales \cite{bal2021multiscale}.
%\footnote{
%\textcolor{blue}{ did these references discuss the decay of \underline{topological} states?}\\
%\textcolor{red}{Yes, as far as I understand. In what is now the seminal Rudner Lindner \cite{rudner2013anomalous}, the calculate the winding number of the bulk driven system, but then as far as I understand show in Section IV.D. and FIG 3 that there is a correspondence between that index and the number of {\em numerically computed} chiral edge modes. They go further in Section IV and provide an alternative way to count modes, but that goes through a sort of replica argument (truncating the bandwidth of the Hamiltonian in temporal-Fourier expansion), not unlike what is done for forced SSH models in Dal Lago et al \cite{dal2015floquet} and Delplace and collaborators in \cite{asboth2014chiral}, who begin their outlook section with ``the topologically protected states our theory predics should have...''.}\\
%\textcolor{red}{I would not like this current paper to be a bashing of existing literature, partly because I do not understand the topological piece well enough (maybe something to do next) but I think this is a big speculation, if not a mistake.}
%}

{\it In this article we show, in a class of models, that defect states which exhibit topological robustness in 
the autonomous (unforced) setting are only metastable against time-periodic parametric forcing; on a sufficiently 
long time scale, such states resonantly couple to bulk (radiation) modes and radiate their energy away from the defect and into the bulk. We derive an asymptotic theory based on a parametrically forced effective (homogenized) Dirac operator, which gives predictions in agreement with our numerical simulations of both the Schr{\"o}dinger and the effective Dirac models. Our numerical simulations cover a broader set of models than our analysis; while smoothness plays a key role in the derivation, the phenomenon is demonstrated for non-smooth domain walls, which are often more similar to experimental settings.

 Our results stand in contrast to the common assertions in the  topological insulators literature. There, it is common to consider a {\em discrete} (tight-binding) model with finitely many bands, and then its temporally-driven analog [4,9,18,36]. These works either compute numerically or assert the existence of {\em Floquet edge states}, i.e., point-spectrum of the Floquet Hamiltonian (or the monodromy, the period-evolution operator). But such tight-binding models are phenomenological, or derived as approximations of a PDE, e.g., the Schr{\"o}dinger equation. Do Floquet edge states even exist in the generic PDE model? Our results  suggest that, to the contrary, only metastable modes exist, and those eventually decay into the bulk due to the resonant effect of forcing. It is an interesting open question, beyond the scope of this current paper, whether these meta-stable resonant modes can be understood as topologically protected, or whether this phenomena is universal in PDE models.
}
%\footnote{\textcolor{blue}{ confused about the statement. do we cite assertions in the physics community about infinite time floquet states? for example our analytical and numerical results would seem to support the assertions described in bold face above.}\\
%\textcolor{red}{The physics literature papers cited above talk about modes of the Hamiltonian. If they are modes, they are modes for infinite times. If they are long-lived wavepackets, that is something else (and they make no such analysis). The more rigorous statement is found in Tauber and Graf \cite{graf2018bulk} in 2d (not 1d). They define edge and bulk indexes for the time-dependent problem and show their equivalence, but no claim of modes in any space is made.}}

\subsection{Overview of the model}\label{overview}

Our unforced Hamiltonian,   $H^{\varepsilon}_{\rm dw}\equiv -\D_x^2+U_\varepsilon(x)$, is an asymptotically periodic 
Schr{\"o}dinger operator; i.e.,
 \begin{align*}
H^{\varepsilon}_{\rm dw} =
\begin{cases} H^{\varepsilon}_+&=-\D_x^2+V(x)+\varepsilon W(x)\quad \textrm{ for $x\gg+1$}\\ 
  H^{\varepsilon}_-&= -\D_x^2+V(x)-\varepsilon W(x)\quad  \textrm{for $x\ll-1$,}
  \end{cases}
  \end{align*}
Here,  $H^{\varepsilon}_+$ and $H^{\varepsilon}_-$ are $\Z-$ periodic. Both $H^{\varepsilon}_+$ and $H^{\varepsilon}_-$ 
 are perturbations of the underlying bulk operator $H=-\D_x^2+V(x)$, which has a linear crossing in its band structure at energy $E_D$ (Dirac point). $H^{\varepsilon}_\pm$ have  a common spectral gap 
   (called the bulk gap) of width 
  $\mathcal{O}(\varepsilon)$ about $E=E_D$, which is induced by the small but spatially periodic perturbations $ \pm \varepsilon  W(x)$.  The operator $H^{\varepsilon}_{\rm dw}$  interpolates  between $H^{\varepsilon}_-$ 
   at $x=-\infty$ and $H^{\varepsilon}_+$ 
   at $x=+\infty$, via a {\it domain wall}.  The Hamiltonian $H^{\varepsilon}_{\rm dw}$ has, for all $\varepsilon$ small,  a protected  mid-gap eigenstate $\psi_\star^\varepsilon(x)$ of energy $E_D+\mathcal{O}(\varepsilon^2)$. Moreover, $\psi_\star^\varepsilon(x)$ has a multi-scale structure; with respect to the band structure of $H$, it is (to an excellent approximation) a wave-packet which is spectrally concentrated about the Dirac point with an envelope characterized   by the zero energy eigenstate of an effective (homogenized) Dirac operator. A  more detailed discussion of the model and its analytic properties is given in Section \ref{model}. A  realization of this model in photonic waveguides has been studied in \cite{LT-etal:16}.

In this article we study the initial value problem for the {\em parametrically forced Schr{\"o}dinger equation }
\begin{subequations}\label{eq:lsa_gen}
 \begin{equation}
i\psi_t = \left(H_{\rm dw}^\varepsilon + 2i\varepsilon A(\varepsilon t)\partial_x\right)\psi ,\qquad 0<\varepsilon\ll1 \, , 
\end{equation}
 where 
\begin{equation}
 A(T) \equiv 2\beta \cos (\omega T) \, , \qquad \omega >0 \, .
\end{equation}
\end{subequations}

%where $A(T)$ is real-valued and periodic. 
   The perturbing operator $2i\varepsilon A(\varepsilon t)\partial_x$ arises from an  {\it effective vector potential},~$\varepsilon A(\varepsilon t)$, induced by a  time-dependent deformation of $H_{\rm dw}^\varepsilon$. In Section \ref{sec:physics}, we  derive
   \eqref{eq:lsa_gen} in the setting of a coupled array of optical waveguides. This model is also relevant in the context of a coupled array of trapped ultracold fermions \cite{jotzu2014experimental}.     We ask the following: 
   \begin{center}
\noindent {\it  
 Question:\    Does the topologically protected  edge state $\psi_\star^\varepsilon$ of $H_{\rm dw}^\varepsilon$ persist in the presence of time-period  forcing ($A(t)\neq 0$)? }
 \end{center}
    
\begin{figure}[h]
          \begin{subfigure}[t]{1\textwidth}
            \caption{Unforced}
            \centering
            \includegraphics[width=.35\linewidth]{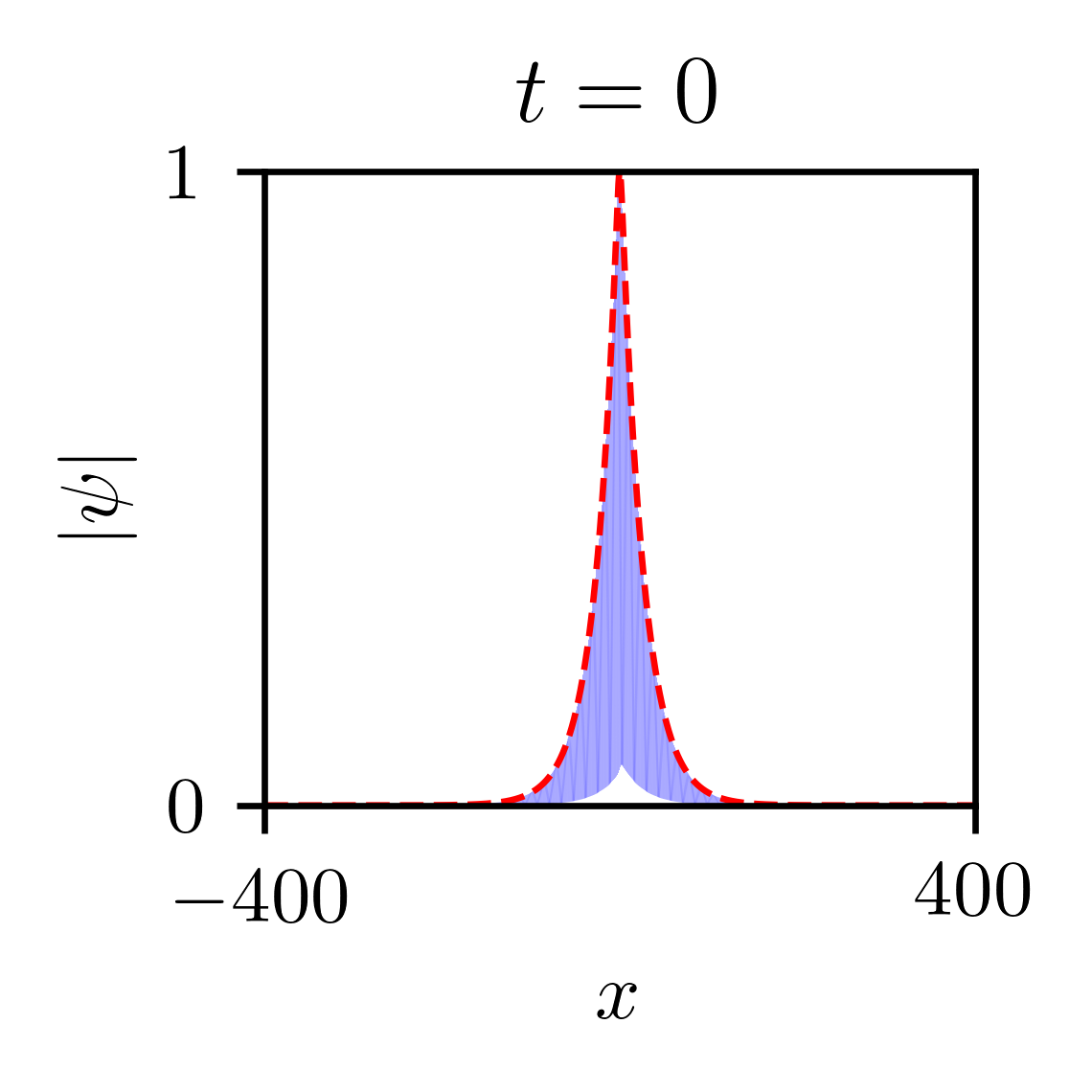}
             \hskip -4ex
            \includegraphics[width=.35\linewidth]{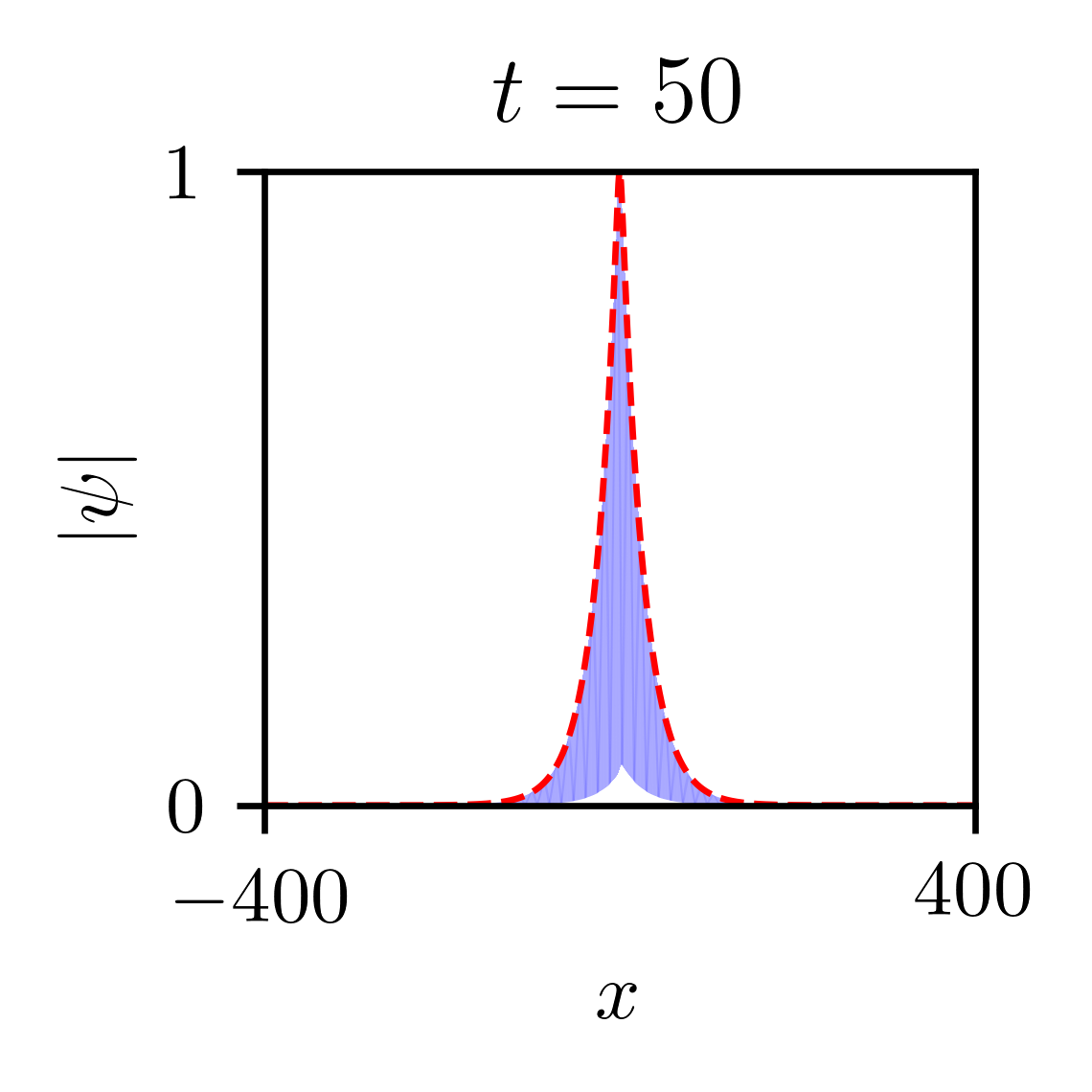}
             \hskip -4ex
            \includegraphics[width=.35\linewidth]{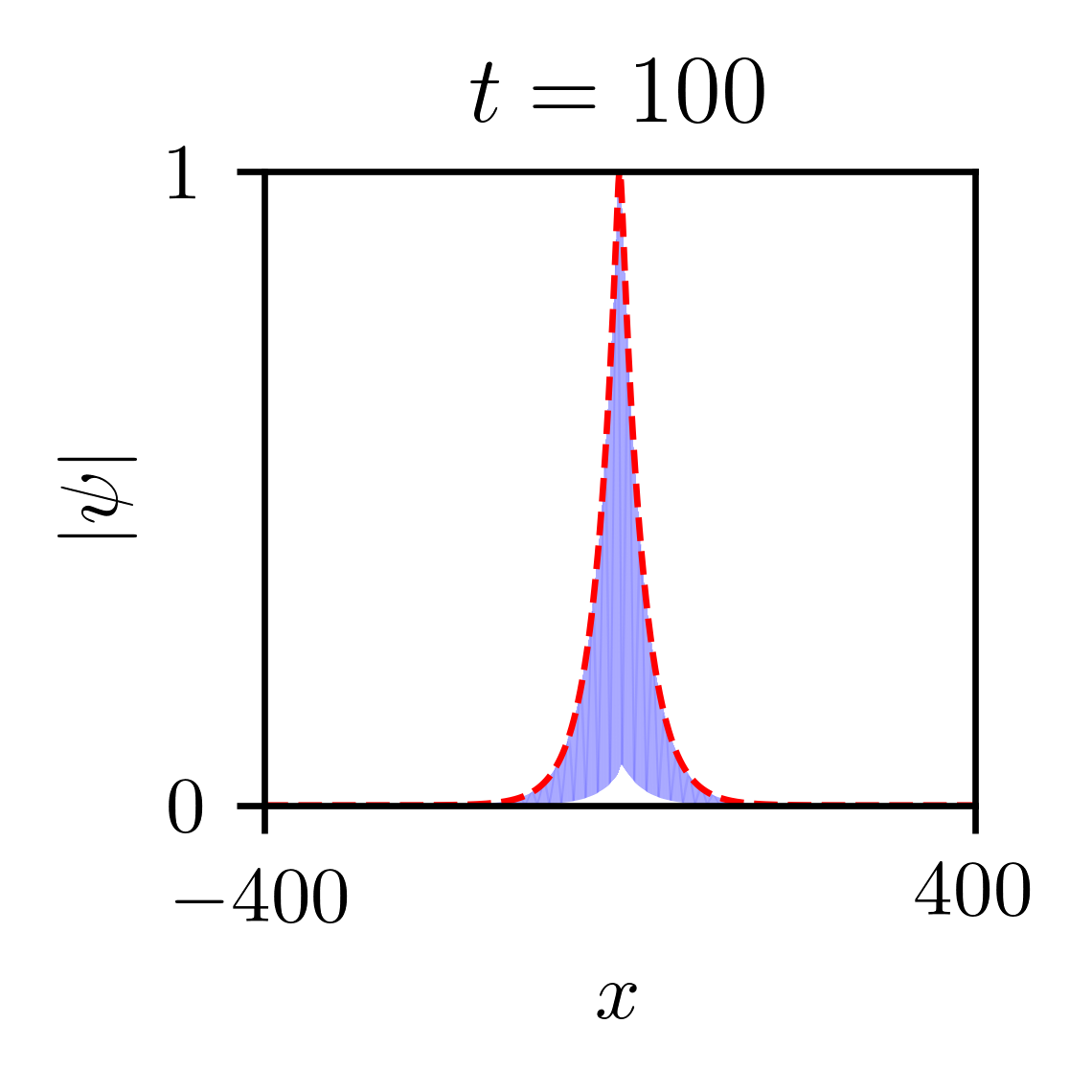}
            \label{fig:DirSch_cos_unf-intro}
         \end{subfigure}
         \begin{subfigure}[t]{1\textwidth}
            \caption{Forced}
            \centering
            \includegraphics[width=.35\linewidth]{sch_0.png}
              \hskip -4ex
            \includegraphics[width=.35\linewidth]{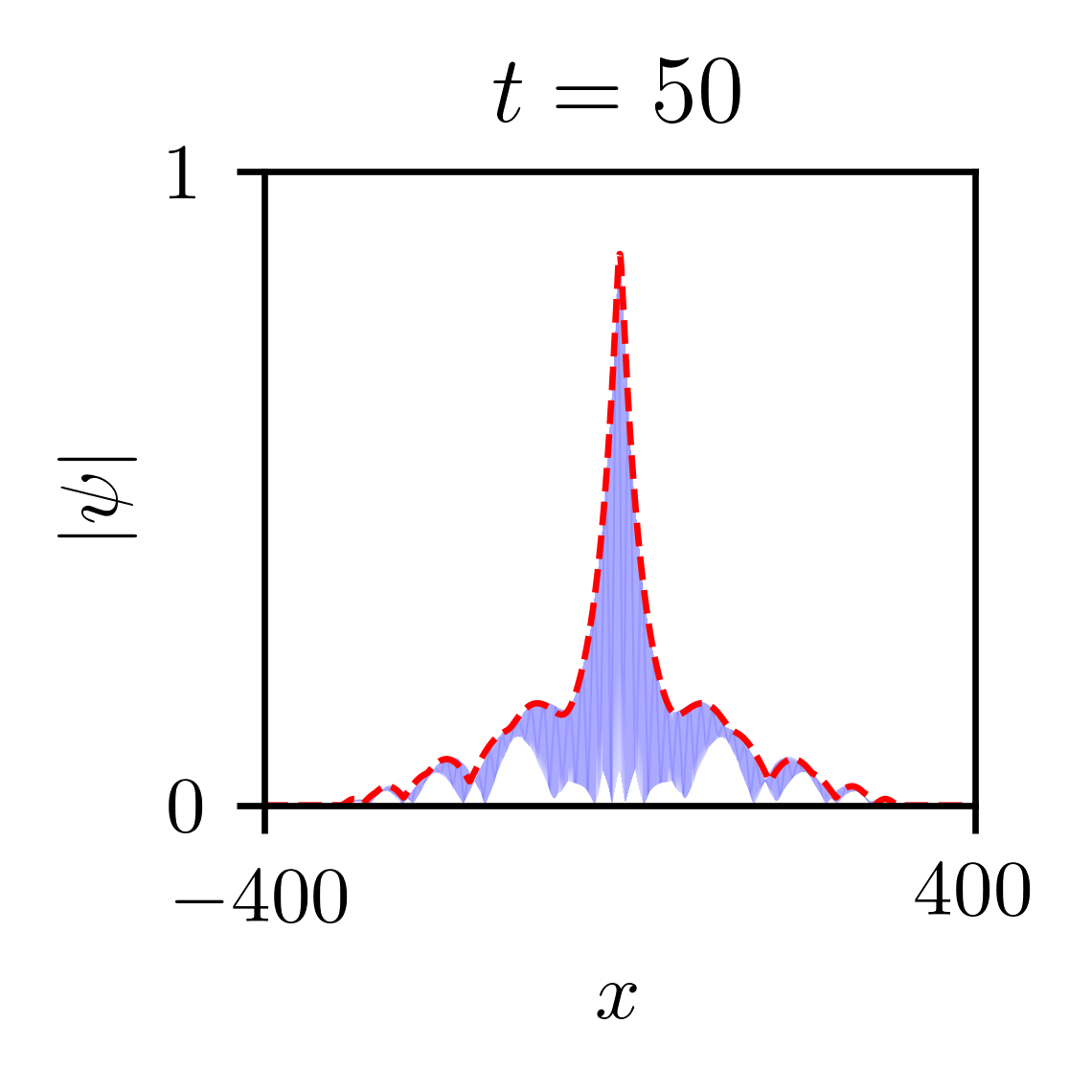}
             \hskip -4ex
            \includegraphics[width=.35\linewidth]{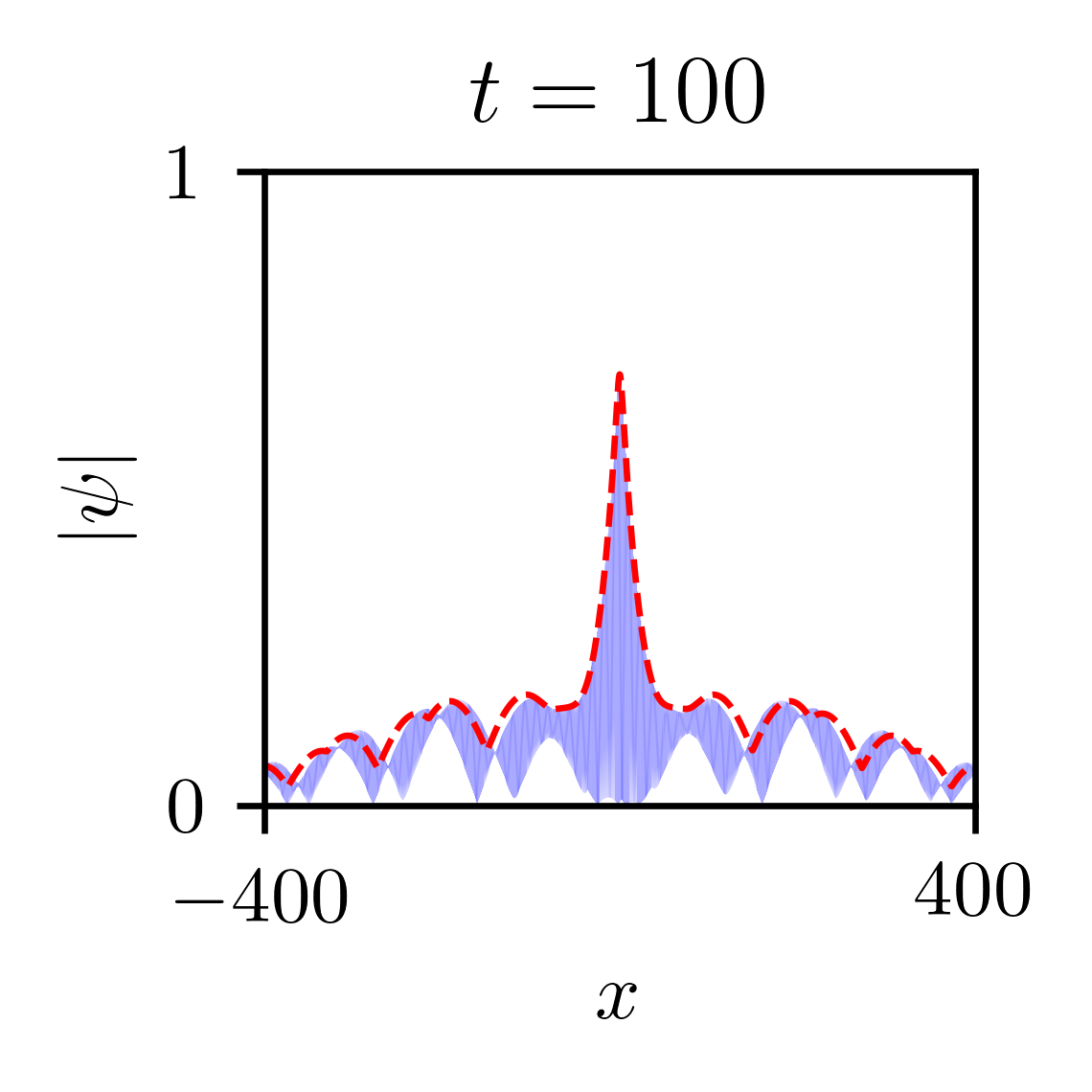}
            \label{fig:DirSch_cos_for-intro}
        \end{subfigure}
        \caption{{\bf (a)} Numerically computed evolution of edge state, $\psi_\star^\varepsilon$, under the unforced Schr{\"o}dinger equation, \eqref{eq:lsa_gen} with $A(T)\equiv0$, (blue, solid) and its envelope evolution approximated by the effective (unforced) Dirac equation (red, dashes). {\bf (b)} Periodically-forced Schr{\"o}dinger evolution ($A(T)\ne0$) and its envelope approximated by an effective periodically-forced Dirac dynamics; see Sections \ref{numerics} and \ref{sec:Dirac_simul}.}
\label{fig:DirSch_cos-intro}
\end{figure}

\subsection{Summary of analytical and numerical results}\label{summary}
   
   \begin{enumerate}[label=(\Alph*)]
\item \textbf{Defect mode decay in the Schr{\"o}dinger equation.}
In Figure \ref{fig:DirSch_cos-intro} we contrast the time-evolution \eqref{eq:lsa_gen} for the initial data $\psi(x,0)=\psi_\star^\varepsilon$, in the cases of unforced ($A=~0$) and forced ($A\ne0$) dynamics; a detailed discussion is given in Section~\ref{numerics}. The top row of Figure \ref{fig:DirSch_cos-intro} is consistent with the persistence of the edge 
 state; indeed, for all $t$ we have $\psi(t,x)=e^{-iE_\varepsilon t}\psi_\star(x)$.
The slow time-decay of the solution is demonstrated in the bottom row of Figure~\ref{fig:DirSch_cos-intro}. For a forcing $\beta A(T)= \beta\cos(\omega T)$ with $\beta$ small and fixed, the envelope decays at an exponential rate $\approx \exp(-\beta^2\varepsilon \Gamma_0 t)$, for some $\Gamma_0= \Gamma_0(\omega)>0$. This type of decay is observed only for driving frequencies $\omega$ above some threshold frequency, which we discuss in item (C) below.
 \item  \textbf{Effective Dirac equation.} In Section \ref{sec:dirac} we derive and prove the validity of an effective (spatially homogenized) time-periodically forced Dirac equation (Theorem \ref{thm:valid} and \cite{SW21}), 
as an approximation to the envelope dynamics of forced Schr{\"o}dinger evolution
for initial data which are spectrally localized near the Dirac point. The Dirac dynamics accurately approximate the envelope dynamics
for data corresponding to the multiple scale edge state~$\psi_\star^\varepsilon(x)$.
      
     The time-dependent Schr{\"o}dinger equation \eqref{eq:lsa_gen}  excites a wide range of spatial and temporal scales. In contrast,  the effective Dirac approximation is comparatively easy to solve numerically on long time-scales; its numerical solution does not require the simultaneous resolution of multiple temporal and spatial scales. 
       
A numerical comparison between the Schr{\"o}dinger equation with $\psi(0,x)=~\psi_\star^\varepsilon$
     and the effective forced Dirac equation, for data given by the envelope of $\psi_\star^\varepsilon$,  shows excellent agreement on long time scales. In Figure~\ref{fig:DirSch_cos-intro} we plot both the numerically computed solution of the (multi-scale)
      Schr{\"o}dinger equation \eqref{eq:lsa_gen} and the solution of the effective forced Dirac equation. Figure~\ref{fig:DirSch_cos-intro}
       demonstrates that the Schr{\"o}dinger envelope is very well tracked by the (slowly varying) solution of the effective forced Dirac equation.
For nontrivial forcing $A(T)$, the envelope-decay in the effective Dirac equation, matches that of the Schr{\"o}dinger equation on large time-scales.

\begin{figure}[h]
	\begin{subfigure}[t]{1\textwidth}
            \centering
             \hskip -1ex
            \includegraphics[width=.7\linewidth]{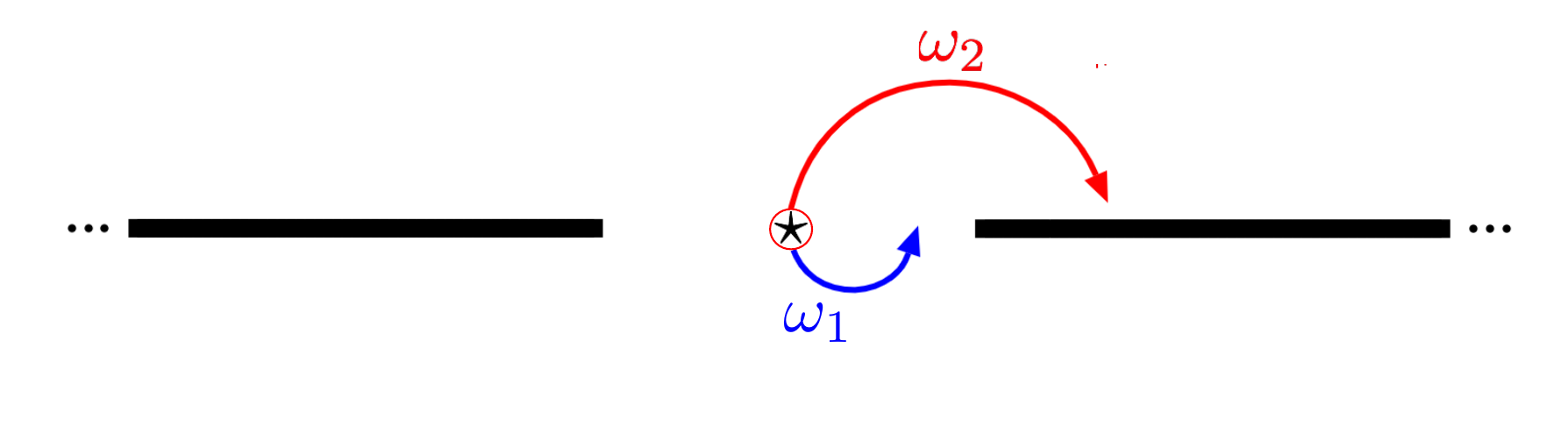}
            \label{fig:dpec_D0_omega1}
        \end{subfigure} \vskip 2ex
        \begin{subfigure}[t]{1\textwidth}
            \centering
            \hskip - 5ex
            \includegraphics[width=.7\linewidth]{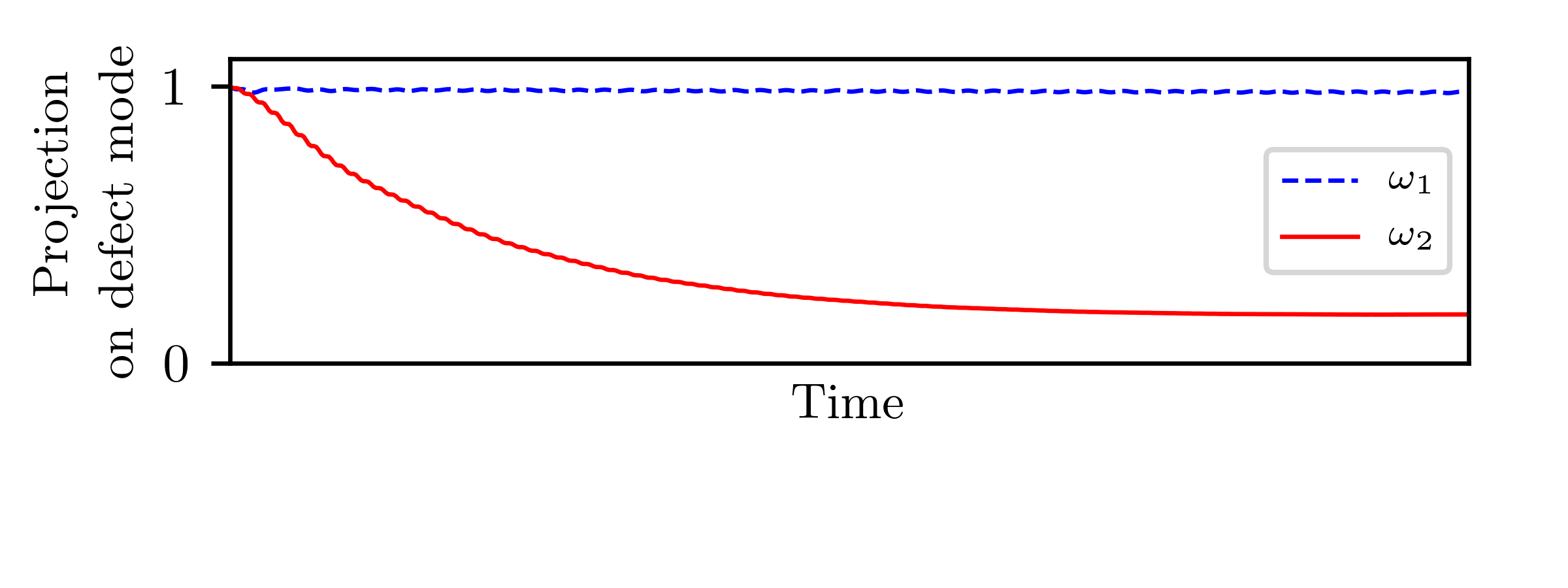}
            \label{fig:omega}
        \end{subfigure} \vskip -2ex
  	\caption{\textbf{Top:} The spectrum of the unforced Hamiltonian (star and bold lines denote point and continuous spectrum, respectively). \textbf{Bottom:} Simulations of the effective forced Dirac equation \eqref{eq:diracA}, see details in Section \ref{sec:Dirac_simul}. When the frequency of the forcing is small ($\omega_1$, blue dashes), the defect mode does not couple to the continuum modes, and the mode's localization persists. When the frequency is sufficiently large to couple the point and continuous spectrum,
	 i.e., $\omega_2>\textrm{bulk spectral gap}$,  (solid red), power is transferred between the localized mode and the radiation; the defect mode decays.}
  	\label{fig:decay_cartoon_Intro}
\end{figure}
\item {\bf Multiscale analysis of radiation damping.} Our results on the approximation 
 of \eqref{eq:lsa_gen} by an effective Dirac equation, and an asymptotic solution ($\beta$ small) of 
  the effective Dirac equation, imply that on large and finite time-scales, $0\le t\le \varepsilon^{-1}\beta^{-2}$, the wave-packet decays
  exponentially 
  \begin{equation}  \Big| \textrm{Projection of $\psi^{\varepsilon}(t, \cdot)$ on $\psi_\star^\varepsilon$}\Big|\ \approx e^{-\Gamma_0(\omega) \beta^2 \varepsilon  t}  \, .\label{Sch-decay}\end{equation}

The mechanism of decay is radiation damping;
   parametric forcing resonantly couples the edge state to radiation modes associated with the continuous spectrum, which acts as an energy sink; see Figure~\ref{fig:decay_cartoon_Intro}.  The radiation rate is given by a variant of the {\it Fermi golden rule}, see, e.g., \cite[Chapter 5]{sakurai2020modern}.
  A derivation of this radiation damping effect and, in particular, its rate is given in Section \ref{sec:decay_anal}. Specifically, we show a threshold effect: if the driving frequency $\omega$ in the effective Dirac operator is larger than half the spectral gap width, %ensuring   coupling to radiation modes at order $\beta^2$ in the perturbation theory, 
   then radiative decay takes place on the time scale $\beta^{-2}$.\footnote{If $\omega$ does not satisfy this resonance condition, then the time scale on which radiative decay takes place is expected to be $T\approx \beta^{-2r}$, for some $r\ge2$, with a corresponding result for the Schr{\"o}dinger evolution.} %Correspondingly in the Schr{\"o}dinger evolution \eqref{Sch-decay}, showing decay on the time-scale $\varepsilon^{-1}\beta^{-2}$.    
  \end{enumerate}
  \subsection{Derivation of the  effective vector potential, $A(T)$}\label{sec:physics}

  \begin{figure}[h]
	\begin{subfigure}[t]{0.45\textwidth}
            \centering
             \hskip -1ex
            \includegraphics[width=1\linewidth]{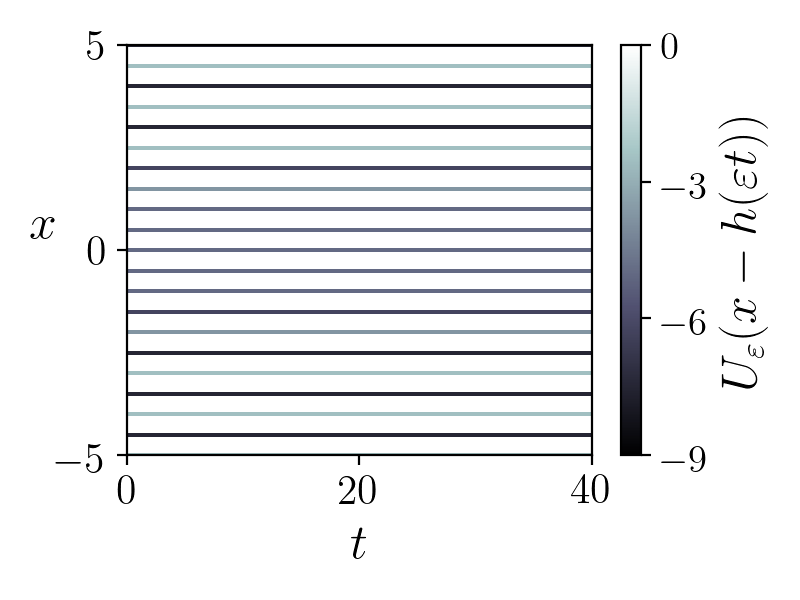}
 %           \label{fig:dpec_D0_omega}
        \end{subfigure} 
        \begin{subfigure}[t]{0.45\textwidth}
            \centering
            \hskip - 3ex
            \includegraphics[width=1\linewidth]{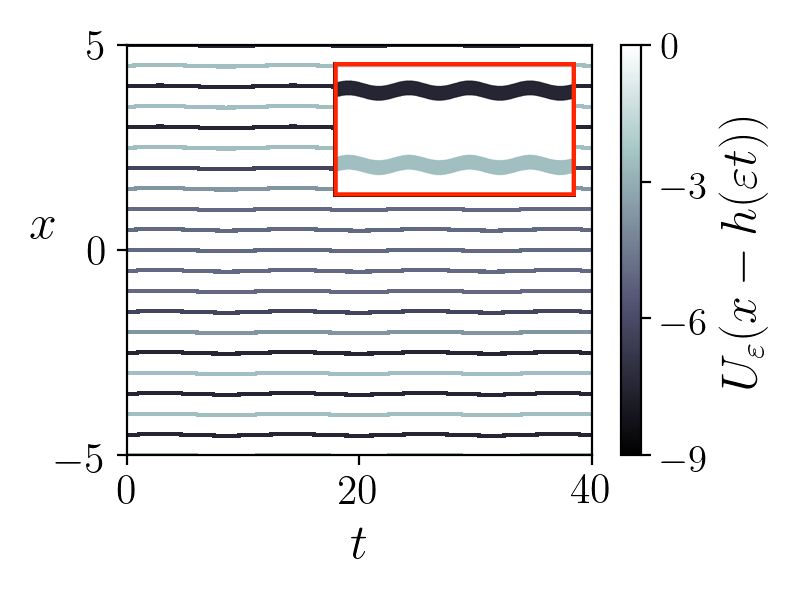}
  %          \label{fig:omega}
        \end{subfigure} 
  	\caption{{\bf Left:} The domain-wall potential $U_{\varepsilon}(x)$, see \eqref{eq:H0} (in particular $U^{(2)}_{1/2}$, see App.\ \ref{ap:potentials}). {\bf Right:} The coiled waveguides potential $U_{\varepsilon}(x-h(t))$ with $h(t)=2\omega^{-1}\beta\sin(\omega\varepsilon t)$ with a zoom-in in the inset.}
  	\label{fig:coiling}
\end{figure}

We briefly explain how an  effective vector potential in \eqref{eq:lsA}, $A(t)$, arises in the context of nearly monochromatic light-propagation in a planar array of undulating waveguides; see, for example, \cite{ablowitz2015adiabatic, ablowitz2017tight,jurgensen2021quantized, plotnik2013observation}. Consider a planar array of waveguides, whose index of refraction variations along the transverse ($x$) direction are described by the time-independent potential $U_{\varepsilon}(x)$; see the left panel of Fig.\ \ref{fig:coiling}. Illumination of the array on its input side (left) by a continuous wave (CW, nearly monochromatic) laser results in propagation along the longitudinal direction (a time-like direction, labeled~$t$) which is approximated by the paraxial Schr{\"o}dinger equation $i\psi_t=(-\D_x^2+U_\varepsilon)\psi$. Consider now a planar array of {\em undulating} waveguides, where the propagation is given by 
$$i\partial_t \psi(t,x) = \left[-\partial^2_x  + U_{\varepsilon}(x-h(t))\right]\psi \, .$$
Here $h:\R \to \R$ describes the center-point of waveguides in the array; see right panel of Fig.\ \ref{fig:coiling}. Switching to the undulating system of coordinates $\tilde{x}=x-h(t)$, $\tilde{t}=t$, one gets 
$$\left[i\partial_{\tilde{t}} -ih'(t)\partial_{\tilde{x}}\right] \psi(\tilde{t},\tilde{x}) = \left[-\partial^2_{\tilde{x}}  + U_{\varepsilon}(\tilde{x}) \right]\psi \, .$$
Dropping the tilde notation and setting $h'(t)\equiv 2A(t)$, we obtain \eqref{eq:lsa_gen}. 

Note that by the further change of variables: $\phi \equiv \psi \exp(-i\int_0^t |A(s)|^2 \, ds)$ we obtain a Schr{\"o}dinger equation $i\partial_t \phi = (\partial_x + iA(t))^2\phi + U_{\varepsilon}(s)\phi$  with a vector-potential term. By Maxwell's equations, a spatially-independent  vector potential $A(t)$ gives rise to a time-periodic {\em electric} field and {\em zero magnetic field}.

 We comment here that, from a physical point of view, $A(T)$ could be {\em any} time-dependent function, or even spatially dependent (if the coiling is not uniform in $x$). However, in the context of demonstrating and deriving the Fermi Golden Rule, we focus on the simple periodic forcing $A(T) = \beta \cos (\omega T)$. For a study of multi-frequency (almost periodic) Fermi Golden Rule radiation damping, see \cite{kirr2003metastable}.

\subsection{Future directions and open problems}\label{sec:discussion} 

\begin{itemize}
\item {\bf Analytic challenges in rigorously establishing radiative decay.} In our study of the time-decay of defect modes for the effective periodically-forced Dirac Hamiltonian, $\slashed{D}(T)=\sD+A(T)\sigma_1$, we derive a formal asymptotic solution, whose validity on the radiation damping time-scale is supported by estimating the  expansion terms on the appropriate time-scale. A fully rigorous analysis of this derivation is an open question, which 
cannot be addressed by an application of previous methods developed in the study of radiative decay in linear and nonlinear radiation damping problems; see, e.g., the time-dependent resonance approach in \cite{miller2000metastability, SW:98,costin2001resonance}. 

 We elaborate on the analytic issues. The first challenge arises from 
 the matrix character of the effective Dirac equation;  the forcing operator, $\beta A(T)\sigma_1$, 
  does not commute with the unperturbed Dirac operator $\sD$ and hence cannot be removed via a change of phase. 
Even if $\beta$ is assumed to be 
 small, this operator cannot be treated perturbatively; indeed there is no spatial localization in the perturbing operator, which would have allowed for the control of this perturbation via local energy time-decay estimates for $\exp(-i\sD t)$ (on the continuous spectral part of $\sD$).  What could succeed is a variant of the above methods for \underline{Floquet} Dirac Hamiltonians; in particular, dispersive estimates for forced Dirac Hamiltonians and the necessary spectral decomposition with which to reduce to an appropriate parametrically forced particle-field model (compare with our analysis in Section \ref{sec:ld-est}). Research in this direction is currently in progress.

\item {\bf Resonant decay in discrete models.} Our study of resonant decay of defect modes is in the context of continuum PDE models. However, the question of existence of Floquet edge modes also appears in the context of discrete (e.g., tight-binding ) models in the physics literature \cite{asboth2014chiral, dal2015floquet, rudner2013anomalous}. For concreteness, consider the Su-Schrieffer-Heeger (SSH) model, which for a certain parameter regime has a localized edge mode inside a bulk spectral gap. In the numerical investigations of \cite{dal2015floquet}, 
the existence and non-existence  of defect modes of time-periodically driven SSH, and the dependence on the driving frequency, is studied. We believe that our approach is applicable to the analysis of such Floquet models.

\item {\bf Non-smooth potentials.} Finally we believe that our smoothness assumptions on the coefficients of $H^{\varepsilon}_{\rm dw}$ can be relaxed.
However, our discontinuous-transition domain wall / dislocation model, $U^{(3)}_{\varepsilon}$,  (see Section \ref{sec:sgn_sch} and Appendix \ref{ap:potentials}) violates
the assumption of a slowly varying domain wall transition. In this case, our derivation of an effective Dirac operator 
and hence the asymptotic theory would seem not to apply. Nevertheless, numerical investigation shows that the midgap mode  persists and undergoes radiation damping on the time-scale of the previous examples. Furthermore, and remarkably, the damped envelope dynamics are well-described by the effective Dirac operator $\slashed{D}^{(3)}(T)$. It would be of interest to understand this robustness of radiative phenomena against large deformation, i.e., in regimes other than those allowing for a multi-scale structure.

One natural setting to consider would be Schr{\"o}dinger Hamiltonians, and corresponding effective Dirac Hamiltonians, within the same ``topological class''; those which can be continuously deformed into each other without closing the bulk gap.
  For a study of topological robustness of defect modes against such deformations within a family of dislocation Hamiltonians, see~\cite{Drouot:21}.
\end{itemize}

\subsection{Structure of the paper} In Section \ref{sec:preliminaries} we introduce notation, conventions and briefly review Floquet-Bloch theory for Schr{\"o}dinger operators on the line with a $\Z-$ periodic potential. The autonomous Hamiltonian $H^\varepsilon_{\rm dw}$ and the parametrically forced Schr{\"o}dinger equation are discussed in Section \ref{model}. Detailed numerical simulations of the parametrically forced Schr{\"o}dinger equation are presented in Section~\ref{numerics} % where, in particular, the rate of decay due to radiation damping is measured
 for concrete choices of domain-wall Schr{\"o}dinger Hamiltonians. In Section \ref{sec:dirac} we derive, via a multiple scales analysis, the parametrically forced effective Dirac system
 and state a theorem on its validity on time-scales of interest. In Section \ref{metastability} we present results on  the radiation damping mechanism in terms of an effective Dirac model, % and give a systematic derivation of the radiative decay, which 
 and corroborate our predictions by numerical simulations. The derivation of the radiative decay phenomena is presented in Section \ref{sec:decay_anal}, and necessary dispersive decay estimates for the unperturbed Dirac operator are derived in Section \ref{sec:ld-est}

\subsection{Acknowledgments}  The authors would like to thank M.\ Rechtsman and J.\ Shapiro for stimulating discussions.  The authors would like to thank the anonymous referees for their thoughtful comments and suggestions. The authors would also like to thank R.\ Kassem for a careful reading of the manuscript and many useful comments. A.S.\ acknowledges the support of the AMS-Simons Travel Grant.
M.I.W.\ was supported in part by US National Science Foundation grants  DMS-1620418 and DMS-1908657 as well as by the Simons Foundation Math + X Investigator Award \#376319.

%a Hamiltonian representing a slowly modulated edge represents the direction of propagation. When the wave-guides oscillate slowly and $h(z)$ is periodic, then by a simple change of variables, we get 

 %a topic of intense study in recent decades in periodic media is a topic both old are known to arise at interfaces (`edges') between dissimilar media. These edge-localized states stem from degeneracies in the spectra of bulk Hamiltonians and have long been studied as mechanisms for energy localization and transport \cite{Malkova2009ObservationOO, Raghu2008AnalogsOQ, Rechtsman2013PhotonicFT, Rechtsman2013TopologicalCA, Shockley1939OnTS, Wang2008ReflectionfreeOE, Yu2008OnewayEW}. Edge modes are often robust to perturbations of the interface which preserve the bulk spectral gap; they are said to be \emph{topologically protected}. The chiral edge states associated with the quantum Hall effect are an early example of such topologically protected states. More recently, topologically protected edge states have been observed in graphene-like systems [CITE]. Here, the breaking of time-reversal symmetry (as occurs due to a time-periodic magnetic driving) opens a gap at a `Dirac point' \cite{Fefferman2012HoneycombLP, Hasan2010ColloquiumT, Katsnelson2006GrapheneCI} and one finds chiral edge states with energies in the bulk energy gap \cite{Raghu2008AnalogsOQ, Rechtsman2013PhotonicFT,Wang2008ReflectionfreeOE,Yu2008OnewayEW}.

\section{Preliminaries}\label{sec:preliminaries}

\subsection{Floquet-Bloch theory}

 We begin with general remarks on the  Floquet-Bloch  spectral theory of periodic self-adjoint differential operators; for details, see, e.g.,\cite{eastham73spectral, FLTW17, Kuchment:16, RS4}. Let $V$ be a sufficiently smooth real-valued $\Z$ periodic potential, i.e., $V(x+1)=V(x)$ for all $x\in\R$, and introduce the periodic Schr{\"o}dinger operator 
\begin{equation}
    H \equiv -\partial^2_x+V(x)  \, ,
\label{eq:Hbulk}
\end{equation}
acting in the space $L^2(\R)$.  $H$ is self-adjoint and commutes with translations in the integer lattice, $\Z$. 

Any $f\in L^2(\R)$ can be represented, via the Floquet-Bloch transform, as a superposition of eigenstates 
of the integer translation operator $f(x)\to f( x+1)$, i.e.,  $L^2(\R) = \int^\oplus_\mathcal{B}L^2_{k} \,   dk$,
where
$$
    L^2_k\equiv \big\{f\in L^2_{\mathrm{loc}}:f(x+1;k)=e^{ik}f(x;k)\big\} \, , 
$$
where $k$ varies over $\mathcal{B}\equiv [0, 2\pi]$, the fundamental cell of the dual lattice $2\pi \mathbb{Z}$ (the Brillouin zone). Since $H$ commutes with integer translations, the spectral theory of $H$ acting on $L^2 (\R )$ can be reduced to the study of $H$ acting in each $L^2_k$ space:
$$H=\int^\oplus_\mathcal{B}H_{k} \,   dk \, ,~~~ {\rm where} \quad H_{k} = H\big|_{L^2_k} .$$ 
For each $k\in\mathcal{B}$, $H_k$ is a self-adjoint operator and has  compact resolvent. Hence, each $H_{k}$ has an infinite sequence of finite multiplicity real eigenvalues, tending to infinity
\[E_1(k) \leq E_2(k) \leq \cdots \leq E_b(k) \leq \cdots ~~ .\]
The corresponding eigenmodes ({\rm Bloch modes}) $\Phi_b(\cdot;k)\in L^2_k$ satisfy
\begin{equation}\label{eq:Hbulk_k}
    H \Phi_b(x;k)=E_b(k)\Phi_b(x;k) \, , \qquad \Phi_b(\cdot;k) \in L^2_k \, .
\end{equation}
The maps $E_b:\mathcal{B}\to \R$ are Lipschitz continuous, and their graphs $E_b(k)$ are called the {\em dispersion curves} of $H$. The collection of all pairs $(E_b(k), \Phi_b(x;k))$ for all $k\in \mathcal{B}$ and $b\geq 1$ is called the {\em band structure} of $H$.

\section{Hamiltonians; $H^\varepsilon_{\rm dw}$ and its periodic forcing, \eqref{eq:lsa_gen}}\label{model}

In this section we present a concise systematic discussion (more detailed than in the introduction) of the underlying bulk Hamiltonian, $H$, the domain-wall Hamiltonian $H^\varepsilon_{\rm dw}$ with asymptoics $H^\varepsilon_\pm$,  and the parametrically forced Hamiltonian $H_\varepsilon(t) = H^\varepsilon_{\rm dw} + 2i\varepsilon\beta A(\varepsilon t) \D_x$. Figure \ref{fig:U_delta_steps} will serve as a guide to the of  hierarchy of time-independent operators and their spectra.

\begin{figure}[h]
           \begin{subfigure}[t]{0.32\textwidth}
            \caption{}
            \centering
           			 \hskip -2ex
            \includegraphics[width=1\linewidth]{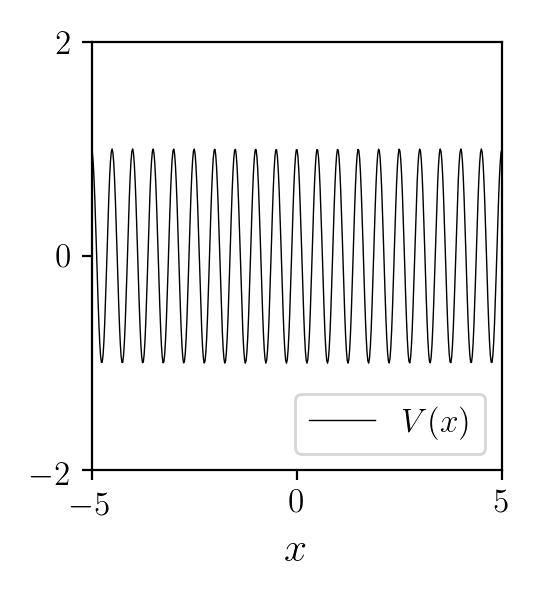}\\
            \hskip -.8ex
            \includegraphics[width=1\linewidth]{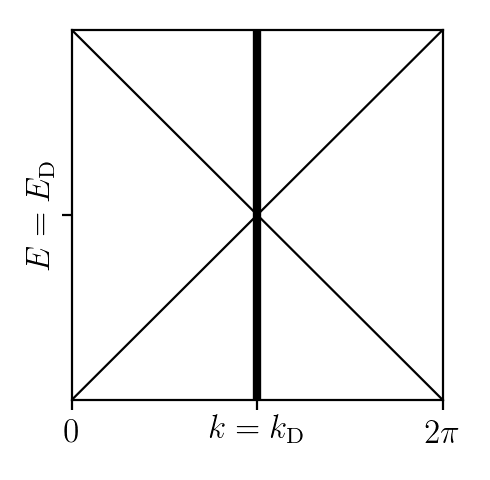}
         \end{subfigure}
         \begin{subfigure}[t]{0.32\textwidth}
            \caption{}
            \centering
                        \hskip -2ex
            \includegraphics[width=1\linewidth]{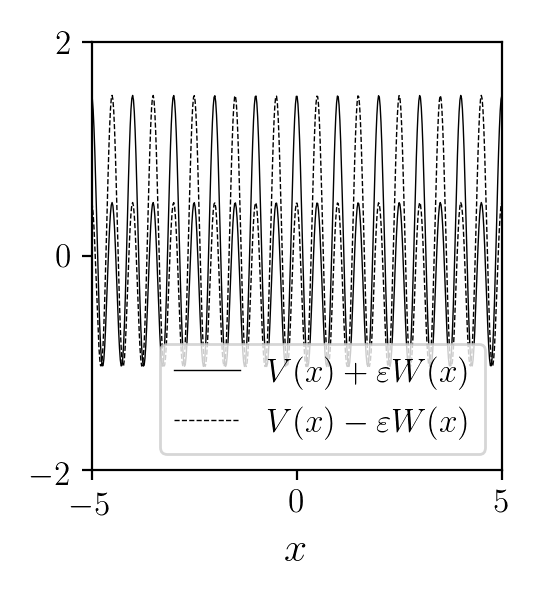}\\
            \hskip -.8ex
            \includegraphics[width=1\linewidth]{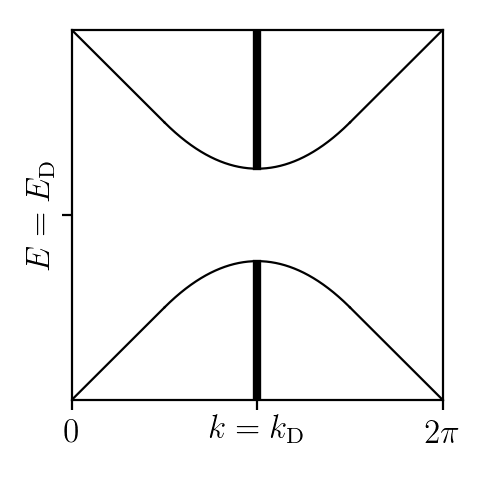}
         \end{subfigure}
         \begin{subfigure}[t]{0.32\textwidth}
            \caption{}
            \centering
                        \hskip -2ex
            \includegraphics[width=1\linewidth]{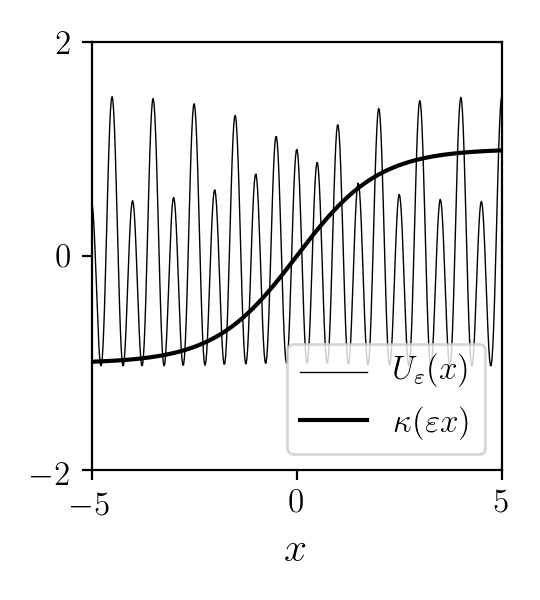}\\
            \hskip -.8ex
            \includegraphics[width=1\linewidth]{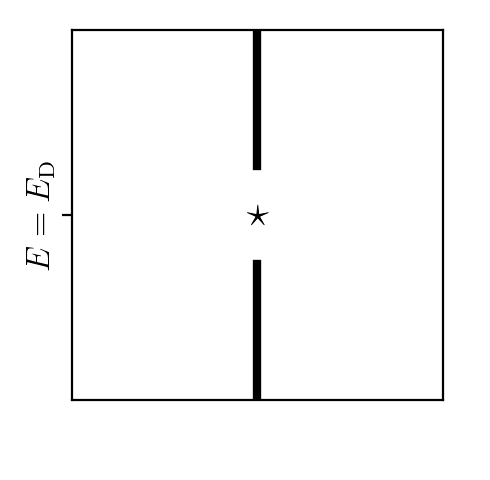}
         \end{subfigure}
        \caption{
        Step by step construction of multi-scale domain wall potential $U_{\varepsilon}$  (top row) and corresponding spectra (bottom row), following Section \ref{model}.
        {\bf (a)} $\Z-$ periodic potential, $V(x)$, and its linear band-crossing (Dirac point) due to additional $\frac12\Z$-translation symmetry.  {\bf (b)} Perturbed  potentials $V(x)\pm \varepsilon W(x)$ (bulk structure) with minimal period lattice $\Z$, and its spectral gap  about the Dirac point, due to broken symmetry. {\bf (c)} Asymptotically periodic potential with domain wall defect, $U_{\varepsilon}(x)$; see \eqref{eq:H0}. Interpolating domain wall function $\kappa(\varepsilon x)$ superimposed.  Mid-gap point eigenvalue is indicated. Note that since $U_{\varepsilon}$ is not translation invariant, $k$ is not indicated in the lower panel, just the energy spectrum.
        }
\label{fig:U_delta_steps}       
\end{figure}

\subsection{Schr{\"o}dinger Hamiltonians with Dirac points}\label{sec:bulk}

We begin with a Schr{\"o}dinger operator $H=-\D_x^2+V(x)$, where $V$ is real-valued  
and $\frac12\Z-$ periodic, i.e., $V(x+\frac12)=V(x)$. For simplicity, we assume that $V(x)$ is even;
 for modifications required to treat the general case, see \cite{drouot2020defect}.
We shall embed $H$ in a family of Hamiltonians $H^\varepsilon_\pm$, with  $H^0_\pm=H$ 
and $H^\varepsilon_\pm$ of minimal period equal to one.  
Hence, it is natural to consider $H$ as acting in $L^2(\R/\Z)$
as having  an additional symmetry; $H$ commutes with $\frac12\Z$ lattice translations.

Due to this additional symmetry, the band structure of $H$ acting in $L^2(\R/\Z)$ has {\em Dirac points} \cite{FLTW14,FLTW17}:  quasi-momentum / energy pairs  $(k_{\rm D},E_{\rm D})$, where $k_{\rm D}=\pi$. Specifically,  
there exists $b\ge1$ such that there are dispersion curves $k\mapsto  E_b(k), E_{b+1}(k)$ curves, which cross linearly:
 \begin{subequations}
\begin{align}
 E_b(k)-E_{\rm D} &= -v_{\rm D}|k-k_{\rm D}|(1+\mathcal{O}(|k-k_{\rm D}|)) \, , \\
 E_{b+1}(k)-E_{\rm D} &= +v_{\rm D}|k-k_{\rm D}|(1+\mathcal{O}(|k-k_{\rm D}|)) \, ,
 \end{align}
 \label{DPcrossing}
 \end{subequations}
for $|k-k_{\rm D}|\ll 1$, where $v_{\rm D}>0$ is a constant, the Dirac (or Fermi) velocity.  
The band structure around a Dirac point $(k,E)=(\pi,E_{\rm D})$ is shown in the first column of Figure~\ref{fig:U_delta_steps}.

It is convenient to introduce a basis $\{\Phi_1, \Phi_2 \}$
 of the two-dimensional $L^2_{k_{\rm D}}$ eigenspace of $E_{\rm D}$, such that $\Phi_2(x)=\Phi_1(-x)$.   In terms of this basis
 \begin{equation}\label{eq:lambda}
v_{\rm D} \equiv 2i\langle \Phi_1,\partial_x \Phi_1 \rangle_{L^2([0,1])} = -2i\langle \Phi_2,\partial_x \Phi_2 \rangle_{L^2([0,1])} \, .
\end{equation}
 By this appropriate choice of $\Phi_1$, one arranges for $v_{\rm D}>0$.  

\subsection{Bulk Hamiltonians $H^\varepsilon_\pm$}\label{sec:bulk-ops}
\begin{figure}[h]
            \includegraphics[width=.8\linewidth]{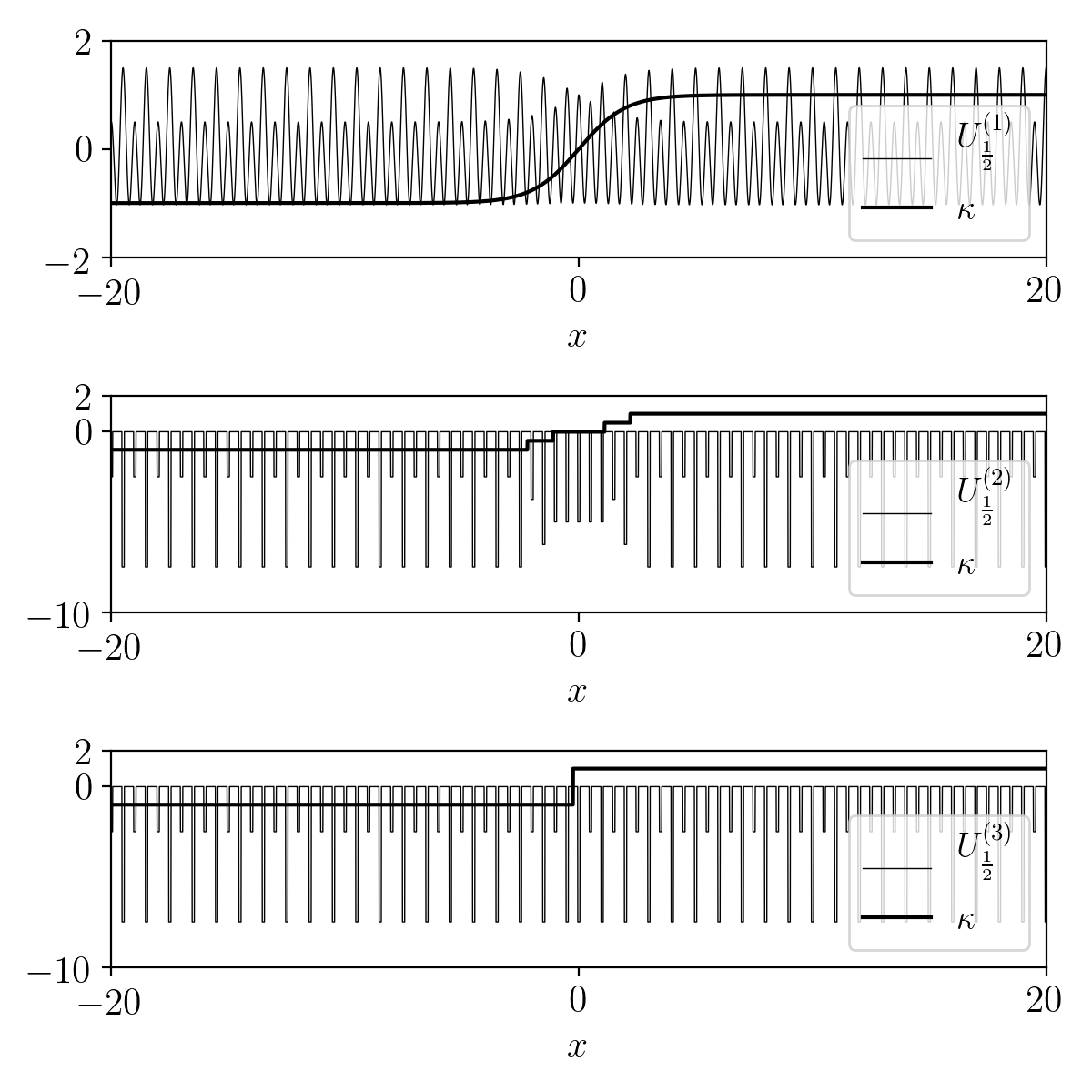}
        \caption{Three examples of the potential $U_{\varepsilon}$~\eqref{eq:lsA} \textbf{Top:} Cosine potential \eqref{eq:lsa_cos}. \textbf{Middle:} Periodic array of square-wells \eqref{eq:lsa_well}. \textbf{Bottom:} Periodic array of square-wells with a sharp transition~\eqref{eq:lsa_sgn}. }
\label{fig:U_delta}       
\end{figure}

Let $W(x)$ be real-valued and periodic with minimal period $1$. 
 We introduce, for $\varepsilon\ne0$ and real,  bulk operators which break the $1/2-$ periodicity of $H$
 \[ H^\varepsilon_\pm = H\pm\varepsilon W(x),\quad {\rm where}\quad W(x+1/2)=-W(x).\]
%  A simple example is $H^\varepsilon_\pm=-\D_x^2+\cos(4\pi x) \pm \varepsilon \cos(2\pi x)$. 
  Since $W(x)$ is a non-compact perturbation of $H$, it may change the essential spectrum. Indeed, 
the operators $H^\varepsilon_{\pm}$ have a  spectral gap about $E_{\rm D}$ of width $\mathcal{O}(\varepsilon)$;
  see \cite[Appendix F]{FLTW17}.
  %
  %\footnote{\textcolor{blue}{It looks like we may have misnamed \cite[Appendix F]{FLTW17}?}}
 %
  %
Introducing the shift operator $S_{\frac12}[f](x)\equiv f(x+\frac12)$, we have that $S^*_{\frac12}=S_{-\frac12}$ and 
    $S_{\frac12} H^\varepsilon_+S^*_{\frac12} = H^\varepsilon_-$. Thus, $H^\varepsilon_+$
     and $H^\varepsilon_-$ have the same spectrum and, in particular, have a common spectral gap about $E_{\rm D}$.  We refer to this common gap as  the {\it bulk spectral gap}.

 \subsection{$H^\varepsilon_{\rm dw}$,  asymptotically periodic Hamiltonian with domain wall defect }\label{sec:edge}
 Introduce a {\it domain wall function}, $\kappa(X)$, which is real-valued and such that
 \begin{align}
 \kappa(X) &\to \pm\kappa_\infty\quad \textrm{as $X\to\pm\infty$,\quad where $\kappa_\infty>0$}.
 \label{dom-wall}\end{align}
 We then define a  {\it domain wall Hamiltonian} which interpolates between the bulk Hamiltonians $H^\varepsilon_-$ and $H^\varepsilon_+$:\footnote{We take $V(x)$, $W(x)$ and $\kappa( X)$ to be sufficiently smooth and $\kappa^2(X)-\kappa_\infty^2$ tending to zero sufficiently rapidly. 
Precise smoothness hypotheses on these functions, which we believe can be somewhat relaxed, are spelled out in \cite{FLTW17}. }
 %
 %An adiabatic defect at the origin is introduced into the model through the following noncompact perturbation of %$H_{\rm bulk}$:
\begin{equation}\label{eq:H0}
    H_{\rm dw}^{\varepsilon}\equiv -\D_x^2 + U_{\varepsilon}(x) \, , \qquad U_\varepsilon(x) \equiv V(x) + \varepsilon\kappa(\varepsilon x)W(x) \, ,
    \end{equation}
Two examples of asymptotically periodic potentials with domain wall defects are displayed in the top and middle panels of  Figure \ref{fig:U_delta}.  

The essential spectrum of $H_{\rm dw}^{\varepsilon}$ is equal to that of the operators $H^\varepsilon_\pm$,
 but point spectrum may arise in the spectral gaps of $H_{\rm dw}^{\varepsilon}$ and, in particular, 
 in the {\rm bulk spectral gap} about the Dirac energy $E_{\rm D}$. We next give a brief discussion of the bifurcation of discrete eigenvalues of $H_{\rm dw}^{\varepsilon}$, from the Dirac point at $\varepsilon=0$ into the spectral gap about $E_{\rm D}$ for $\varepsilon$ non-zero and small.

Introduce the effective Dirac operator, $\sD$, defined in terms of the domain wall function $\kappa(X)$,  Dirac velocity $v_{\rm D}$, and 
$\vartheta_{\sharp} \equiv \langle \Phi_1, W\Phi_2 \rangle$:
\begin{equation}
 \sD \equiv iv_{\rm D}\sigma_3\partial_X+\vartheta_\sharp\kappa(X)\sigma_1 ,
 \label{eq:D0_def}
 \end{equation}
 where the Pauli matrices are 
\begin{equation}
\sigma_1 = \left(\begin{array}{ll}
 0 &1 \\1 &0
\end{array} \right) \, , \quad  \sigma_2 = \left(\begin{array}{ll}
 0 &-i \\i &0
\end{array} \right) \, , \quad \sigma_3 = \left(\begin{array}{ll}
 1&0 \\ 0 &-1
\end{array} \right) \, .\label{pauli}
\end{equation}
 Note that the parameters $v_{\rm D}$ and $\vartheta_{\sharp}$ are determined by the ``Dirac eigenspace'' for the energy $E_{\rm D}$.
 Since $\kappa(X)\to\pm\kappa_\infty$ as $X\to\pm\infty$, the continuum spectrum of $\sD$ acting in $L^2(\R; \C^2)$ is equal to $(-\infty,-|\vartheta_\sharp|\kappa_\infty]\cup [|\vartheta_\sharp|\kappa_\infty,\infty)$.
 
 The spectrum of $\sD$ also contains a discrete eigenvalue at zero energy, $\alpha_\star(X)$, satisfying: 
 \[ \sD\alpha_\star=0,\quad \|\alpha_\star\|_{L^2(\R;\C^2)}=1\]
 It is easily verified that $\alpha_\star$ is given, up to a constant phase,  by
\begin{equation}
    \alpha_\star(X)=C\begin{pmatrix}1\\i\end{pmatrix}e^{-(\vartheta_\sharp/v_{\rm D})\int^X_0\kappa(s) \; ds},
\label{zero_mode_rep}
\end{equation}
where $C>0$ is a normalization constant. Since $\kappa_\infty>0$, the function \eqref{zero_mode_rep} is exponentially decaying at infinity.
Thus it is the change in sign of $\kappa(X)$ that is responsible for the existence of a robust zero energy eigenstate; zero is a rigid eigenvalue,
 with respect to perturbations of $\kappa(X)$ which are spatially localized.

To simplify the discussion, we make the 
\begin{gather} 
 \text{{\bf Assumption:} $0$ is the only eigenvalue of $\sD$.}   \label{eq:0mode}\\
 \text{Its corresponding normalized eigenstate is denoted $\alpha_\star$:} \nonumber\\
\sD \alpha_\star = 0 , \quad \|\alpha_\star\|_{L^2(\R;\C^2)}=1.
\nonumber 
\end{gather}
This assumption is satisfied, for example, by $\sD$ for  $\kappa(X)=\tanh(X)$; see \cite[Appendix A]{lu2020dirac}. A schematic of the spectrum of $\sD$ is shown in Figure \ref{fig:specD0}.  \begin{figure}[h!]
        \includegraphics[width=0.7\linewidth]{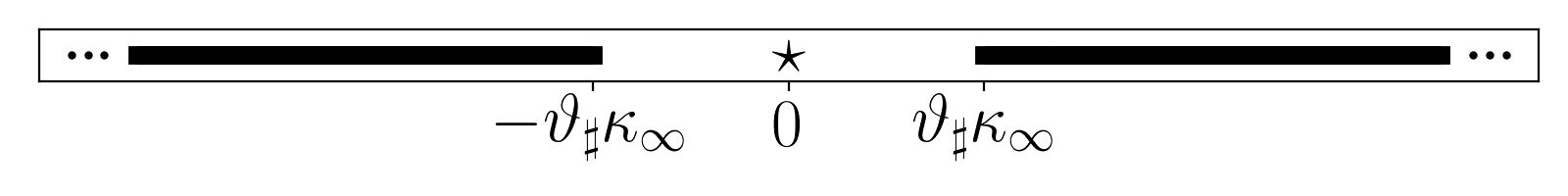}
        \caption{Spectrum of the (unforced) Dirac operator $\sD$, see~\eqref{eq:D0_def}. }
        \label{fig:specD0}
\end{figure} 

The defect mode $\alpha_{\star}$ of the Dirac equation \eqref{eq:D0_def} seeds a corresponding defect mode of the Schr{\"o}dinger equation \eqref{eq:H0}:
\begin{theorem}[Theorem 5.1 in \cite{FLTW17}; see also \cite{FLTW14,drouot2020defect}]
\label{bifurcation}
There exists $\varepsilon_0>0$ such that for all  $0<|\varepsilon|<\varepsilon_0$, the following is a characterization of all point spectrum in the spectral gap about $E_{\rm D}$, bounded away from the edges of the continuous spectrum:
\begin{enumerate}
\item  There exists a smooth curve defined for $\varepsilon\in[0,\varepsilon_0)$:
 \begin{align*}
 \varepsilon&\mapsto E^\varepsilon=E_{\rm D}+\mathcal{O}(\varepsilon^2)\in\left(E_{\rm D}-|\vartheta_\sharp|\kappa_\infty \varepsilon,E_{\rm D}+|\vartheta_\sharp|\kappa_\infty\varepsilon\right),\\
 \varepsilon&\mapsto \psi^\varepsilon_\star\in H^1(\R)
 \end{align*}
 such that $(E^\varepsilon,\psi_\star^\varepsilon)$ is a normalized eigenpair of $H_{\rm dw}^{\varepsilon}$
\begin{equation}\label{eq:psistar_def}
H_{\rm dw}^{\varepsilon}\psi^\varepsilon_\star=E^\varepsilon \psi^\varepsilon_\star \, .
\end{equation}
\item The eigenstate $\psi^\varepsilon_\star$ has a multi-scale structure; to leading order in $\varepsilon$, it is a {\em Dirac wave packet}, a slow modulation of the two-dimensional eigenspace  ${\rm span}\{\Phi_1,\Phi_2\}$, 
 associated with the Dirac point at $(k,E)=(\pi,E_{\rm D})$ (see Section \ref{sec:bulk}):
\begin{align*}
\psi^\varepsilon_\star(x) &= \varepsilon^{\frac12} \Big( \alpha_{\star,1}(\varepsilon x)\Phi_1(x;k_{\rm D}) + \alpha_{\star,2}(\varepsilon x)\Phi_2(x;k_{\rm D}) \Big)+\mathcal{O}_{H^2(\R )}(\varepsilon) \\
&\equiv \varepsilon^{\frac12}\alpha_\star(\varepsilon x)^{\top}\Phi(x) +\mathcal{O}_{H^2(\R )}(\varepsilon ) \, .\numberthis  \label{eq:psialpha_star}
\end{align*}
The envelope, $\alpha_\star(X) = (\alpha_{1\star}(X),\alpha_{2\star}(X))^\top$,  is given by \eqref{zero_mode_rep}, the zero energy eigenstate of the effective Dirac operator $\sD$.
\end{enumerate}
\end{theorem}
In Appendix \ref{ap:eff_pf} we present a derivation of  the effective Dirac equation in the more general (temporally-forced) setting of 
\eqref{eq:lsa_gen}; see also \cite{FLTW17, SW21}. Two examples of defect states $\psi_{\star}^{\varepsilon}(x)$ and the corresponding envelopes $\alpha_{\star}(X)$ are displayed  in Figure~\ref{fig:modes_star}.

\begin{remark}\label{multi-mode} We may relax the assumption \eqref{eq:0mode}. In general $\sD$ has $2N+1$ of eigenvalues in its bulk
gap $(-|\vartheta_\sharp|\kappa_\infty,+|\vartheta_\sharp|\kappa_\infty)$ for some $N\ge0$: a zero energy eigenvalue, \eqref{zero_mode_rep},  and $2N$ eigenvalues located symmetrically about zero. In \cite{drouot2020defect}, a more general version of Theorem \ref{bifurcation} 
 is proved in which the point spectrum of  $H_{\rm dw}^{\varepsilon}$  is given by $2N+1$ eigenvalues in its $\mathcal{O}(\varepsilon)$ width spectral gap about $E_{\rm D}$, with analogous
  multiple scale expansions of the corresponding edge eigenstates. We believe that the results of this article (numerical and analytical) on radiation damping can be extended to this multimode setting; see \cite{KW03} for relevant analysis.
  \end{remark}

\begin{figure}[h]
            \includegraphics[width=.8\linewidth]{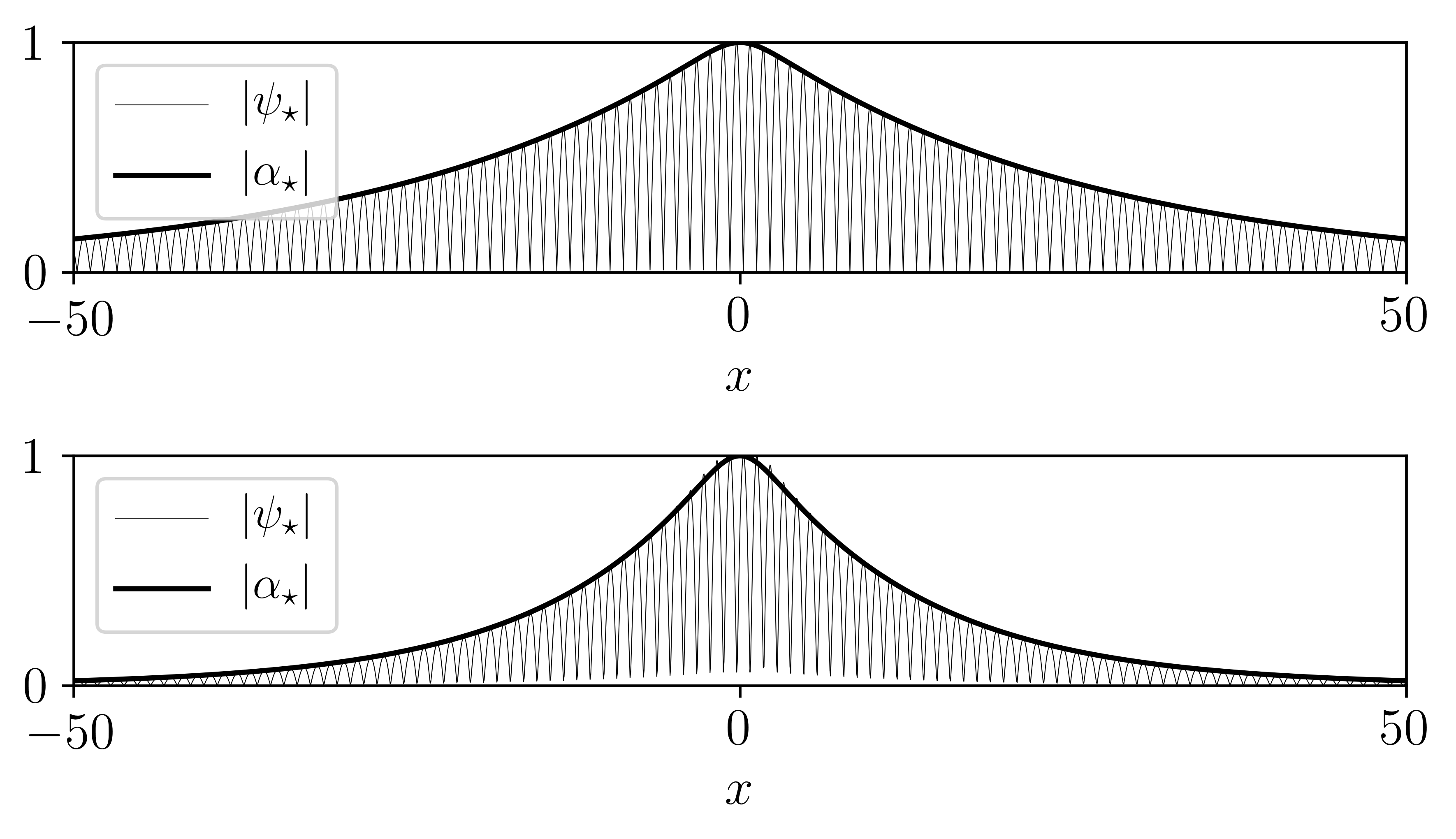}
        \caption{
        The amplitude of $|\psi^\varepsilon_{\star}|$, the defect mode of the Schr{\"o}dinger Hamiltonian $H^\varepsilon_{\rm dw}$(\eqref{eq:psistar_def}, grey), showing its multi-scale (carrier + envelope) structure for  \textbf{(a)}~$U_{\varepsilon}^{(1)}$ \eqref{eq:lsa_cos}, and \textbf{(b)} $U_{\varepsilon}^{(2)}$~\eqref{eq:lsa_well}, see also Figure \ref{fig:U_delta} . The envelope  is well-approximated by $\alpha_*$, (\eqref{eq:0mode}, black), of the corresponding Dirac operator $\sD$, see \eqref{eq:D0_def}.
        }
\label{fig:modes_star}       
\end{figure}

\subsection{Floquet model; parametric time-periodic forcing of $H^\varepsilon_{\rm dw}$}

Our goal in this paper is to study the effect of time-periodic (parametric) forcing on the defect state $\psi_\star$ of $H^\varepsilon_{\rm dw}$; see \eqref{eq:H0} and \eqref{eq:psialpha_star}.
Denote our time-periodic Hamiltonian by
 \begin{equation}
  H_\varepsilon(t) \equiv H_{\rm dw}^{\varepsilon} + 2i\varepsilon \beta A(\varepsilon t)\partial_x\ =\ 
   -\D_x^2 + U_\varepsilon(x)+ 2i\varepsilon  A(\varepsilon t)\partial_x
   \label{Hepst}
   \end{equation}
 where $ A(T)=2\beta \cos (\omega T)$.
 We shall study the initial value problem:
\begin{subequations}
\label{eq:lsA}
\begin{align}
     i\partial_t \psi(t,x) &= H_\varepsilon(t)\psi\, , \\
    \psi(0,x) &= \psi^\varepsilon_\star(x) \, . \label{psi*dat}
\end{align}
\end{subequations}
In Section \ref{numerics}, we present long-time numerical simulations of \eqref{eq:lsA} showing the resonant radiation damping effect, and in In Sections \ref{sec:dirac} and \ref{metastability}, we provide an analytical understanding of this  phenomenon on the physically interesting time-scale.

\section{Long-time simulations of the forced Schr{\"o}dinger equation, \eqref{eq:lsA}; slow radiative decay of the defect state} \label{numerics}

In this section we numerically investigate  the long-time behavior of solutions of the initial value problem \eqref{eq:lsA}
 for  two Hamiltonians:
\begin{equation}\label{eq:Hleps} H_\varepsilon^{(\ell)}(t)= -\D_x^2 + U_{\varepsilon}^{(\ell)} (x) + 2i\ \varepsilon \beta \cos (\varepsilon \omega t)  \D_x \, , \qquad \ell=1,2 \, ,
\end{equation}
\begin{enumerate}
\item $U_\varepsilon^{(1)}$ smoothly interpolates between asymptotic phase shifted cosine-potentials
\item $U_\varepsilon^{(2)}$ interpolates between asymptotic phase-shifted arrays of double square wells
\end{enumerate}
for different choices of the parameters $\beta, \omega >0$.
See the top two panels of Figure~\ref{fig:U_delta} for illustrations of  $U_\varepsilon^{(1)}$ and $U_\varepsilon^{(2)}$. 
Throughout our simulations, we set $\varepsilon = 1/2$. 
The precise definitions of $U_\varepsilon^{(j)}$  are given in Appendix \ref{ap:potentials}.
 Details concerning the numerical method are given in Appendix \ref{ap:numMethods}.

\begin{figure}
\centering
\begin{subfigure}{.49\textwidth}
  \caption{$U^{(1)}_{\frac12}$, Unforced ($\beta =0$)}

  \centering
  \includegraphics[width=1\linewidth]{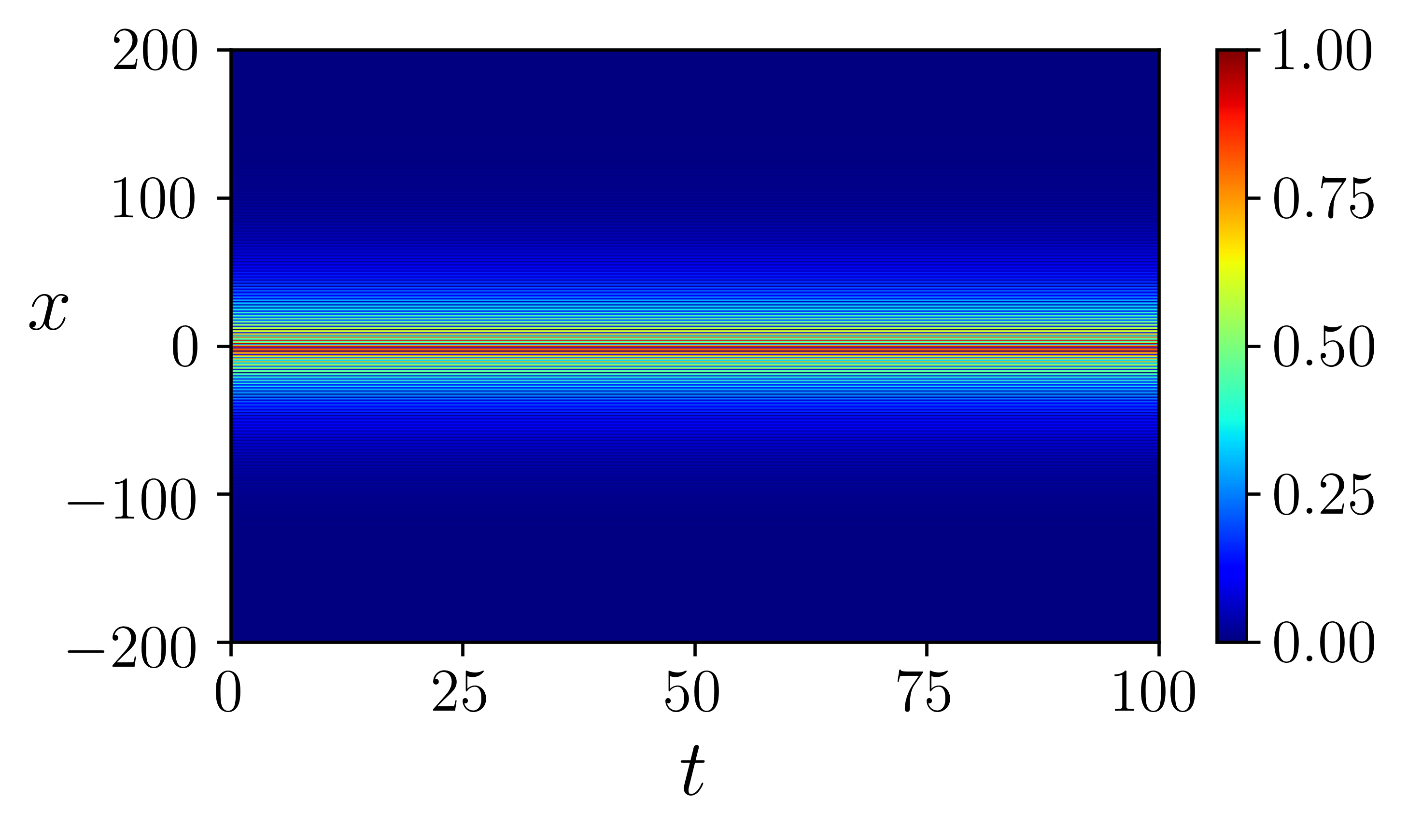}
  \label{fig:Sch_cos_stat}
\end{subfigure} 
\hskip -0ex
\begin{subfigure}{.49\textwidth}
  \caption{$U^{(1)}_{\frac12}$, Forced ($\beta =0.01$)}
  \centering
  \includegraphics[width=1\linewidth]{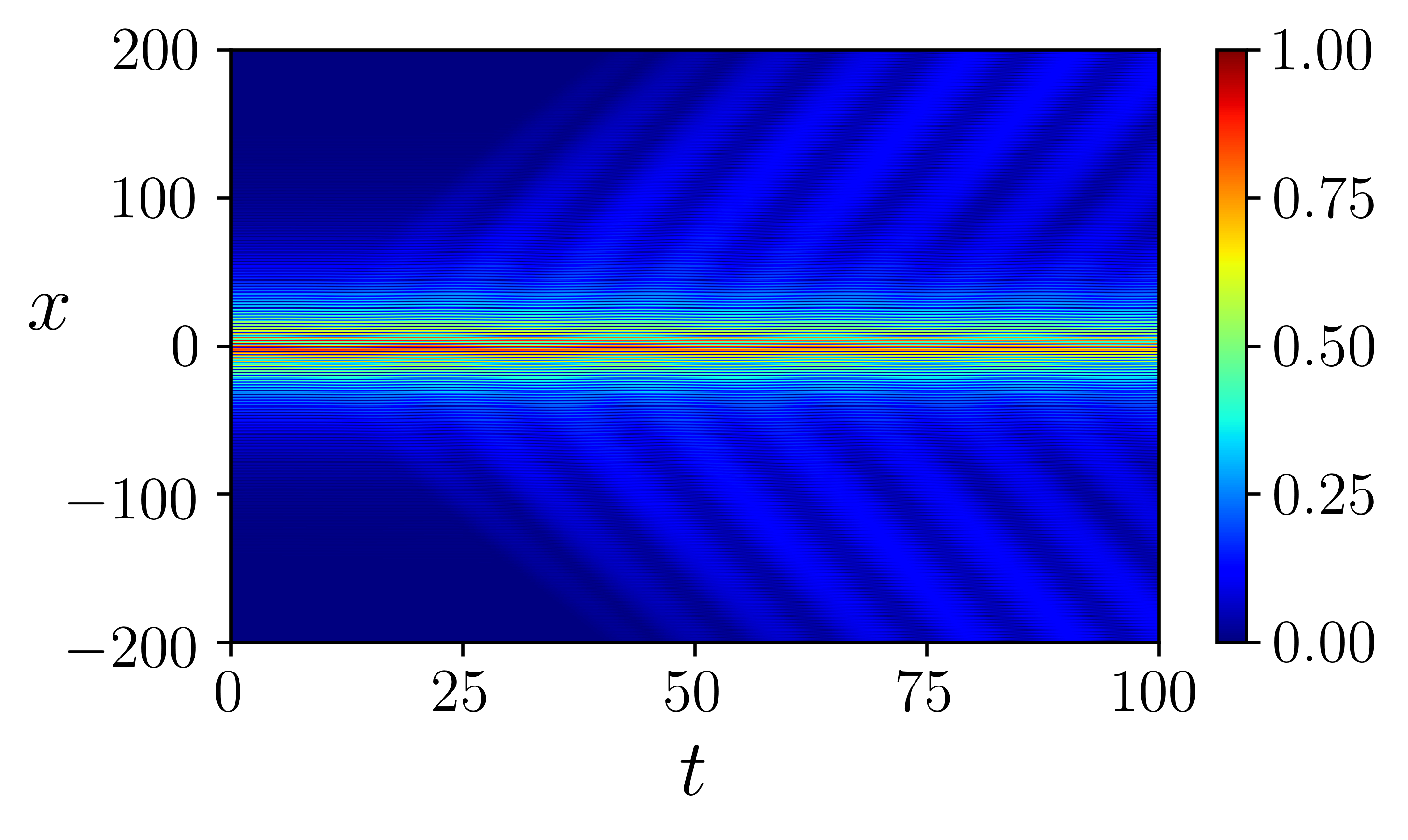}
  \label{fig:Sch_cos_for}
\end{subfigure}
\begin{subfigure}{.49\textwidth}
  \caption{$U^{(2)}_{\frac12}$, Unforced ($\beta =0$)}
  \centering
  \includegraphics[width=1\linewidth]{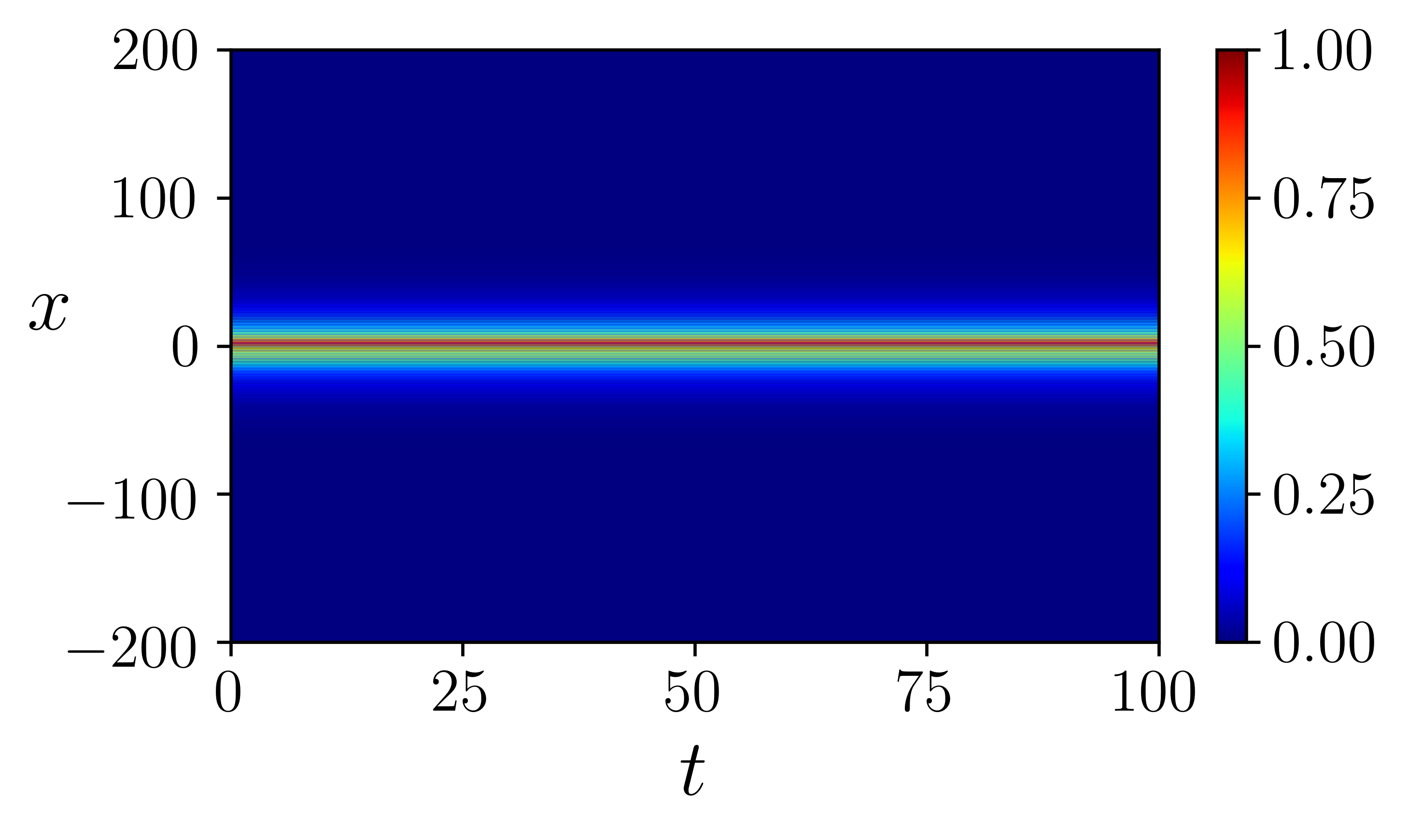}
  \label{fig:Sch_well_stat}
\end{subfigure} 
\hskip -0ex
\begin{subfigure}{.49\textwidth}
  \caption{$U^{(2)}_{\frac12}$, Forced ($\beta =0.01$)}
  \centering
  \includegraphics[width=1\linewidth]{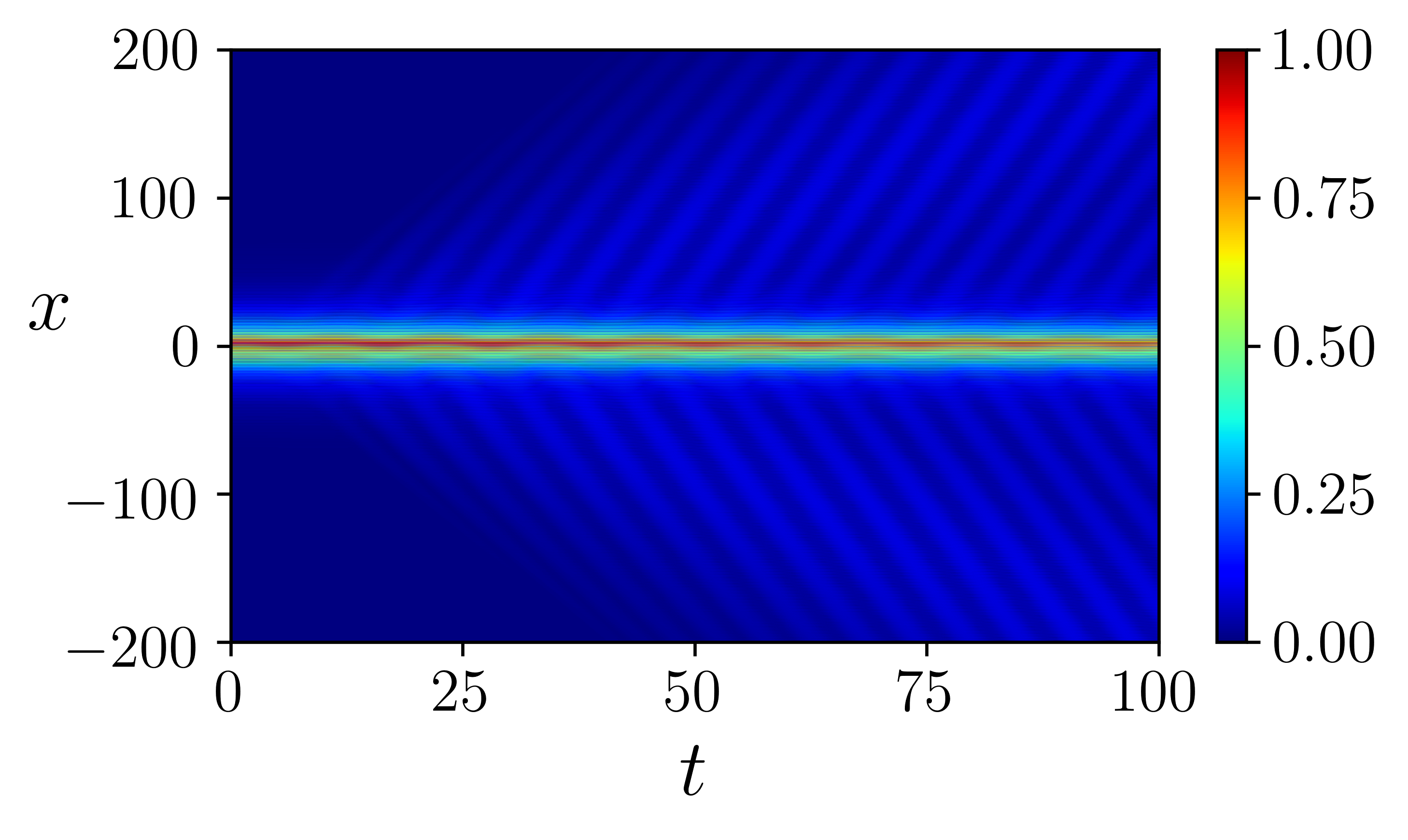}
  \label{fig:Sch_well_for}
\end{subfigure}
\caption{Solutions of \eqref{eq:lsA} with $\psi(0,x)=\psi_{\star}(x)$. {\bf(a)} $U^{(1)}_{\frac12}$ with $\beta = 0$, the unforced problem, and {\bf(b}) with $\beta = 0.01$ and $\omega =~0.6$. {\bf (c) and (d)} Same as (a) and (b), respectively, for $U^{(2)}_{\frac12}$ with $\omega =~1.1$.}
\label{fig:Sch_heat}
\end{figure}

Consider first the case $\beta=0$. Defect-mode initial data $\psi(x,0)=\psi_\star(x)$ gives rise to the solution $\psi(t,x)=e^{-iE^{\varepsilon}t}\psi^\varepsilon_\star(x)$; see \eqref{eq:psistar_def}.  Figures \ref{fig:Sch_cos_stat} and~\ref{fig:Sch_well_stat} indeed show, for both potential choices, that  $|\psi(t,x)|=|\psi_\star(x)|$ for $0\le~t\le 100$.

We next consider the dynamics {\em under temporal forcing}.  For $U^{(1)}_{\varepsilon}$ (and $U^{(2)}_{\varepsilon}$) we fix $\omega = 0.6$ (and $\omega =1.1$, respectively) and $\beta>0$ and observe decay of the solution with advancing time; see Figs.\ \ref{fig:Sch_cos_for} and \ref{fig:Sch_well_for}.
\footnote{The choice of $\omega =1.1$ in Fig.\ \ref{fig:Sch_well_for} is not an arbitrary one. In Section \ref{metastability} we show that radiation damping occurs 
if $|\omega|$ exceeds half the bulk spectral gap width of $\sD$, 
the effective Dirac Hamiltonian, see Fig.\ \ref{fig:decay_cartoon_Intro}. For $\sD^{(1)}$, which is derived from $U^{(1)}_{\varepsilon}$, the size of the bulk spectral gap is~$1$, and therefore $|\omega|>0.5$ satisfies the resonance condition (similarly with $|\omega|>1$ for $\sD^{(2)}$).}

\begin{figure}[t]
             \begin{subfigure}[t]{1\textwidth}
            \centering
            \hskip -2.5ex
            \includegraphics[width=.9\linewidth]{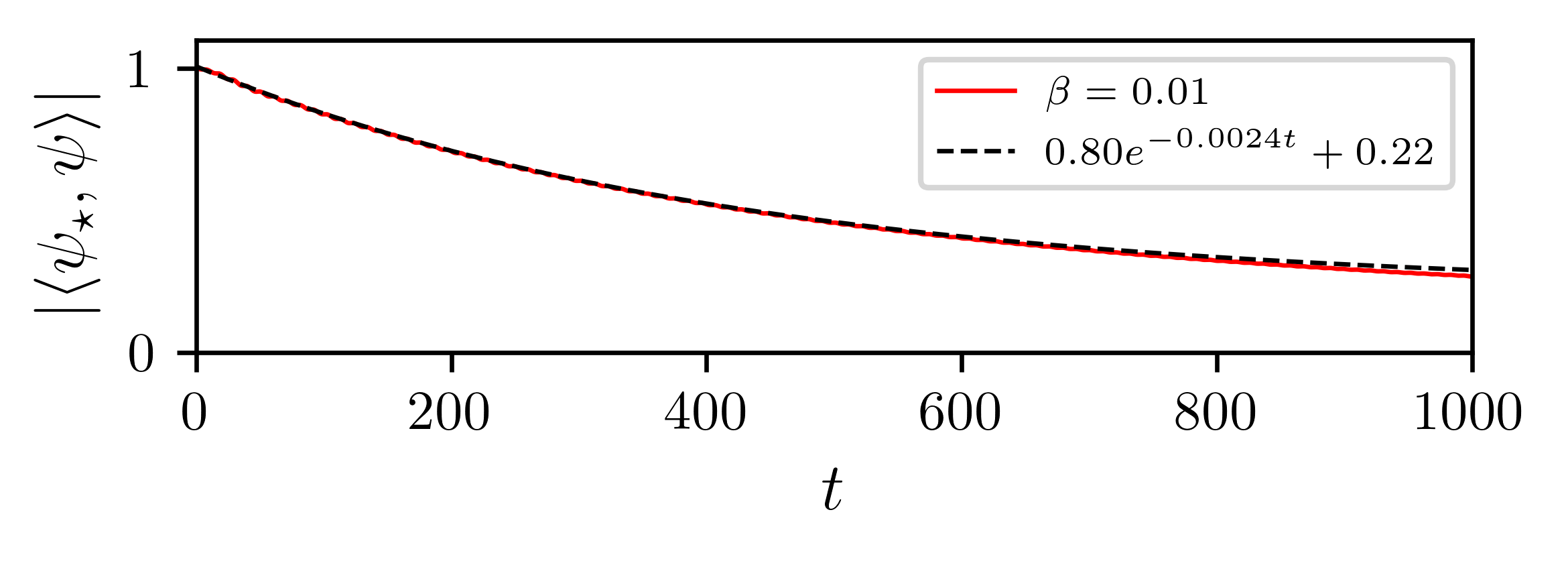} \vskip -2ex
             \hskip -2.5ex
            \includegraphics[width=.9\linewidth]{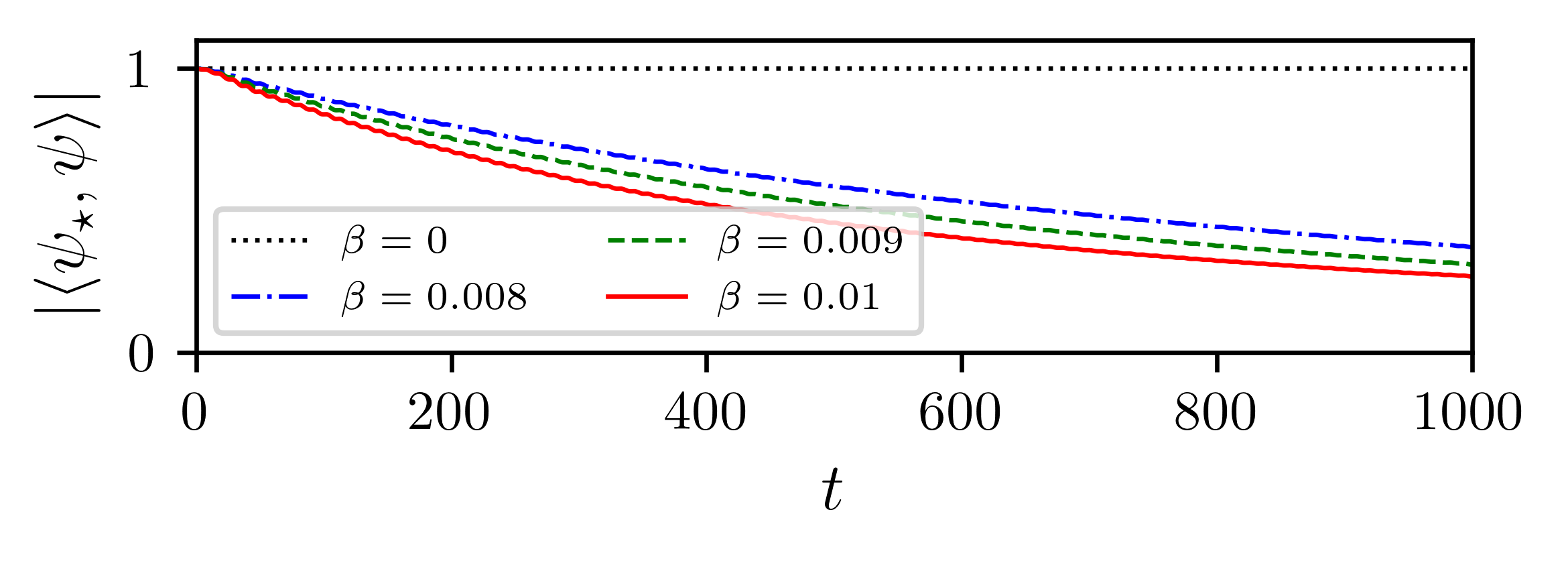}
        \end{subfigure}
        \caption{Projections of solutions of \eqref{eq:lsA} with $U^{(1)}_{\varepsilon}$, $\omega = 0.6$, and $\psi(0,x)=~\psi_{\star}(x)$ onto $\psi_\star(x)$. \textbf{Top:} An exponential decay fitted to the projection for $\beta=~0.01$. The projection curve and the fitted line are nearly indistinguishable. \textbf{Bottom:} The projection $|\langle\psi_\star,\psi(t,\cdot)\rangle_{L^2(\R)}|$ for $\beta=~0,0.008,0.009,0.01$.}
\label{fig:Sch_cos_proj}
\end{figure}

In Fig.\ \ref{fig:Sch_cos_proj}, we explore the decay rate by tracking the time-dependence of projection of the solution, $\psi(\cdot,t)$, onto the defect state
, i.e., $\langle  \psi_\star^\varepsilon, \psi(t,\cdot)\rangle_{L^2(\R)}$. Our numerical results  
($U^{(1)}_\varepsilon$ with $\varepsilon=1/2$)
are consistent with exponential decay of this projection over the very long time scale, $t\le 1000$, 
$\langle \psi_{\star}, \psi(t,\cdot)\rangle_{L^2(\R)}\approx \exp(-\Gamma  t)$. 
Further, we studied the dependence of the exponential rate, $\Gamma$, on $\beta$ over a range of small~$\beta$.
    The decay is consistent with a decay $\sim\exp(-\Gamma(\beta) t)$, 
     where $\Gamma(\beta) \sim \beta^{2.07}$ and $\sim \beta^{2.10}$ for $U^{(1)}_{\varepsilon}$ and $U^{(2)}_{\varepsilon}$, respectively.

\begin{figure}[t]
            \centering
             \hskip -5ex
              \includegraphics[width=.8\linewidth]{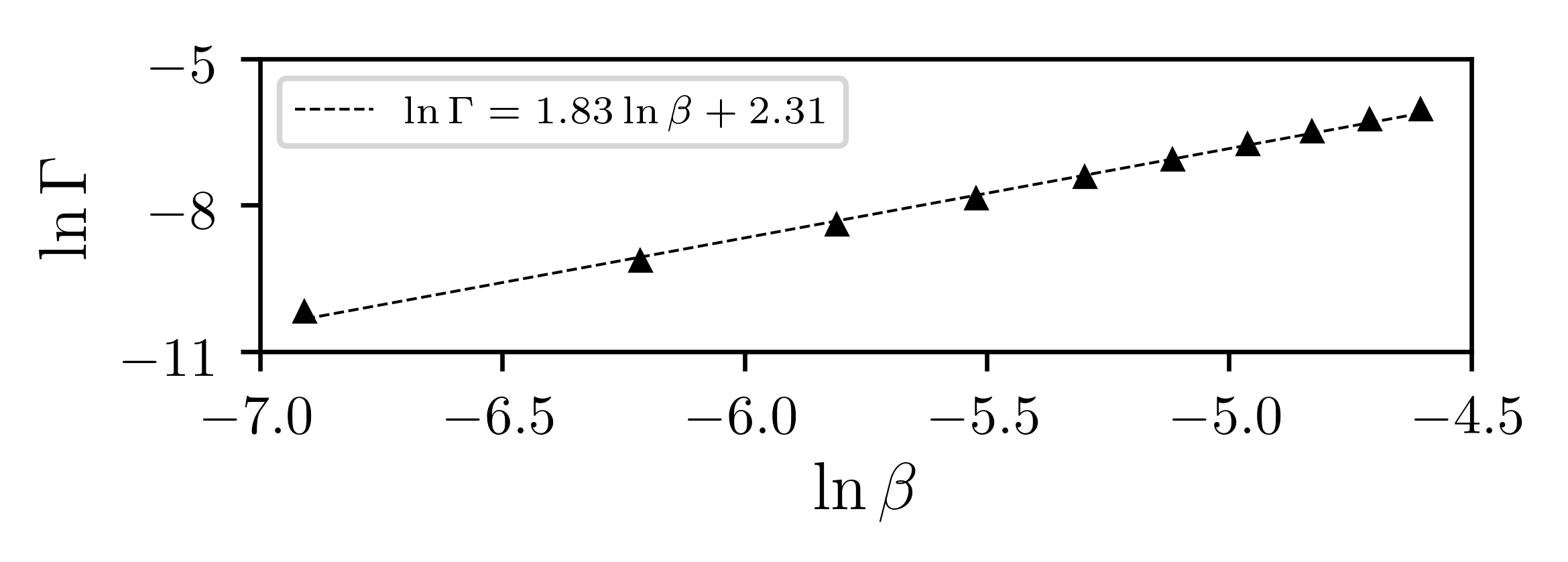}\vskip -2ex
              \hskip -5ex
              \includegraphics[width=.8\linewidth]{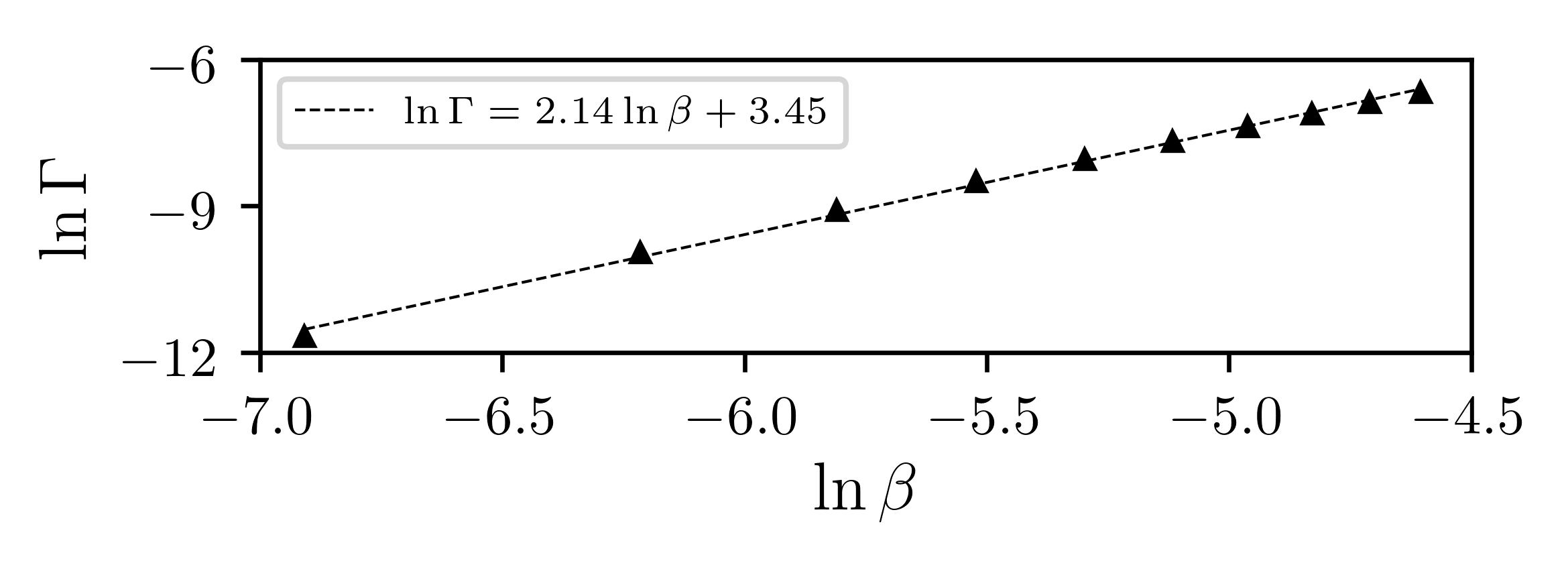}\vskip -2ex 
               \hskip -5ex
              \includegraphics[width=.8\linewidth]{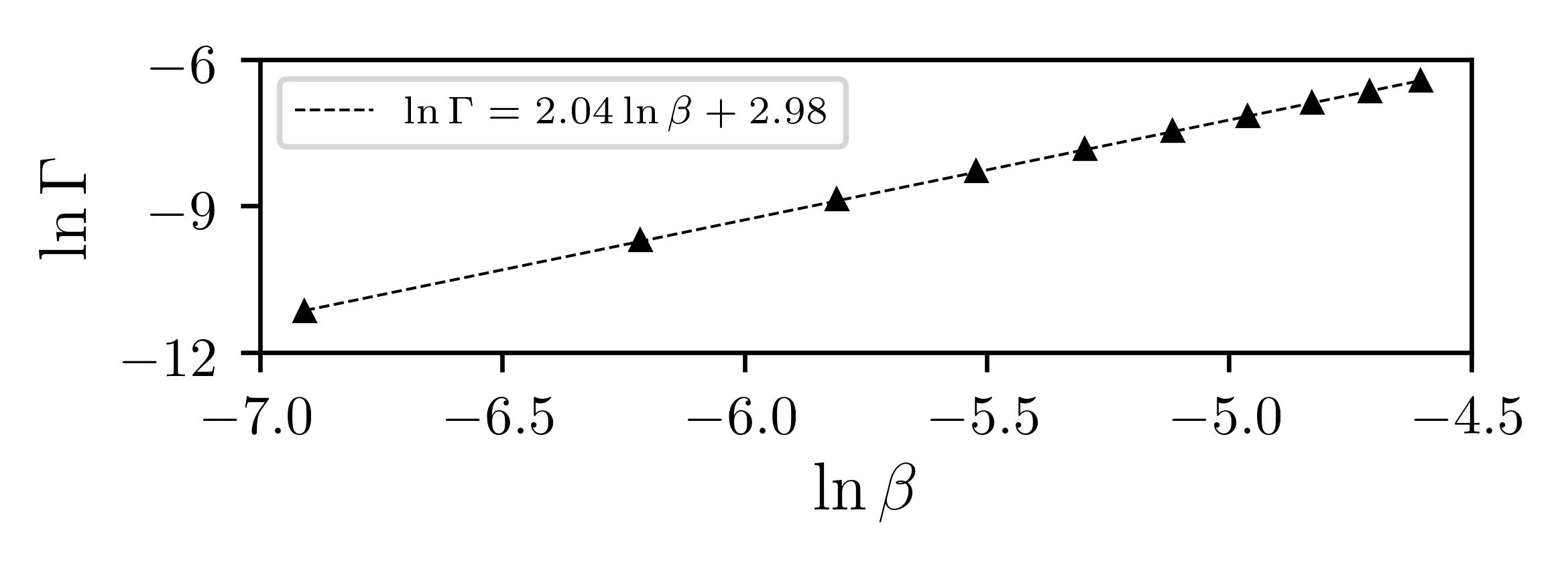}
        \caption{\textbf{Top:} The decay rate $\Gamma$ against the forcing amplitude $\beta$ on a log-log scale (triangles) and an exponential fit of $\Gamma \sim \beta^{1.83}$. \textbf{Middle and Bottom:} Same as the first row, for $U^{(2)}_{\varepsilon}$ and $U^{(3)}_{\varepsilon}$ respectively, with $\omega = 1.1$, and exponential fits of $\Gamma \sim \beta^{2.14}$ and $\Gamma \sim \beta^{2.04}$.}
\label{fig:Sch_gamma}
\end{figure}

%\footnote{\textcolor{red}{Michael, I think this paragraph here is redundant.} Theorem  \ref{thm:valid} of the following section, Section \ref{sec:dirac}, states that on  long time scales, the envelope dynamics of wave packets under the Schr{\"o}dinger equation \eqref{eq:lsA} is well approximated via an effective time-periodic Dirac Hamiltonian. In Sections \ref{metastability} and \ref{sec:decay_anal} we present results on the $\exp(-\beta^2\Gamma_0 T)$- radiative decay for this effective Dirac evolution, which corresponds to the  $\exp(-\beta^2\Gamma_0 \varepsilon t)$ radiative-decay of the envelope, and hence of the entire solution  of the Schr{\"o}dinger equation \eqref{eq:lsA}. Section \ref{metastability} is a summary of results on radiation damping for the effective forced Dirac equation and Section \ref{sec:decay_anal} presents the multi-scale expansion and analysis used to reach these conclusions.}
 
\subsection{The case of a sharp transition}\label{sec:sgn_sch}

In the potential $U^{(2)}_{\varepsilon}$, the transition between the asymptotics at $-\infty$ and $+\infty$ is done {\em slowly} via a piecewise constant domain wall function $\kappa(X)$. A more common
configuration in optical experiments is a sharp transition between phase-shifted waveguide arrays  \cite{bellec2017non}; see lower panel of Figure \ref{fig:U_delta}.

Thus we consider \eqref{eq:Hleps} with $U^{(3)}_{\varepsilon}$, where the asymptotic bulk potentials, $V(x)\pm \varepsilon W(x)$, are identical to those of $U^{(2)}_{\varepsilon}$ (see Appendix \ref{ap:potentials}), but now $\kappa(X) = {\rm sgn}(X)$. 
Due to the very different scaling properties of $U^{(3)}_{\varepsilon}$, the analysis leading to Theorem~\ref{bifurcation} does not apply. However,
we expect similar qualitative behavior.

% \footnote{\textcolor{red}{Michael, let's discuss - I think we should omit this part from here, it is somewhat speculative:}
%Indeed, one can continuous deform the potential of type $U^{(2)}_{\varepsilon}$ 
%into $U^{(3)}_{\varepsilon}$, while maintaining the identical spectral gap throughout the deformation. We expect that for $\varepsilon$ in some range $(0,\varepsilon_1)$, the defect state will persist throughout this deformation.Hence,  the above discussed radiation damping mechanism should apply.\\
%\textcolor{blue}{I agree. I inserted a remark instead in the discussion section. let me know what you think.}} 
%  %Intuitively, this is because the transition between the two asymptotic bulk potentials does not slow down as $\varepsilon \to 0^+$.
We repeated our simulations for $U^{(3)}_{\varepsilon}$ and have corroborated these heuristics. We observe, as  in the previous examples,  that $H^{\varepsilon}_{\rm dw}$ has a defect mode which, under periodic forcing, radiation damps on time scales of order $\beta^{-2}$; see Figure\ \ref{fig:Sch_gamma}.

\section{The effective Dirac equation} \label{sec:dirac}

Our goal in this section is to approximate the multi-scale evolution of the Schr{\"o}dinger equation \eqref{eq:lsA} by 
  simpler effective envelope  dynamics.
 The separation of fast scales ($x$ and $t$) of the underlying Bloch modes and slow scales ($X=\varepsilon x$ and $T=\varepsilon t$)
  of the domain wall
and parametric forcing, allow for the derivation of an effective homogenized equation for wave-packet envelope.

In Appendix \ref{ap:eff_pf} we show that the envelope of the solution of the time-periodically forced initial value problem \eqref{eq:lsA},
 with initial data
 \[ \psi(0,x)=\psi^\varepsilon_{\star}(x) \approx \varepsilon^{1/2}\alpha _{\star}(\varepsilon x)^{\top}\Phi(x)\]
 evolves in the form 
$$\psi(t,x)\approx \varepsilon^{\frac12}\alpha(\varepsilon t,\varepsilon x)^{\top}\Phi(x)e^{-iE_{\rm D} t}$$
% = \varepsilon^{\frac12}\sum_{j=1}^2  \alpha_j(\varepsilon t,\varepsilon x)\Phi_j(x)e^{-iE_{\rm D} t}\, .
 where $\alpha(T,X)$ is governed by the initial value problem for an effective periodically-forced Dirac equation:
\begin{subequations}
\label{eq:diracA}
\begin{equation}
     i\partial_T \alpha(T,X) = \sDT\alpha \, , \qquad  \alpha (0,X)= \alpha_0 (X)\in H^4(\R ;\C^2 ) \, .
 \label{eq:diracA-a}   \end{equation}
     \begin{equation} \sDT \equiv  \sD + v_{\rm D} A(T)\sigma _3     \, ,
 \label{eq:diracA-b} \end{equation}
 \end{subequations}
Here, $\sD$ is  given by \eqref{eq:D0_def} and the Pauli matrices $\sigma_1 $ and $\sigma _3 $, see \eqref{pauli}.  The effective Dirac model, and its validity guaranteed in Theorem \ref{thm:valid}, hold for all smooth and periodic choices of $A(T)$.

Let $\psi_0\mapsto \mathcal{U}^{\varepsilon}[\psi_0](t,x)$ denote the solution of Schr{\"o}dinger equation \eqref{eq:lsA} with $\psi(0,\cdot) =~\psi_0$, and $\alpha_0(X) \mapsto \mathcal{U}_{\rm Dir}[\alpha_0](T,X)$ denote the solution of the effective Dirac equation \eqref{eq:diracA} with $\alpha(0,\cdot) =~\alpha_0$.

\begin{theorem}[Effective Dirac dynamics]
\label{thm:valid}
 Suppose $\kappa$ has bounded derivatives to all orders. There exists $\varepsilon_0>0$ such that for all $0<\varepsilon<\varepsilon_0$, the following holds:
 Consider \eqref{eq:lsA}, the Schr{\"o}dinger equation with time periodic Hamiltonian, and initial data of the form, $\psi_0 (x) = \varepsilon^{1/2}\alpha_0(\varepsilon x)^{\top}\Phi (x)$, with $\alpha_0 \in H^4(\mathbb{R}^2;\mathbb{C}^2)$. 
Fix constants $T_0>0$ and  $0<\rho < 1$.  

Then, there is a constant, $C$, which depends on the Hamiltonian $H$,  $\rho$ and  ${\alpha} _0$,  such that for $0\le t\le  T_0\ \varepsilon^{-(3/2-\rho)}$,
\begin{equation}\label{eq:psi0_validity}
 \Big\|\mathcal{U}^{\varepsilon}[\psi_0] (t,x)- \varepsilon^{\frac12} \mathcal{U}_{\rm Dir}[\alpha_0](\varepsilon t,\varepsilon x)^\top \Phi(x;k_{\rm D}) e^{-iE_{\rm D}t} \Big\|_{L^2(\mathbb{R}_x)}  \le C \varepsilon^{\rho}  \, . 
\end{equation}
\end{theorem}

The proof is presented in Appendix \ref{ap:eff_pf}. Analogous results on the effective Dirac dynamics in two-dimensional analogs of~\eqref{eq:lsA}, unforced and forced, were obtained in \cite{FW14, SW21}. Theorem \ref{thm:valid} follows from arguments which closely follow those in \cite[Section 8]{SW21}.

 Note that, of the three potentials considered in this paper (see Fig.\ \ref{fig:U_delta}), Theorem~\ref{thm:valid} is only applicable to \eqref{eq:lsa_cos}. However, our numerical simulations suggest that the formally-derived Dirac equations effectively approximate the corresponding Schr{\"o}dinger equation for very long times, even if slightly less so than in the smooth settings. The study of this robustness in non-smooth settings remains an interesting open problem.

\section{Metastability of edge states and radiation damping}\label{metastability}

In Section \ref{summary-rd} we summarize our analytical results on radiation damping for 
 initial value problem for the periodically forced Dirac equation:
\begin{align}
\label{eq:dirac_beta-A}
     i\partial_T \alpha &= (\sD +\beta A(T)\sigma _3)\alpha \, , \qquad    \alpha(0,X) = \alpha_{\star}(X) \, .
\end{align}
Here, 
\begin{enumerate}
\item $A(T)$ is $\Tper$-periodic and of mean zero
\item  $\sD=i\sigma_3\D_X+\kappa(X)\sigma_1$, and 
\item  $\alpha_\star$ is the zero-energy defect mode of $\sD$:
\[ \sD\alpha_\star=0,\quad \|\alpha_\star\|_{L^2(\R;\C^2)}=1.\] 
\end{enumerate}
The parameter $\beta$, the strength of periodic forcing, is taken to be small.

After discussing our results for the forced Dirac initial value problem  \eqref{eq:dirac_beta-A}, we apply them  to radiation damping 
in the initial value problem for the forced Schr{\"o}dinger equation.
In Section \ref{sec:Dirac_simul} we  present numerical simulations which corroborate our theoretical predictions.
The multiscale analysis underlying the results in Section \ref{summary-rd} is presented in Section \ref{sec:decay_anal}.

\subsection{Summary of analytical results on radiation damping}\label{summary-rd}
For $\beta=0$, the solution of \eqref{eq:dirac_beta} is $\alpha(T,X)=\alpha_{\star}(X)$ for all $T\in\R$.

For $\beta$ small, it is natural to represent the solution to the perturbed evolution equation \eqref{eq:dirac_beta} 
 using the spectral theory of the operator $\sD$, outlined in Section \ref{sec:edge}. In particular, for the case that $\sD$ has only  one bound state we have:
 \[ {\rm spec}(\sD)\ =\   \{0\}\ \cup\ \R\setminus [-\vartheta_\sharp\kappa_{\infty}, \vartheta_\sharp \kappa_{\infty}].
 \]
 The $L^2(\R;\C ^2)$-orthogonal  projections onto the bound (defect) state subspace and the continuous (dispersive) spectral subspace of $\sD$ are, respectively: 
\begin{equation}\label{eq:P0Pc_def}
\projz=\left\langle\alpha_\star,\cdot\right\rangle\alpha_\star\quad {\rm and}\quad  \proj= I - \projz\qquad (\projz+\proj=I) \, .
\end{equation}
Thus, we decompose the solution of \eqref{eq:diracA} relative to these orthogonal projections:
\begin{equation}
    \alpha(T,X) = g(T)\alpha_\star(X)+ \alpha_d(T,X) \, ,\quad  \big\langle\alpha_\star(\cdot) , \alpha_d(T,\cdot)\big\rangle = 0 \, ,\qquad T\ge0 \, .
\label{eq:decomposition}    
\end{equation}
In Section \ref{sec:decay_anal} we derive a coupled dynamical system for the oscillator-like degree of freedom, $g(T;\beta)$
 and field-like degree of freedom $\alpha_{\rm d}(T,X;\beta)$, which we solve for small $\beta$ via a multiple scale expansion.
Our constructed $g(T)$ decays on the time scale $\mathcal{O}(\beta^{-2})$.  In particular, we show the following:

\noindent  {\bf Leading order multiscale expansion:}\ Fix $T_0>0$ and arbitrary.
   Then, there exist $\rho_0, \beta_0>0$ such that for all $0<\beta <\beta_0$ and  $0\le T\le T_0\beta^{-2}$ 
\begin{subequations}
  \begin{align}
  g(T;\beta) &= e^{i\left( \beta^2 \Lambda_0T - \beta\eta_A(T)\right) }\ e^{-\Gamma_0\beta^2 T} \left(\ 1\ +\ o(1)\ \right)
  \label{g-exp}\\
  \alpha_{\rm d}(T,X) & = \mathcal{O}\left( \beta e^{-\Gamma_0\beta^2 T}  \right)\quad \textrm{in $L^2(\langle x\rangle^{-\rho_0}dx)$},\quad \textrm{as $\beta\to0$.}
  \label{alph-exp}
  \end{align}
  \label{galpha-exp}
\end{subequations}
Here, $\Gamma_0=\Gamma_0(\omega)$ is given by
\begin{equation}\label{eq:gamma0}
\Gamma_0(\omega) \equiv \frac{\pi}{4} \Big< \proj \sigma_3 \alpha_{\star},\left(\delta(\sD+\omega)+\delta(\sD-\omega)\right)\proj \sigma_3\alpha_{\star} \Big> \ge0 \, ,
\end{equation} 
is generically {\em strictly positive}, 
\begin{equation}\label{eq:Lambda0}
\Lambda_0(\omega) \equiv \frac{1}{4} \left< \proj \sigma_3 \alpha_{\star}, \left({\rm PV}\frac{1}{\sD+\omega} +{\rm PV}\frac{1}{\sD-\omega}\right)\proj \sigma_3\alpha_{\star} \right> \, ,
\end{equation}
and
\[ \eta_A(T)= \left\langle \alpha_\star,\sigma_3\alpha_\star\right\rangle\int_0^TA(s) ds\]
 is a real, bounded, and $\Tper$-periodic.

\medskip

\noindent{\bf Exponential decay on the time scale $\mathcal{O}(\beta^{-2})$:} Note from \eqref{galpha-exp} that exponential decay 
 on the time-scale $T\sim \beta^{-2}$ holds if $\Gamma_0(\omega)>0$.
Since $\delta(\sD \pm \omega)$ are non-negative self-adjoint operators, $\Gamma_0(\omega)\ge0$ for any $\omega$.  The vector $\delta(\sD \pm \omega)\proj \sigma_3\alpha_{\star}$ is the projection of  vector $\proj \sigma_3\alpha_{\star}$ onto the continuous spectral subspace of $\sD$ at frequency $\mp\omega$. Hence, 
if $|\omega|<\kappa_{\infty}|\vartheta_{\sharp}|$, i.e., $\omega$ is in the spectral gap of $\sD$, then $\Gamma_0(\omega)=0$, see Fig.\ \ref{fig:decay_cartoon}.

On the other hand, suppose $|\omega|>\kappa_{\infty}|\vartheta_{\sharp}|$, then $\omega$ is in the continuous spectrum of $\sD$. For generic data, $\delta(\sD \pm \omega)\proj \sigma_3\alpha_{\star}$ is non-zero, as we expect generic functions to have non-zero projection on every part of the continuous spectrum \cite[Section 4]{agmon1989perturbation}. Hence, generically $\Gamma_0>0$ implying exponential decay.

Since $\proj \sigma_3 \alpha_{\star}\in L^2(\R ;\C^2)$, the inner product \eqref{eq:gamma0} which gives the exponential rate of decrease, $\Gamma_0(\omega)$,  decays  to zero for $|\omega| \to\infty$. Our numerical simulations in the bottom row of Figure \ref{fig:decay_cartoon} are consistent with this behavior; we observe that  $\omega\mapsto g(T;\omega)\large|_{T=500}$ {\it increases} for $\omega$ sufficiently large, consistent with rate of decay of  $g(T;\omega)$ being slower for large $|\omega|$.

\begin{remark}
 We expect that the exponential decay is a large-time transient,
which applies on the time scale $T\lesssim \beta^{-2}$. While our analysis does not cover $|T|\gg \beta^{-2}$, we expect the time-decay to be algebraic and given by the free dispersive rate, $\mathcal{O}(t^{-1/2})$; see, e.g., \cite{SW:98}.
\end{remark}

\begin{figure}[h]
%	\begin{subfigure}[t]{1\textwidth}
%	\caption{}
%            \centering
%             \hskip -2ex
%            \includegraphics[width=.7\linewidth]{spec_D0_Hdw_omega_1.pdf}
%            \label{fig:dpec_D0_omega}
%        \end{subfigure} \vskip 0ex
        \begin{subfigure}[t]{1\textwidth}
        	\caption{}
            \centering
             \hskip -2ex
            \includegraphics[width=.7\linewidth]{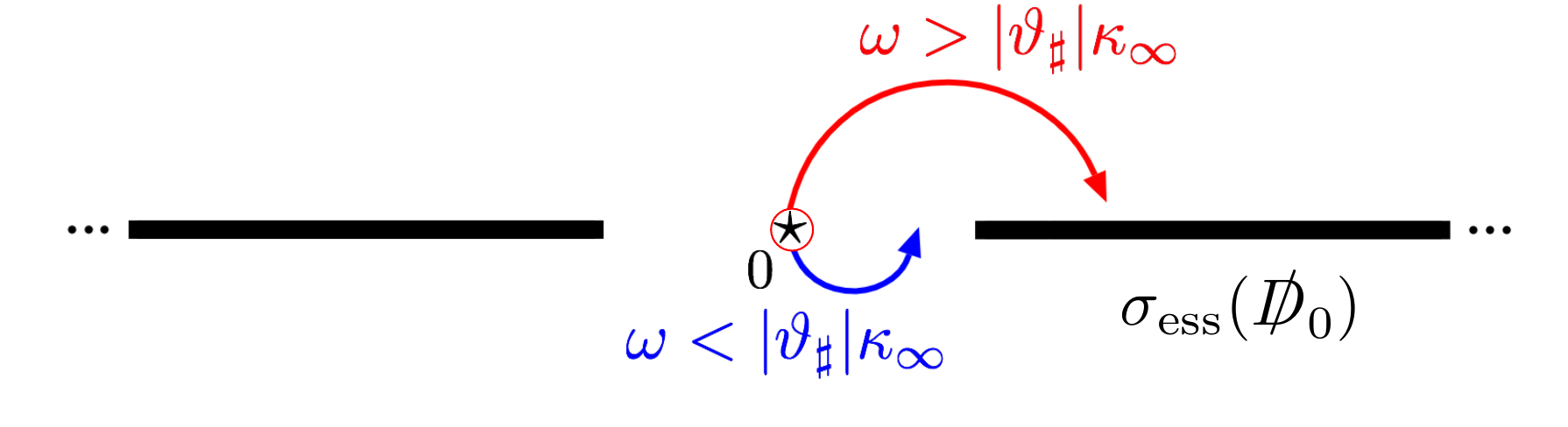}
            \label{fig:dpec_D0_omega}
        \end{subfigure} \vskip 0ex
        \begin{subfigure}[t]{1\textwidth}
        	\caption{}
            \centering
            \hskip - 5ex
            \includegraphics[width=0.7\linewidth]{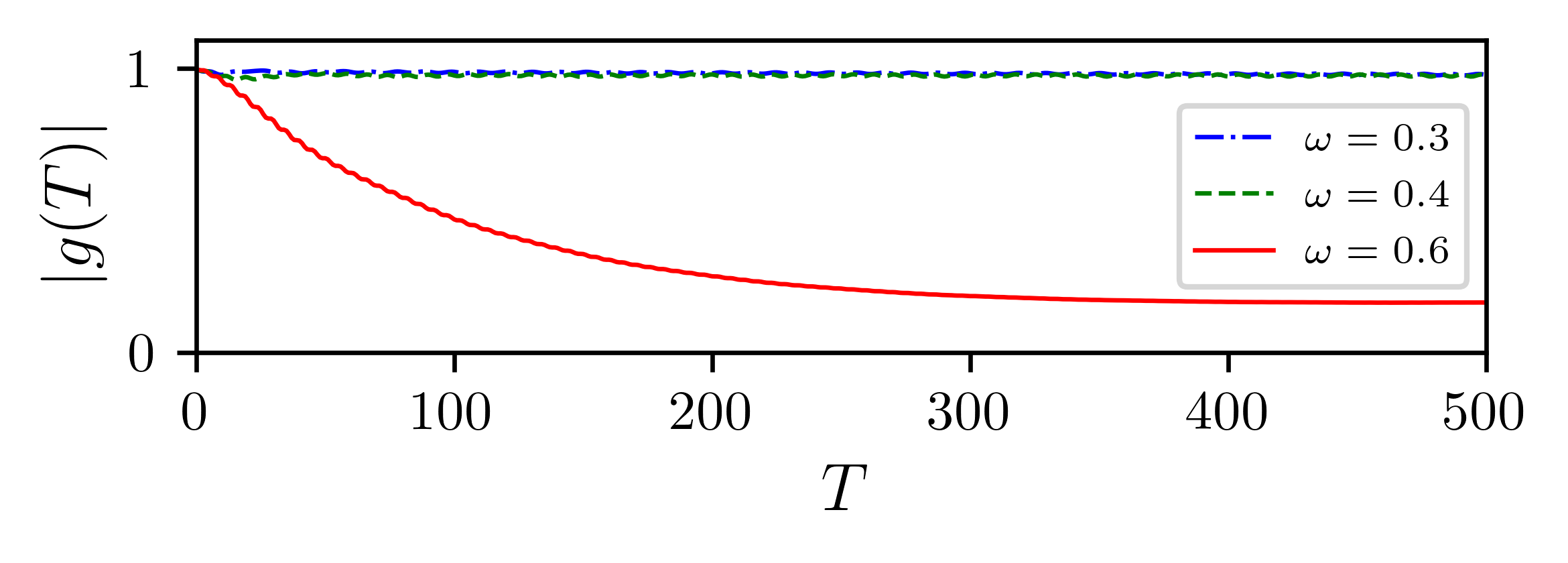}
            \label{fig:omega}
        \end{subfigure} \vskip -0ex
        \begin{subfigure}[t]{1\textwidth}
        	\caption{}
            \centering
            \hskip - 5ex
            \includegraphics[width=0.7\linewidth]{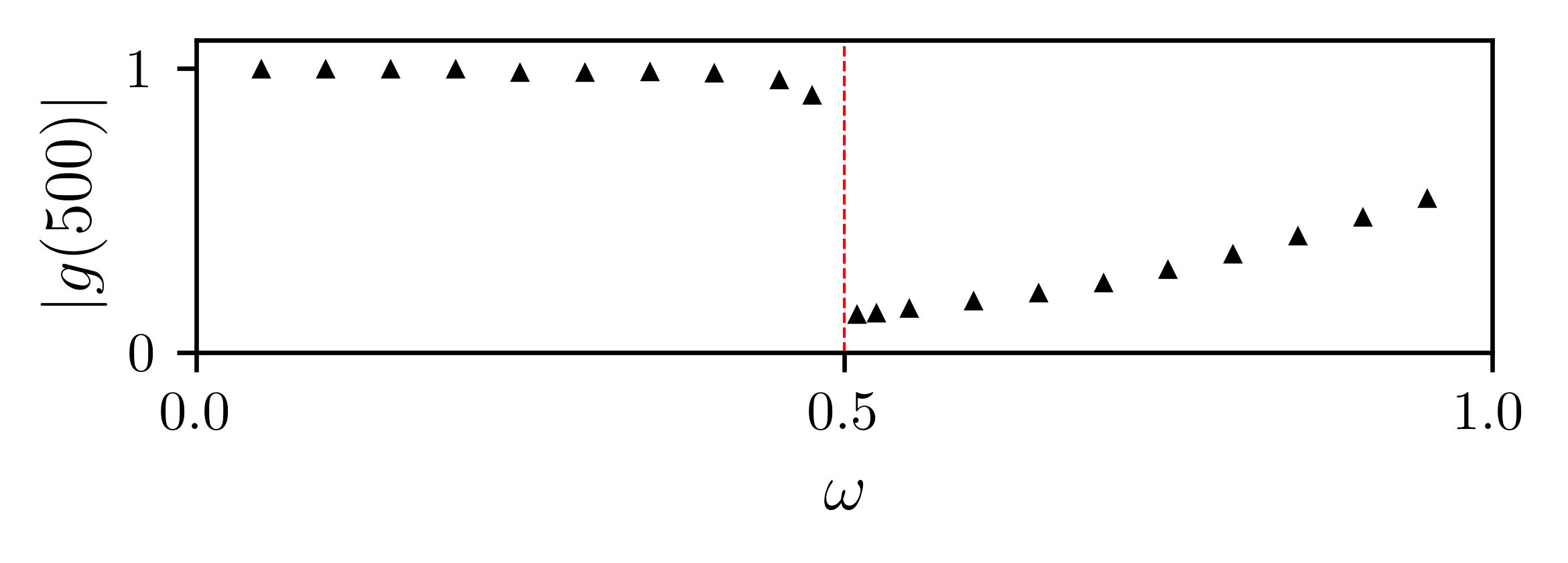}
            \label{fig:g500}
        \end{subfigure}
  	\caption{The effect of the frequency $\omega$ of the forcing $A(T)$ on the decay of the Dirac zero mode $\alpha_{\star}$, see \eqref{eq:0mode}. \textbf{(a)} When $|\omega|>~\vartheta_\sharp\kappa_\infty$ (red), $\alpha_{\star}$ couples to the essential spectrum. Otherwise $|\omega| < \vartheta_{\sharp}\kappa_{\infty}$ (blue), and the coupling effect is expected to be negligible. \textbf{(b)} Same settings as in Fig.~\ref{fig:DirSch_cos_unf-intro}, where $\vartheta_\sharp\kappa_\infty=0.5$. The projection onto the zero mode $g(T)=\langle\alpha_\star,\alpha(\cdot,T)\rangle_{L^2(\R; \C ^2)}$ for $\beta=0.01$, with $\omega\in\{0.3,0.4,0.6\}$. The curves for $\omega=0.3$ and $\omega=0.4$ are nearly indistinguishable. \textbf{(c)} $g(500)$ for varying values of $\omega$ (triangles) and a dotted line for the $\omega=\vartheta_{\sharp}\kappa_{\infty}$ threshold.}
  	\label{fig:decay_cartoon}
\end{figure}

%
% If $\omega < \kappa_{\infty}\vartheta_{\sharp}$, then $\varepsilon(\sD-\omega)\proj \sigma_3\alpha_{\star} = 0$, since the spectrum of $\sD$ acting on the image of $\proj$, which is the essential spectrum of $\sD$, is bounded away from $\pm \omega$. On the other hand (and by the same argument), it is {\em generically true} that $\Gamma_0 >0$ when $\omega > \kappa_{\infty}\vartheta_{\sharp}$, and so it is a generic condition for us to observe decay of $\alpha_{\star}$; see Fig.\ \ref{fig:decay_cartoon}.
%

\medskip

\noindent{\bf Radiative time-decay in parametrically forced Schr{\"o}dinger equation:}\\
In terms of the original variables of the parametrically forced Schr{\"o}dinger equation~\eqref{eq:lsA}
  we have, by \eqref{g-exp}, for $\rho\in(0,1)$ and $0\leq t \leq T_0 \varepsilon^{-(2-\rho)}$,
 $$\langle \psi_{\star}, \psi(t,\cdot) \rangle_{L^2(\R)} = e^{-i\beta \eta _A(\varepsilon t)} e^{-iE^{\varepsilon}t}e^{-\beta^2 \varepsilon(\Gamma_0 + i\Lambda_0)t} + \mathcal{O}(\beta , \varepsilon^{\rho}) \,  , $$
 and where $E^{\varepsilon}=E_D +\mathcal{O}(\varepsilon^2)$; see \eqref{eq:psistar_def}.  This gives exponential decay 
 on the time-scale $\mathcal{O}(\varepsilon^{-1}\beta^{-2})$.

 \subsection{Simulations of the effective Dirac equation}\label{sec:Dirac_simul}

\begin{figure}[h]
          \begin{subfigure}[t]{1\textwidth}
            \caption{Unforced}
            \centering
            \includegraphics[width=.35\linewidth]{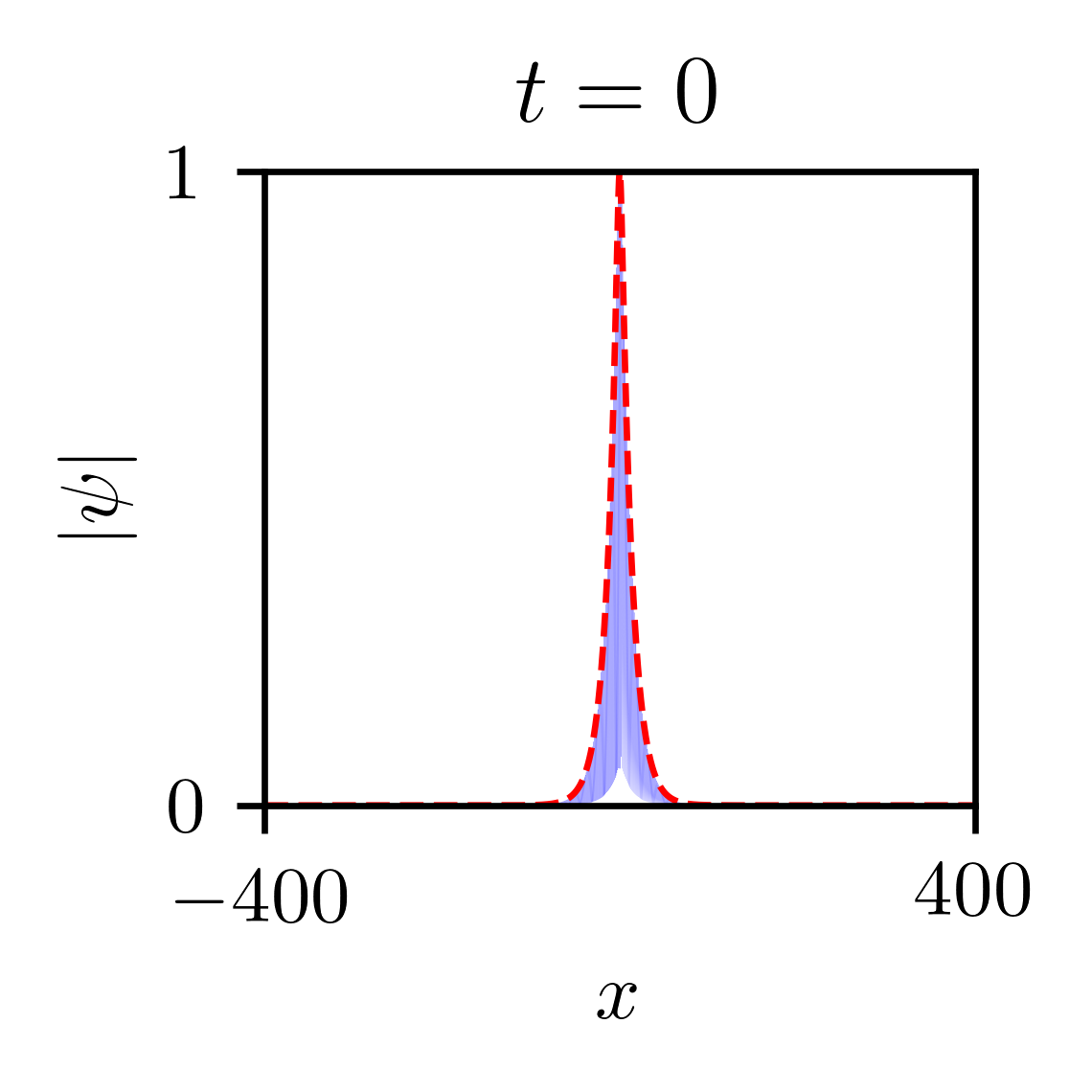}
             \hskip -4ex
            \includegraphics[width=.35\linewidth]{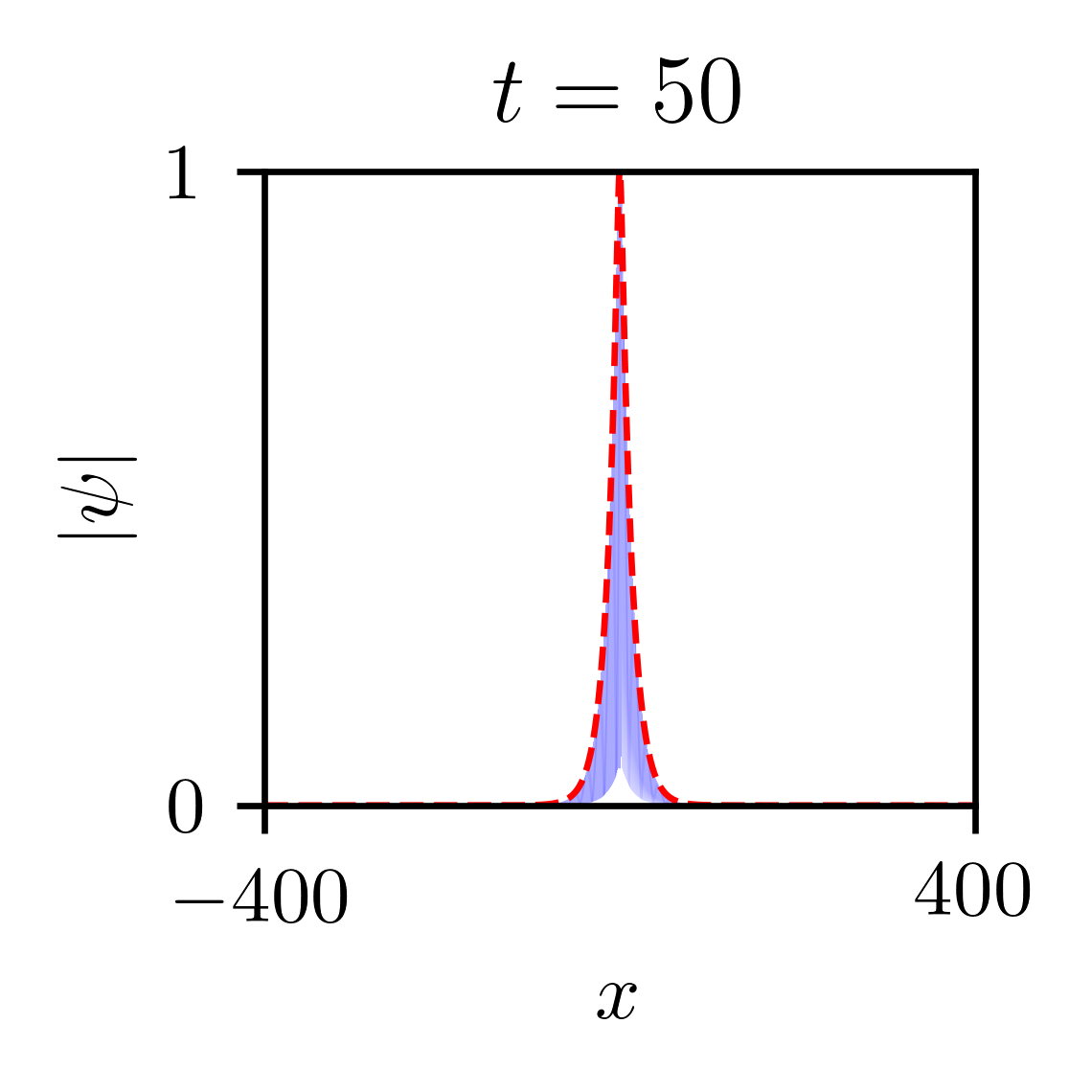}
             \hskip -4ex
            \includegraphics[width=.35\linewidth]{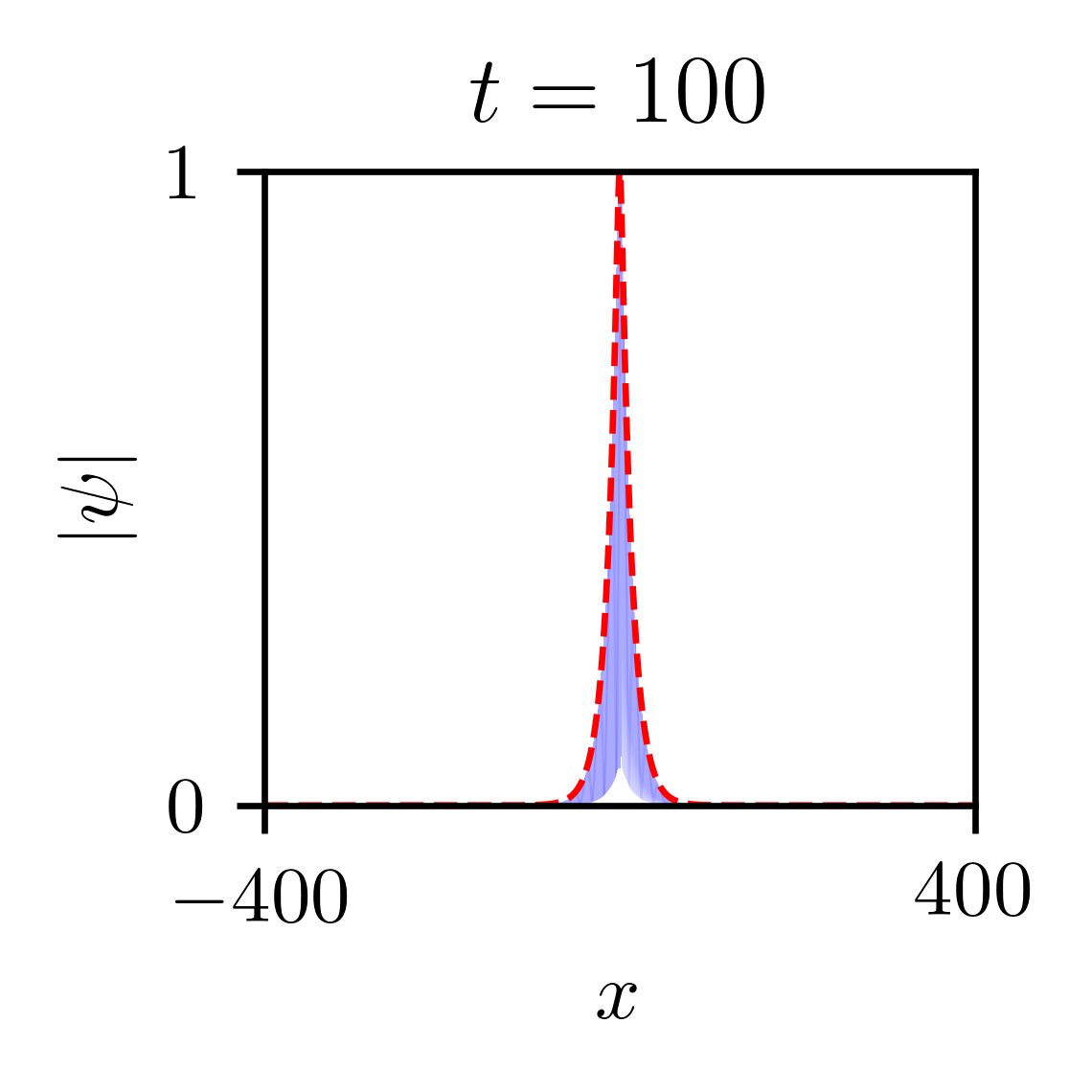}
            %\label{fig:unforcedDynSchro_atomic}
         \end{subfigure}
         \begin{subfigure}[t]{1\textwidth}
            \caption{Forced}
            \centering
            \includegraphics[width=.35\linewidth]{t=0_wells.png}
              \hskip -4ex
            \includegraphics[width=.35\linewidth]{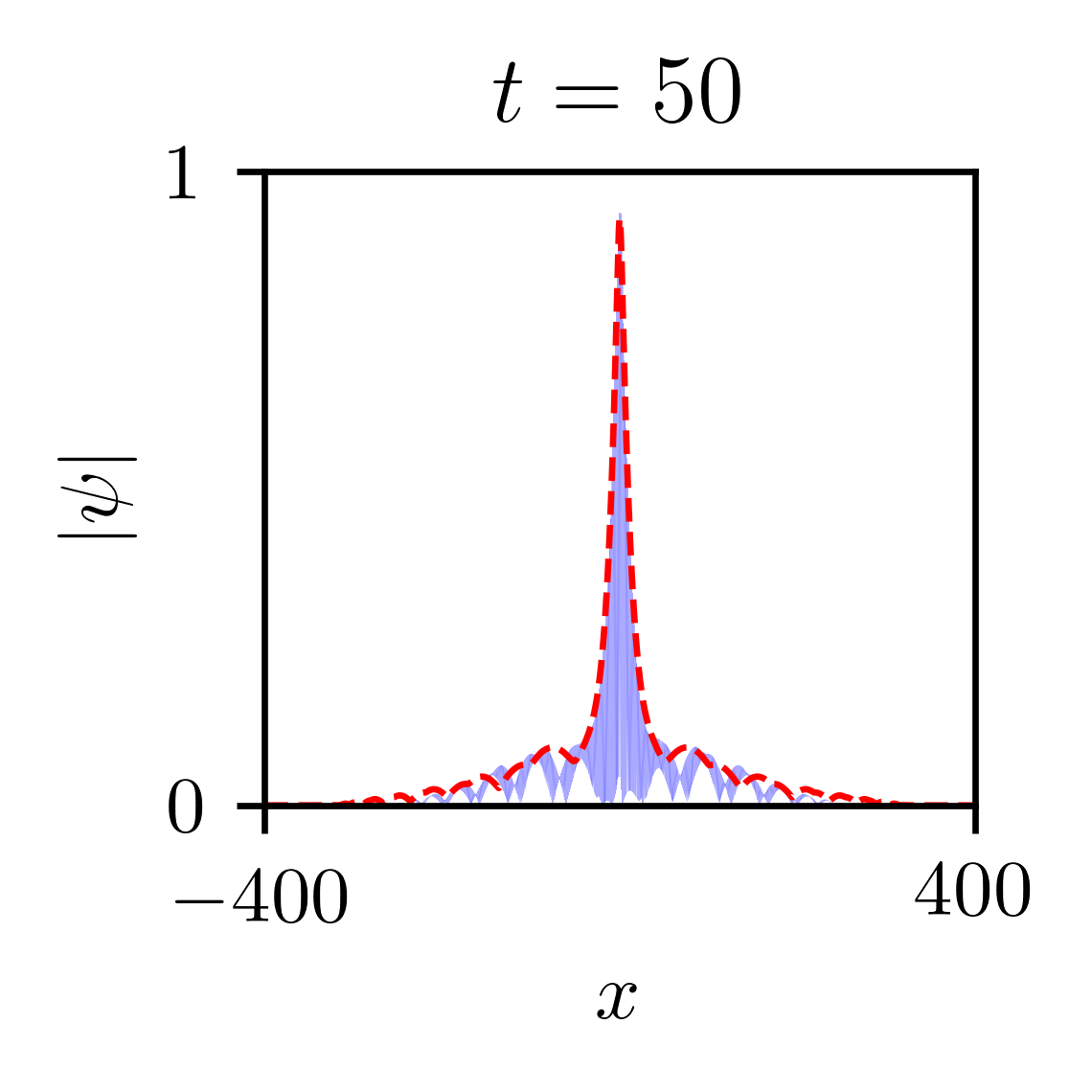}
             \hskip -4ex
            \includegraphics[width=.35\linewidth]{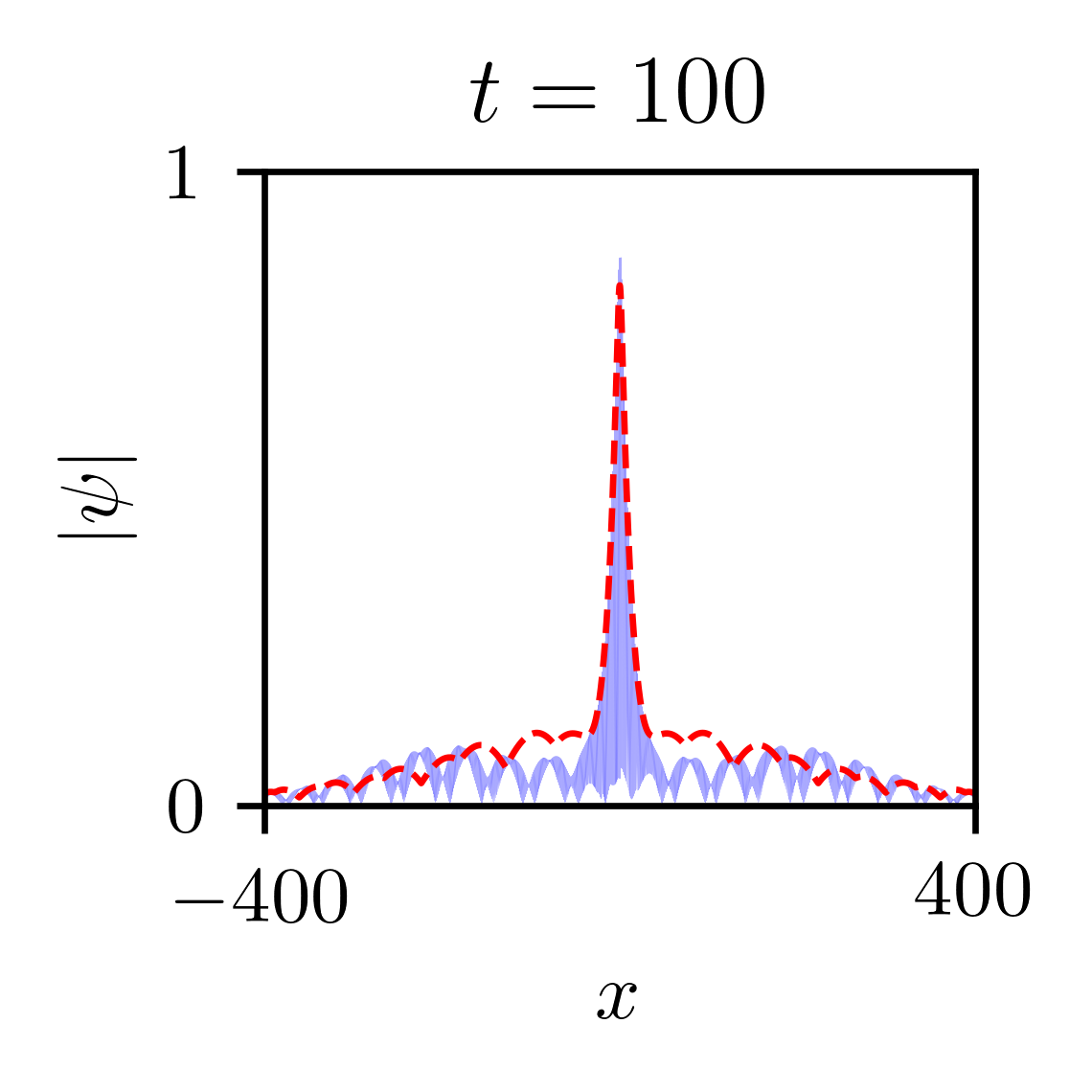}
            %\label{fig:forcedDynSchro_atomic}
        \end{subfigure}
        \caption{Same as Fig.\ \ref{fig:DirSch_cos-intro}, only for the square well Schr{\"o}dinger and Dirac equations, \eqref{eq:lsa_well} and \eqref{eq:diracA_well}, respectively, with $\omega = 1.1$.}
\label{fig:DirSch_well}
\end{figure}

For the  time-periodically forced Schr{\"o}dinger equations \eqref{eq:lsA} with the potentials $U^{(\ell)}_{\varepsilon}\ (\ell=1,2,3)$, introduced in Section \ref{numerics},
%$U^{(1)}_{\varepsilon}$ and $U^{(2)}_{\varepsilon}$, 
we consider the effective periodically-forced Dirac Hamiltonians
\[ \slashed{D}^{(\ell)}(T)= \sD^{(\ell)}+\vD A(T)\sigma_3\quad  (\ell=1,2,3)\,;\]
 see~\eqref{eq:diracA-b}. The parameters which define these effective operators 
 are given in Appendix~\ref{ap:potentials}.

To confirm the analytical asymptotic results outlined in Section \ref{summary-rd}, we simulate the evolution under the
 effective Hamiltonians $\slashed{D}^{(1)}(T)$ (Fig.\ \ref{fig:DirSch_cos-intro}) and $\slashed{D}^{(2)}(T)$(Fig.\ \ref{fig:DirSch_well}), and compare the solution $\alpha (T,X)$ with the wave-envelope of the solution of the corresponding Schr{\"o}dinger equations \eqref{eq:lsA}  with $\varepsilon=1/2$. In the unforced case ($\beta = 0$), the envelope of the Schr{\"o}dinger defect mode $| \psi_{\star}(x)|$ is tracked  very well by $|\alpha_{\star}(\varepsilon x)|$ for $0\leq t \leq 100$, as we expect from \eqref{eq:psialpha_star}.\footnote{Note that, since $\varepsilon = 1/2$ in \eqref{eq:lsA}, the simulation of \eqref{eq:diracA} runs on the slow time-scale of $0\leq T \leq 50$.} For the forced cases, we see that the Dirac envelope decays and fits the multi-scale Schr{\"o}dinger solution extremely well for $t=50$, and quite well (though small discrepancies do appear) up to $t=100$.

\begin{figure}[t]
	%\hskip -4ex
            \centering
             \hskip -5ex
            \includegraphics[width=.8\linewidth]{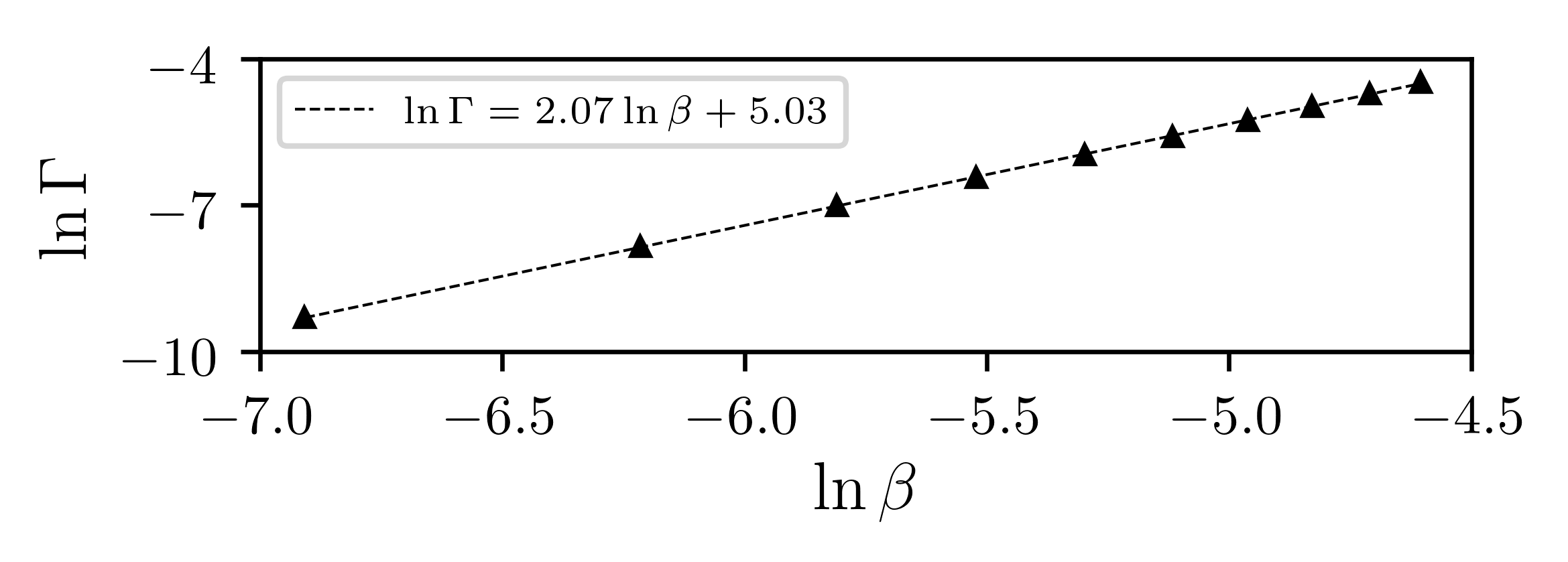}\vskip -2ex
             \hskip -5ex
              \includegraphics[width=.8\linewidth]{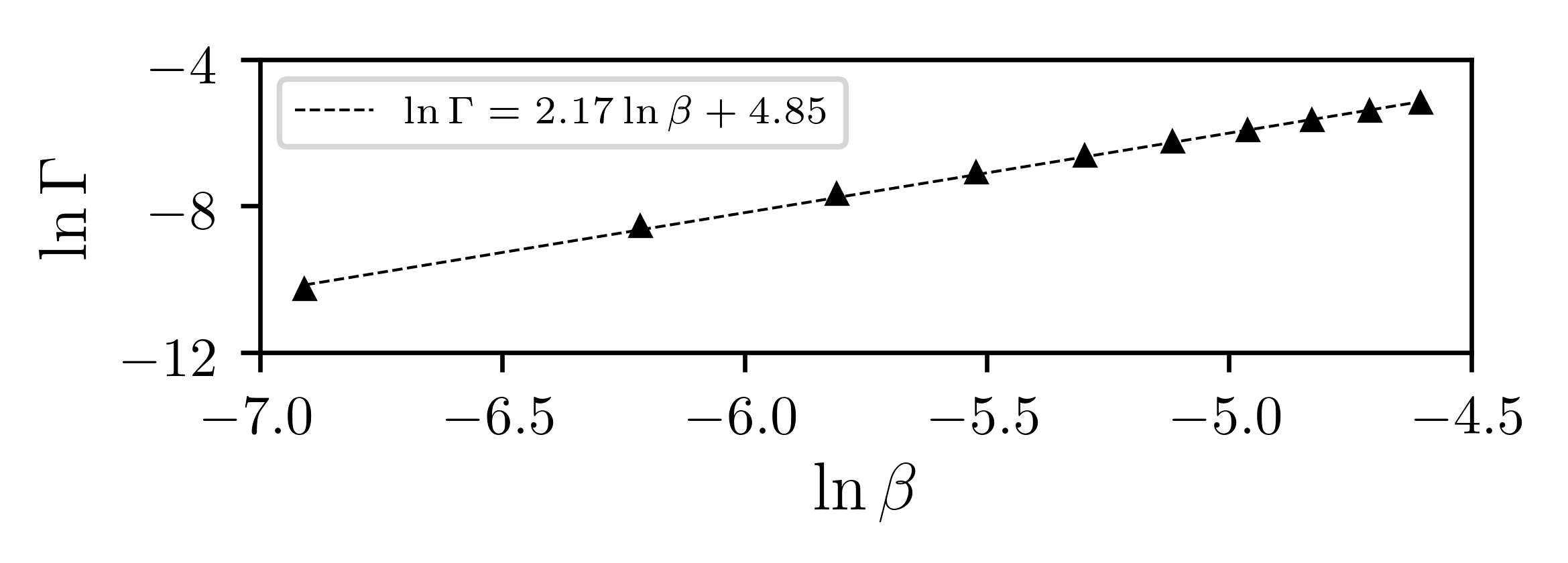} \vskip-2ex
               \hskip -5ex
            \includegraphics[width=.8\linewidth]{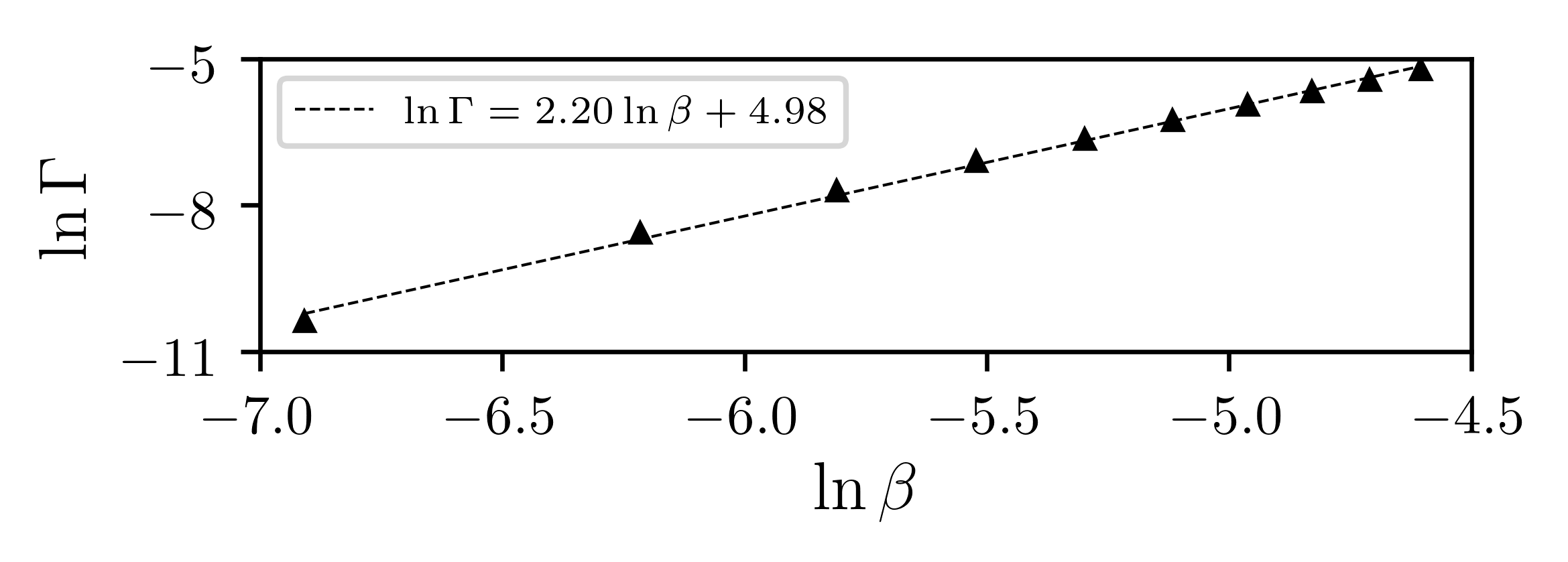}
        \caption{Decay rates for the effective Dirac equations as a function of $\beta$, see \eqref{eq:diracA}. {\bf Top:} $\sD^{(1)}$. {\bf Middle:} $\sD^{(2)}$. {\bf Bottom:} $\sD^{(3)}$.}
\label{fig:Dirac_cos_proj}
\end{figure}

Figure\ \ref{fig:Dirac_cos_proj} presents  numerical results on radiation damping rate for the effective Dirac Hamiltonians 
$\slashed{D}^{(\ell)}(T),\ (\ell=1,2,3)$ which are consistent with the
 predicted exponential decay rate, $\Gamma(\beta)=\Gamma_0 \beta^2$,  (see \eqref{g-exp}). These results are almost one-to-one comparable with those of the corresponding Schr{\"o}dinger equations; compare with Figure \ref{fig:Sch_gamma}.
 
It is remarkable that the effective dynamics given by $\slashed{D}^{(3)}(T)$ (with a discontinuous $\kappa(X)$) tracks the Schr{\"o}dinger dynamics for large time, since
  the Hamiltonian violates the scaling assumptions used to derive the effective Dirac dynamics from the  Schr{\"o}dinger dynamics.
  We remarked on this as an open problem in Section~\ref{sec:discussion}.
%  \footnote{
%\textcolor{red}{We note here that, while formally $\sD^{(3)}$ is the Dirac equation corresponding to \eqref{eq:lsA} with $U^{(3)}_{\varepsilon}$, it cannot be rigorously justified by our current analysis: $U^{(3)}_{\varepsilon}$ does not enjoy scale separation, as the transition between the two bulk potentials does not happen along the slow space variable $X=\varepsilon x$. Nor does our current analysis of the radiation damping phenomena applied to $\sD^{(3)}$, as we rely on the smoothness of $\kappa (X)$. Nevertheless, in as of itself, we believe that our radiative decay analysis in Section\ref{sec:decay_anal} can be extended to $U^{(3)}_{\varepsilon}$ and $\sD^{(3)}$, as condirmed by our numerical results here and in Section \ref{sec:sgn_sch}. These results lead to some of the open problems discussed in Section \ref{sec:discussion}.}
%}

\section{Multiscale analysis of radiation damping}\label{sec:decay_anal}

In this section, we provide a multiple scale analysis and derivation of the radiation damping phenomena observed in our numerical simulations. Consider the initial-value-problem (IVP) for $\alpha(T,X)$, the  effective parametrically forced Dirac equation \eqref{eq:diracA} with initial data given by the zero energy defect mode,  $\alpha_{\star}(X)$, of the unforced Dirac operator $\sD$:
\begin{align}
\label{eq:dirac_beta}
     i\partial_T \alpha &= (\sD +\beta A(T)\sigma _3)\alpha \, , \qquad    \alpha(0,X) = \alpha_{\star}(X) \, .
\end{align}
 From here on, it will be useful to make the forcing amplitude parameter $\beta$ explicit, multiplying the forcing function $A(T)=\cos (\omega T)$, by a slight abuse of notation; $\beta$ is real and will be taken sufficiently small.  
Recall that since $\alpha_\star$ is the zero energy eigenstate of $\sD$, we have that if $\beta=0$, then $\alpha(T,X)=\alpha_{\star}(X)$.

%We shall derive, by a multiple scale procedure, an asymptotic expansion of the solution of this IVP for $\beta>0$ small.
%\footnote{\textcolor{red}{Michael, I think this paragraph does not belong here (and was made clear in the previous section already).}By Theorem \ref{thm:valid},  the solution $\alpha(T,X)$ of \eqref{eq:dirac_beta} gives rise to a long-time valid approximate solution: 
%\[ \psi^\varepsilon_{\rm approx}(t,x) = \varepsilon^{\frac12}\alpha(\varepsilon t,\varepsilon x)^{\top}\Phi(x)e^{-iE_{\rm D} t}\]
%of the Schr{\"o}dinger equation \eqref{eq:lsA}.
%$\varepsilon^{\frac12}\alpha(\varepsilon t,\varepsilon x)$ tracks the envelope, as observed numerically,  and $\psi^\varepsilon_{\rm approx}(t,x)$ tracks the solution on both the envelope and carrier length scales; see Sections \ref{sec:dirac} and \ref{sec:Dirac_simul}.} 

Since $\beta$ is small it is natural to 
decompose $\alpha$ into its projection onto the bound state and the dispersive part of $\sD$, orthogonal to $\alpha_\star$
\begin{equation}
    \alpha(T,X) = g(T)\alpha_\star(X)+ \alpha_d(T,X) \, ,\qquad  \big\langle\alpha_\star(\cdot) , \alpha_d(T,\cdot)\big\rangle = 0 \, ;
\label{eq:decomposition-1}
\end{equation}
see also~\eqref{eq:decomposition}.

Inserting \eqref{eq:decomposition} into \eqref{eq:diracA}, and using the relation $\slashed{D}_0\alpha_\star = 0$, we obtain\footnote{In the calculations below, we shall frequently suppress the $X$-dependence of $\alpha_d$. Thus,
$\alpha(T)=g(T)\alpha_\star+\alpha_d(T)$.}
\begin{equation}
    i\D_Tg(T)\alpha_\star + i\partial_T\alpha_d(T) = \sD\alpha_d(T) + \beta \vD A(T)g(T)\sigma_3\alpha_\star +  \beta \vD A(T)\sigma_3\alpha_d(T) \, .
\label{eq:dirac_decomp_subbed}
\end{equation}
Applying the projections $\projz$ and $\proj$, as defined in \eqref{eq:P0Pc_def}, to  \eqref{eq:dirac_decomp_subbed}, we obtain the coupled system for 
 $g(T)$ and $\alpha_d(T)$:
 \begin{align}
i\partial_T g(T) &= \beta \big\langle\alpha_\star, \sigma_3\alpha_\star\big\rangle A(T) g(T) + \beta  A(T)\ \big\langle\alpha_\star, \sigma_3\alpha_d(T)\big\rangle \, ,
\label{eq:g_ode}\\
i\partial_T\alpha_d(T) &= \sD\alpha_d(T) + \beta\ \proj \sigma_3\alpha_\star\   A(T) g(T)+ \beta A(T) \proj \sigma_3\alpha_d(T) \, ,
\label{eq:proj_pde}\\
g(0) &=1 \, ,\qquad \alpha_d(0) =0. \label{galpha-data}
\end{align} 
To avoid cumbersome expressions, we have set $\vartheta_{\sharp}=\vD=1$.
The first term on the right hand side of~\eqref{eq:g_ode} contributes a rapidly varying phase, which we can remove  by setting
\begin{align}
	G(T) \equiv e^{i\beta\eta_A(T)}g(T),\qquad {\rm where} \nonumber\\
\label{eq:eta_def}
		\eta_A(T) \equiv \big\langle \alpha_\star, \sigma_3\alpha_\star\big\rangle \int_0^T A(s) \, ds \, .
\end{align}
Note that since $A(T)$ has mean zero,\footnote{Without loss of generality, we can assume $A(T)$ has mean zero. Otherwise, the mean of $A(T)$ can be removed by a change of variables.}
$\eta_A(T)$ is bounded on $\R$.
Equations \eqref{eq:g_ode}, \eqref{eq:proj_pde}, \eqref{galpha-data} becomes a system for $G(T)$ and $\alpha_d(T)$:
\begin{subequations}
\begin{align}
	i\partial_TG(T) &= \beta e^{i\beta\eta_A(T)} A(T)\ \big\langle \alpha_\star, \sigma_3\alpha_d(T)\big\rangle \, ,\\
	\left( i\partial_T  - \sD\right)\ \alpha_d(T) &= \beta\proj\sigma_3\alpha_\star e^{-i\beta\eta_A(T)}A(T)\ G(T) + \beta\proj\sigma_3\ A(T)\ \alpha_d(T) \, ,\\
	G(0)&=1,\quad \alpha_d(0)=0. \label{data1}
\end{align}
\label{G_and_alpha_d_system}
\end{subequations}

The structure of \eqref{G_and_alpha_d_system} and numerical simulations (Section \ref{sec:Dirac_simul}) suggest that solution varies on  disparate temporal scales: from the rapid time-scale of order $\beta^0=1$, of the forcing function, $A(T)$, to the decay time-scale of the bound state amplitude, $\beta^{-2}$. Thus, we shall seek the solution of \eqref{G_and_alpha_d_system} in the form of a multiple scale expansion.
First, introduce the hierarchy of time scales 
\[T,\quad  \tau_1=\beta T,\quad \tau_2=\beta^2 T,\quad \textrm{($\beta$ small)}\] and view $G$ and $\alpha_d$ as functions of these variables. 
Thus,   \eqref{G_and_alpha_d_system} becomes:
\begin{subequations}
\begin{align}
	& i\big(\D_T+\beta\D_{\tau_1}+\beta^2\D_{\tau_2}\big)G(T,\tau_1,\tau_2) = \beta e^{i\beta\eta_A(T)} A(T)\ \big\langle \alpha_\star, \sigma_3\alpha_d(T,\tau_1,\tau_2)\big\rangle\  \, ,\\
	& i\big(\D_T+\beta\D_{\tau_1}+\beta^2\D_{\tau_2}\big)\alpha_d(T,\tau_1,\tau_2) - \sD\alpha_d(T,\tau_1,\tau_2)\nonumber\\
	&\qquad\qquad = \beta\proj\sigma_3\alpha_{\star} e^{-i\beta\eta_A(T)}A(T)\ G(T,\tau_1,\tau_2)
	+ \beta\proj\sigma_3A(T)\ \alpha_d(T,\tau_1,\tau_2) \, .\
\end{align}
\label{eq:ga_tausys}
\end{subequations}
\noindent
Next, we expand the solution $(G(\cdot;\beta), \alpha_d(\cdot;\beta))$ of  \eqref{eq:ga_tausys} in the small parameter~$\beta$:
\begin{subequations}
\begin{align}
G&=G^{(0)}+\beta G^{(1)} + \beta^2 G^{(2)} + \cdots \, , \qquad G^{(n)}(T,\tau_1,\tau_2)\in \C \, ,\\
\alpha_d&=\alpha^{(0)}+\beta \alpha ^{(1)} + \beta^2 \alpha ^{(2)} + \cdots \, , \qquad \alpha^{(n)}(T,\tau_1,\tau_2) \perp \alpha_\star\ .
\end{align}
\label{expansion}
\end{subequations}
We capture the initial data \eqref{data1} by imposing
 the initial conditions:
\begin{subequations}
\begin{align}
G^{(0)}&=1,\quad \textrm{for}\quad T=0,\ \tau_1=0,\ \tau_2=0 \, , \label{ic-G0}\\
G^{(n)}&=0,\quad\textrm{for}\  T=0,\ \tau_1=0,\ \tau_2=0\ \textrm{and all $n\ge1$} \, ,
\label{ic-Gn} \\
\alpha^{(n)} &=0,\quad\textrm{for}\  T=0,\ \tau_1=0,\ \tau_2=0\ \textrm{and all $n\ge 0$} \, .
\label{ic-aln}
\end{align}
\label{ics}
\end{subequations}

Substitution of \eqref{expansion} into \eqref{eq:ga_tausys} and equating like terms in powers of $\beta$, leads to a hierarchy
 of equations for  the functions $\{G^{(n)}\}_{n\ge0}$ and $\{\alpha^{(n)}\}_{n\ge0}$.
 These functions are  determined  by the requirement that contributions from terms at one order in $\beta$ are smaller or equal
  to terms of lower order in $\beta$ on the radiation damping   time-scale $T\sim \beta^{-2}$. 
  
  For any fixed $\hat{T}>0$ and $n\ge1$, we require that:
\begin{align}  
&\lim_{\beta\downarrow0}\ \sup_{0\le T\le \hat{T}/\beta^2} \beta^n \ |G^{(n)}(T,\beta T, \beta^2T)|\ =\ 0\, ,
\label{small-corrector-G}\\
&\lim_{\beta\downarrow0}\ \sup_{0\le T\le \hat{T}/\beta^2} \beta^n\  \| \left\langle X\right\rangle^{-\rho} \alpha^{(n)}(T, \beta T,\beta^2T,X)\|_{L^2(\R_X)} \ =\ 0 \, .
\label{small-corrector-alph}
\end{align}
Here, $\rho>0$ is fixed and is taken sufficiently large.

The norm in \eqref{small-corrector-alph} is a measure of spatially localized energy and below we see how it naturally arises. We shall construct the order $\beta^0$, $\beta^1$ and $\beta^2$ terms of the expansion \eqref{eq:ga_tausys}  in order to exhibit the transient exponential decay on the time-scale $T\sim\beta^{-2}$. The expansion can be carried out systematically to any finite order in $\beta$.
  Section \ref{interlude} assembles some technical tools used along the way and the expansion is constructed in Section~\ref{sec:expand}.

\subsection{Technical interlude}\label{interlude}

Throughout this section, assume that $\kappa_{\infty}=1$, i.e., that $\lim_{X\to \pm \infty}\kappa(X) = \pm 1$, and furthermore, that the convergence to $\pm 1$ is sufficiently fast.
\footnote{A sufficient condition for the decay rate of $\kappa (X)$ emerges from our application of wave operators in Section \ref{sec:decay_II}; see \cite{weder2000inverse, yajima1995wk}.}
 We shall also assume that $\kappa(X)$ is smooth. While such assumptions are stringent, our numerical simulations in Sections \ref{numerics} and \ref{sec:Dirac_simul} suggest that the phenomena is not dependent on such smoothness properties.

For any $\epsilon>0$, the operator $e^{-i\sD s}/(\sD -\lambda -i\epsilon)$ is bounded in $L^2(\R)$. 
Introduce  the operator 

  \begin{equation} \label{prop-res}
   \frac{e^{-i\sD s}}{(\sD -\lambda -i0)} \equiv 
   \lim_{\epsilon\to0^+} \frac{e^{-i\sD s}}{(\sD -\lambda -i\epsilon)}
  \, .
  \end{equation}
  
  \begin{proposition}\label{thm:ld-est}  
 Assume $\lambda\in {\rm spec}_{\rm ess}(\sD)$ but that $\lambda$ is not an endpoint of  
${\rm spec}_{\rm ess}(\sD)$. Then, for any $r\ge r_0>0$  sufficiently large, there exist $\rho>0$ such that 
the operator 
$e^{-i\sD s}(\sD -\lambda -i0)^{-1}$  is well-defined  from $L^2(\left\langle x\right\rangle^r dx)$ to $L^2(\left\langle x\right\rangle^{-r} dx)$ and satisfies the following local energy decay bound for all $ s>0$:
\[ \Big\| \left\langle x \right\rangle^{-r} \frac{e^{-i\sD s}}{(\sD -\lambda -i0)}\left\langle x \right\rangle^{-r} \Big\|_{\mathcal{B}(L^2)} \le C \left\langle s \right\rangle^{-\rho}.
\]
\end{proposition}
See Section  \ref{sec:ld-est} for the proof of Proposition \ref{thm:ld-est}.
 \begin{proposition}\label{identity}
  \begin{enumerate}
  \item For $f\in\mathcal{S}(\R)$, the following identity holds in  $L^2(\left\langle x \right\rangle^{-\rho}dx)$ for a fixed $\rho>0$:
   \begin{equation}\label{duHam-exp}
  \int_0^s e^{-i\sD(s-s_1)} A(s_1) \proj  f \ ds_1 =  \mathscr{G}(s)   f \, . \end{equation}
  Here,  $A(T)=\cos (\omega T)$ and $\mathscr{G}(t)$ denotes the operator
\begin{subequations}\label{eq:Gs_decomp}
 \begin{align}
 \mathscr{G}(s) &= \frac{1}{2i}e^{i\omega s}{\rm PV}\frac{1}{\sD +\omega}\proj +\frac{1}{2i}e^{-i\omega s}{\rm PV}\frac{1}{\sD -\omega}\proj  \label{PVs}\\
&\quad  +  \frac{\pi}{2} \left(e^{i\omega s} \delta(\sD+\omega )+e^{-i\omega s}\delta (\sD - \omega )\right)\proj  \label{deltas}\\
&\quad -  \frac12\frac{e^{-i\sD s}}{i(\sD+\omega -i0)}\proj \ -\ \frac12\frac{e^{-i\sD s}}{i(\sD-\omega  -i0)}\proj \, .
\label{propagators}\end{align}
\end{subequations}
  \item For $f_1, f_2\in \mathcal{S}(\R)$,  
 \begin{equation}
\left\langle  \proj f_1\ ,\ \int_0^s e^{-i\sD(s-s_1)} A(s_1) \proj  f_2\, ds_1 \right\rangle
 = \left\langle  \proj f_1, \mathscr{G}(s) \proj  f_2\right\rangle.
 \label{duHam-ip-exp} \end{equation}
\item  For $f\in \mathcal{S}(\R)$, we have  $\sup_{s\ge0}\| \langle  x \rangle ^{-\rho}\mathscr{G}(s) f\|_{L^2}<\infty$. 
\end{enumerate}
 \end{proposition}
 \begin{proof}[Proof of  Proposition \ref{identity}] The lemma is a consequence of the following calculation which holds in $\mathcal{S}^\prime(\R)$.
 For any fixed $s\geq 0$, we evaluate the integral via a regularization procedure.
 Using that $A(s)=\cos(\omega s)$ we have: 
 \begin{align*}
\int\limits_{0}^{s} e^{-i\sD (s-s_1)}A(s_1) \proj \, ds_1 &= \frac12 \int\limits_{0}^{s} e^{-i\sD s} \left( e^{i(\sD+\omega)s_1} +e^{i(\sD-\omega)s_1}\right)\proj \, ds_1 \\
\qquad &= \lim\limits_{\epsilon \to 0^+}\frac12 \int\limits_{0}^{s} e^{-i\sD s} \left( e^{i(\sD+\omega -i\epsilon)s_1} +e^{i(\sD-\omega -i\epsilon)s_1}\right)\proj \, ds_1 \\
&= \frac12 \lim\limits_{\epsilon \to 0^+}  \left[ \frac{e^{i\omega s}-e^{-i\sD s}}{i(\sD +\omega -i\epsilon )}  +\frac{e^{-i\omega s}-e^{-i\sD s}}{i(\sD -\omega  -i\epsilon)} \right]\proj \, .
\end{align*}
Recall  the distributional identity (Plemelj-Sokhotski relation)
\begin{equation} \frac{1}{x -i0}  = \lim_{\epsilon\to0^+}\frac{1}{x -i\epsilon} = {\rm PV} \frac{1}{x} + i\pi \delta (x) \, , 
\label{SP}\end{equation}
where ${\rm PV}$ denotes the Cauchy Principal Value.  Applying \eqref{SP} we obtain \begin{align*}
\int\limits_{0}^{s} e^{-i\sD (s-s_1)}A(s_1) \proj \, ds_1 &= \mathscr{G}(s). %\numberthis\label{eq:R_def}
\end{align*}
Here, $\mathscr{G}(s)$ is displayed in \eqref{eq:Gs_decomp}.
Finally, to prove the uniform bound in item (3), we return to the decomposition of $\mathscr{G}(s)$ as given in \eqref{eq:Gs_decomp}.
 That the operators in \eqref{propagators} satisfy the desired bound is a direct result of Proposition \ref{thm:ld-est}. The other terms \eqref{PVs}--\eqref{deltas}, can be written (and are obtained from) as $e^{\pm i\omega s}[\sD \pm \omega ]^{-1}$. The required bound can be reduced to the analogous bounds for $e^{\pm i\omega s}[Z_0 \pm \omega ]^{-1}$, where $Z_0$ is the free Schr{\"o}dinger Hamiltonian. The details of a similar argument are presented in Section \ref{sec:ld-est}, and in particular Section \ref{sec:decay_II}.

 \end{proof}
 
\subsection{Implementing the expansion through order $\beta^2$}\label{sec:expand}
A hierarchy of equations for the functions $G^{(n)}$ and $\alpha^{(n)}$ is obtained by substituting  \eqref{expansion} into \eqref{eq:ga_tausys} and  equating terms of like power in $\beta$. 
We carry this out
 to second order in $\beta$.
%We wish to substitute  \eqref{expansion} into \eqref{eq:ga_tausys} and  equate terms of
% like power in  $beta$ leads to a hierarchy of  equations for the functions $G^{(n)}$ and $\alpha^{(n)}$. 
The only non-algebraic term, $e^{i\beta _A (T)}$, is expanded to  $$e^{i\beta\eta_A(T)}=1+i\beta \eta_A(T) + \mathcal{O}(\beta^2)\, .$$ By recalling that $ \eta_A(T)$ is bounded, we obtain the hierarchy  equations, the first several of which we now display and solve.
 
\subsection*{$ \mathcal{O}(\beta^0)$ equation}
\begin{align}
	i\partial_T G^{(0)}&=0 \, ,\qquad  	\big(i\partial_T-\sD\big)\alpha^{(0)}=0\, .\label{beta0}
	\end{align}
	\medskip

\subsection*{		$\mathcal{O}(\beta^1)$ equation}
	\begin{subequations}
	\begin{align}
i\partial_T G^{(1)}&= -i\D_{\tau_1}G^{(0)}+ A(T)\left\langle\sigma_3\alpha_\star,\alpha^{(0)}\right\rangle \, , \label{beta1G}\\
	\big(i\partial_T-\sD\big)\alpha^{(1)}&=-i\D_{\tau_1}\alpha^{(0)} + \proj\sigma_3\alpha_\star A(T) G^{(0)}\nonumber\\
	&\qquad  + \proj\sigma_3\alpha_\star A(T) \alpha^{(0)} .\label{beta1alph}
	\end{align}
	\end{subequations}
	\medskip
	
\subsection*{ $\mathcal{O}(\beta^2)$ equation}
	\begin{subequations}
	\begin{align}
i\partial_T G^{(2)}&= -i\partial_{\tau_1}G_1-i\partial_{\tau_2}G^{(0)} + A(T)\left\langle\sigma_3\alpha_\star,\alpha^{(1)}\right\rangle\nonumber\\
&\qquad \qquad + i\eta_A(T)\left\langle\sigma_3\alpha_\star,\alpha^{(0)}\right\rangle \, , \label{beta2G}\\
	 \big(i\partial_T-\sD\big)\alpha^{(2)}&=-i\D_{\tau_1}\alpha^{(1)} -i\D_{\tau_2}\alpha^{(0)}  -i \proj\sigma_3\alpha_\star \eta_A(T) A(T) G^{(0)} \nonumber\\
&\qquad \qquad	 + \proj\sigma_3\alpha_\star A(T) \alpha^{(1)}   + \proj \sigma _3 \alpha _{\star} A(T) G^{(1)} .\label{beta2alph}
	\end{align}
	\label{beta2}
	\end{subequations}
	\[ \dots\dots\dots\]
	
		\noindent $\mathcal{O}(\beta^n):$\quad  \dots\dots\dots
 \bigskip\bigskip

 We next solve the equations of this hierarchy 
 through order $\beta^2$.

\subsection*{$\mathcal{O}(\beta^0)$ solution} Since $i\partial_T G^{(0)}=0$, we have that 
\begin{equation} G^{(0)}=G^{(0)}(\tau_1,\tau_2)  \, . \label{eq:G0_eq}\end{equation}
Furthermore, 
	$\big(i\partial_T-\sD\big)\alpha^{(0)}=0$ with initial data $ \alpha^{(0)}(0)=0 $ implies that we can take
\begin{equation} \alpha^{(0)}(T)\equiv 0 \, . \label{eq:alpha0_eq}
\end{equation}
%That  $\alpha^{(0)}$ vanishes leads to simplifications in the $\mathcal{O}(\beta)$ and $\mathcal{O}(\beta^2)$ equations
% in the hierarchy.
%
\subsection*{${\bf \mathcal{O}(\beta^1)}$ solution} Using that $\alpha^{(0)}(T)\equiv 0$, the next order of equations simplifies to
\begin{align}
	i\partial_TG^{(1)} &=  - i\D_{\tau_1}G^{(0)} \, ,\label{eq:G1_eq}\\
		\big(i\partial_T-\sD\big)\alpha^{(1)} &= \proj\sigma_3\alpha_\star A(T)G^{(0)} \, .  \label{eq:alpha1_eq}
\end{align}
Integrating
 \eqref{eq:G1_eq} and using that $G^{(1)}$ vanishes for $T=0$, we have
\[ iG^{(1)}(T,\tau_1,\tau_2) = - i\D_{\tau_1}G^{(0)}(\tau_1,\tau_2)T \, .\]
Condition \eqref{small-corrector-G} for $n=1$ then implies $\D_{\tau_1}G^{(0)}(\tau_1,\tau_2)=0$ and hence $i\partial_TG^{(1)}=0$.
Therefore,
\[ G^{(0)}= G^{(0)}(\tau_2)\quad {\rm and}\quad  G^{(1)}= G^{(1)}(\tau_1,\tau_2) \, .\]
 The degrees of freedom offered by $G^{(1)}(\tau_1,\tau_2)$ are not required to remove resonances in subsequent terms
 of the hierarchy; we however need to satisfy \eqref{ic-Gn} for $n=1$.
  We can therefore set
 \begin{equation}
G^{(1)}\equiv0 \, .
\label{G1eq0}
\end{equation}
%
%We assume that $G^{(0)}=G^{(0)}(\tau_2)$, i.e., that to a leading order, $G$ (or the original $g=\langle \alpha_{\star}, \alpha(T)\rangle$) changes on the time scale of $T\sim \varepsilon^{-2}$. We will see that our analysis is consistent with this assumption.
% Substituting $G^{(0)}=G^{(0)}(\tau_2)$ and $\alpha^{(0)}=0$ (see \eqref{eq:alpha0_eq}) into \eqref{eq:G1_eq} and \eqref{eq:alpha1_eq} yields
%\begin{align*}
%	i\partial_T G^{(1)} &= 0 &\Longrightarrow G^{(1)}=G^{(1)}(\tau_1, \tau_2) \, , \\
%	(i\partial_T - \sD)\alpha^{(1)} &= e^{-i\beta \eta_A (T)}\proj\sigma_3\alpha_\star A(T)G^{(0)} \, . &
%\end{align*}
We solve \eqref{eq:alpha1_eq} using Duhamel's principle and obtain, using Proposition \ref{identity}, that 
\begin{align}	
\alpha^{(1)}(T, \tau _2) &= -i\int_0^T e^{-i\sD(T-s)}\proj\sigma_3\alpha_\star A(s) \, ds\  G^{(0)}(\tau_2)\nonumber \\
&=  -i \mathscr{G}(T)\proj \sigma_3\alpha_\star\  G^{(0)}(\tau_2) \, ,\label{alpha1}\end{align}
where $\mathscr{G}(T)$ is defined in Proposition \ref{identity}. Hence, by item (3) of Proposition \ref{identity}, for $\rho$ taken sufficiently large we have 
\begin{equation}
 \| \left\langle X\right\rangle^{-\rho} \alpha^{(1)}(T, X,\tau _2)\|_{L^2(\R_X)} \lesssim  |G^{(0)}(\tau_2)| \, ,\ \quad {\rm for}\ T\ge0 \, .\label{alpha1-bd}
\end{equation}
Below we determine $G^{(0)}(\tau_2)=G^{(0)}(\beta^2 T)$, which we shall see is bounded on the time scale $\beta^{-2}$. Therefore, by \eqref{alpha1-bd},   \eqref{small-corrector-alph} is satisfied for $n=1$.

\subsection*{${\bf \mathcal{O}(\beta^2)}$ solution}
The system \eqref{beta2} reduces to   (using \eqref{eq:alpha0_eq} and \eqref{G1eq0})
\begin{subequations}
	\begin{align}
i\partial_T G^{(2)}&= -i\partial_{\tau_2}G^{(0)} + A(T)\left\langle\sigma_3\alpha_\star,\alpha^{(1)}(T,\tau_2)\right\rangle \, , \label{beta2bG}
\\
	\big(i\partial_T-\sD\big)\alpha^{(2)}&=   -i \proj\sigma_3\alpha_\star \eta_A(T) A(T) G^{(0)} + \proj\sigma_3\alpha_\star A(T) \alpha^{(1)} .\label{beta2ba}
	\end{align}
	\label{beta2b}
	\end{subequations}
	
Integration of \eqref{beta2bG}, using the initial condition $G^{(2)}(0)=0$, implies
\begin{equation}
 iG^{(2)}(T,\tau_2)  = \ T\ \left(\ -i\partial_{\tau_2}G^{(0)}(\tau_2) + \frac{1}{T}\int_0^T A(s)\left\langle\sigma_3\alpha_\star,\alpha^{(1)}(s,\tau_2)\right\rangle\ ds \right) . \label{G2solve}
 \end{equation}

Since we seek an expansion where $\sup_{0\le T\le \beta^{-2}} \beta^2 |G^{(2)}(T)| = o(1)$ as $\beta\downarrow0$ (see \eqref{small-corrector-G}), we require that
\begin{equation}
 i\partial_{\tau_2}G^{(0)}(\tau_2) = \lim_{T\to\infty} \frac{1}{T}\int_0^T A(s)\left\langle\sigma_3\alpha_\star,\alpha^{(1)}(s,\tau_2)\right\rangle\ ds \, .
\label{beta2-solve}
\end{equation}
Let's next study the inner product: $\left\langle\sigma_3\alpha_\star,\alpha^{(1)}(s,\tau_2)\right\rangle$ appearing in \eqref{beta2-solve}. By~\eqref{alpha1},
\begin{align}
&\left\langle\sigma_3\alpha_\star,\alpha^{(1)}(s,\tau_2)\right\rangle \nonumber  =   \left\langle\sigma_3\alpha_\star, -i \mathscr{G}(T)\proj \sigma_3\alpha_\star  \right\rangle   G^{(0)}(\tau_2) \, .\nonumber 
\label{ip1} \end{align}
Therefore, $G^{(0)}(\tau_2)$ satisfies:
\begin{equation}
 i\partial_{\tau_2}G^{(0)}(\tau_2) = \lim_{T\to\infty} \frac{1}{T}\int_0^T  A(s) \Upsilon(s)  ds\ G^{(0)}(\tau_2),
\label{G0-eqn}
\end{equation}
where
$$
\Upsilon(s) \equiv   \left\langle\sigma_3\alpha_\star, -i \mathscr{G}(T)\proj \sigma_3\alpha_\star  \right\rangle  \, .
$$
In order to evaluate the limit in \eqref{G0-eqn}, we apply Proposition \ref{identity} with  $f_1=f_2=\sigma_3\alpha_\star$ to  expand the expression for $\Upsilon(s)$
\begin{align*}
  & \Upsilon(s) = 
    \left\langle \proj \sigma_3\alpha_\star, -i \mathscr{G}(T)\proj \sigma_3\alpha_\star  \right\rangle  \numberthis  \label{ip2}\\
 &\quad  =
  - \frac{e^{i\omega s}}{2} \left\langle \proj \sigma_3\alpha_\star, {\rm PV}\frac{1}{\sD +\omega}
  \ \proj  \sigma_3\alpha_\star\right\rangle - \frac{e^{-i\omega s}}{2} \left\langle \proj \sigma_3\alpha_\star, {\rm PV}\frac{1}{\sD -\omega}
  \ \proj  \sigma_3\alpha_\star\right\rangle\\
  &\quad  -\frac{i\pi}{2} e^{i\omega s} \left\langle \proj\sigma_3\alpha_\star, \delta(\sD+\omega )  \proj \sigma_3\alpha_\star\right\rangle  - \frac{i\pi}{2} e^{-i\omega s} \left\langle \proj\sigma_3\alpha_\star, \delta(\sD-\omega )  \proj \sigma_3\alpha_\star\right\rangle\\
  &\qquad  +  \Upsilon_1(s) \, . 
 \end{align*}
 
The first four terms in $\Upsilon(s)$ lead to the resonant decay formulas \eqref{eq:gamma0}--\eqref{eq:Lambda0}.
The error term, $\Upsilon_1(s)$, is given by:
$$\Upsilon_1(s) \equiv  \frac12 \left\langle \proj \sigma_3\alpha_\star \, , \,   \left(\frac{e^{-i\sD s}}{(\sD+\omega  -i0)} \ +\ \frac{e^{-i\sD s}}{(\sD-\omega  -i0)}\right)\proj \sigma_3\alpha_\star \right\rangle  ,$$
and will now be shown to decay with $s$.

 Since $\alpha_\star$ is exponentially decaying (see \eqref{zero_mode_rep}), $\langle x \rangle^r \alpha_* \in L^2$ for  any fixed $r>0$, and therefore by inserting $\langle x\rangle^r \langle x \rangle^{-r}$ and applying Cauchy-Schwartz inequality, we have
\begin{align}
|\ \Upsilon_1 (t) | 
&\leq \Big\| \left\langle x\right\rangle ^r \sigma_3\alpha_\star \Big\|_{L^2} \cdot \sum_{\pm}  \Big\| \langle x\rangle^{-r} \frac{e^{-i\sD t}}{i(\sD \pm \omega  -i0)}\proj \sigma_3\alpha_\star \Big\|_{L^2} \ .
\label{Ups-bnd1}\end{align}  
 Here again, since $\alpha _*$ is exponentially decaying, there exists  $r>0$
such that Proposition \ref{thm:ld-est} is applicable and yields 
  \begin{equation} | \Upsilon_1(s)| \lesssim \left\langle s\right\rangle^{-\rho},\quad \textrm{ with $\rho>0$.} \label{Ups1-bd}\end{equation}

Next, substituting  $A(s)=\cos(\omega s)= \frac12(e^{i\omega s}+e^{-i\omega s})$ into \eqref{ip2} and using the bound 
\eqref{Ups1-bd}
we obtain
\begin{align}
&\lim_{T\to\infty} \frac{1}{T}\int_0^T A(s) \Upsilon(s) ds =  -i \Gamma_0 {- \Lambda_0 } \, ,
\label{thelimit}\end{align}
where (see also \eqref{eq:gamma0} and \eqref{eq:Lambda0})
\begin{align}
 \Gamma_0(\omega) &\equiv 
 \frac{\pi}{4} \left( 
  \left\langle \proj\sigma_3\alpha_\star, \delta(\sD+\omega )  \proj \sigma_3\alpha_\star\right\rangle \ 
+  \left\langle \proj\sigma_3\alpha_\star, \delta(\sD-\omega )  \proj \sigma_3\alpha_\star\right\rangle \right) \, , \label{fgr-Gam}\\
\Lambda_0(\omega) &\equiv
\left(
   \left\langle \proj \sigma_3\alpha_\star, {\rm PV}\frac{1}{\sD +\omega}
  \ \proj  \sigma_3\alpha_\star\right\rangle +
  \left\langle \proj \sigma_3\alpha_\star, {\rm PV}\frac{1}{\sD -\omega}
  \ \proj  \sigma_3\alpha_\star\right\rangle \right) \, . \label{fgr-Lam}
 \end{align}

 Equation \eqref{G0-eqn} and  \eqref{thelimit} imply

 \begin{equation}  \partial_{\tau_2} G^{(0)}(\tau_2) =  \left(-\Gamma_0 +i \Lambda_0\right) G^{(0)}(\tau_2) \, . \label{G0-eqn1}
 \end{equation}

The operators $\delta(\sD\pm\omega )$ are non-negative self-adjoint operators.
 Generically, since $\omega\in{\rm spec}_{ess}(\sD)$,  $\Gamma$ is strictly positive \cite[Section 4]{agmon1989perturbation}
 and hence $G^{(0)}(\tau_2)$ is exponentially decaying. Since $\tau_2=\beta^2 T$, the decay 
 is on the time scale $T\sim \beta^{-2}$.

 Returning now to the expression for $G^{(2)}(\tau_2)$ in \eqref{G2solve} we have
 
 \begin{equation}
 G^{(2)}(T,\tau_2) = 
 -iT\ \left(\ \frac{1}{T}\int_0^T A(s)\Upsilon(s)\ ds  - \left(-i\Gamma_0 - \Lambda_0\right)\ \right)\ G^{(0)}(\tau_2) \, .
 \label{G2solve1}\end{equation}
 
Using  \eqref{thelimit} we see that the expression in parenthesis in \eqref{G2solve1} tends to zero as $T\to\infty$.
  
  Recall from \eqref{expansion} and \eqref{G1eq0}  that
$G(T;\beta) \approx G^{(0)}(\tau_2) + \beta^2 G^{(2)}(T,\tau_2)$.
  By \eqref{G2solve1},  we have
  \[ \beta^2 G^{(2)}(T,\tau_2) = \beta^2 \ o(T) \times G^{(0)}(\tau_2)\quad\textrm{as $T\to\infty$}. \]
Therefore,
   \begin{equation} G(T;\beta) \approx \ \left( 1 +  \beta^2 \ o(T) \right)\  G^{(0)}(\beta^2 T)
    \ =\  \left( 1 +  \beta^2 \ o(T) \right)\  
e^{(-\Gamma_0 +i\Lambda_0)\beta^2 T} \, ,
  \label{Gexpand2}\end{equation}
which satisfies \eqref{small-corrector-G} through order $\beta^2$. 
\medskip

Finally $\alpha^{(2)}$, which satisfies \eqref{beta2ba}, can be bounded using  the bound on
 $\alpha^{(1)}$ in \eqref{alpha1-bd}, Proposition \ref{identity} and Theorem \ref{thm:ld-est}.
  Together with \eqref{Gexpand2} we have verified that our approximate solution 
  \begin{align}
G&\approx G^{(0)}+\beta G^{(1)} + \beta^2 G^{(2)}\, , \label{t2G}\\
\alpha_d&\approx\alpha^{(0)}+\beta \alpha ^{(1)} + \beta^2 \alpha ^{(2)} \,,
\label{t2alpha}\end{align}
satisfies  \eqref{small-corrector-G},  \eqref{small-corrector-alph}. This completes our derivation of the radiation damping effect on the time-scale $\beta^{-2}$.

\section{Proof of the time decay estimate of Proposition \ref{thm:ld-est} }\label{sec:ld-est}

We prove the following:  There exists $\rho>0$ and $r>0$ such that for any $f\in \mathcal{S}(\R)$  we have
\begin{equation}\label{eq:norm2decay}
 \sup_{t\ge0}\ t^{\rho }\ \Big\| \langle x\rangle ^{-r} \frac{e^{-i\sD t}}{i(\sD -\omega  - i0)} \proj f \Big\|_{L^2} < \infty \,.
\end{equation}
Here, $\proj$ denotes the continuous (dispersive) spectral part of $\sD$. 

Bounds similar to \eqref{eq:norm2decay} are proved in \cite{soffer1999resonances} for scalar  3D Klein-Gordon equations with 
spatially varying and decaying potentials. Our strategy is to make an algebraic reduction to a problem of this type and to apply appropriate 1D dispersive estimates. In particular, 
to prove \eqref{eq:norm2decay} (and hence the decay of $\Upsilon_1(t)$, see \eqref{Ups-bnd1}), we shall re-express  
 $e^{-i\sD t}$ in terms of a diagonal Klein-Gordon evolution operators to which we can apply  known time-decay estimates. 
 
 We begin with the an algebraic observation.
\begin{lemma}\label{lem:sim}
\begin{enumerate}
\item 
\begin{equation} 
 \sD ^2 = S\dD S^* \, , \label{eq:Dsimilarity}
 \end{equation}
\begin{equation*}
S \equiv   \frac{1}{\sqrt{2}} \left(\begin{array}{cc}
1 &1 \\ i & -i
\end{array} \right)\ ,\quad 
\dD \equiv \left(\begin{array}{cc}
Z_+ & 0 \\ 0 & Z_-
\end{array} \right) \, ,
\end{equation*}
and 
\begin{equation*}
Z_{\pm} \equiv -v_{\rm D}\partial_X^2 +\vartheta_{\sharp}^2\kappa ^2(X)  \mp v_{\rm D}\vartheta_{\sharp} \kappa ' (X) \, .
\end{equation*}
\item For any continuous function $\phi$ on the spectrum of $\sD^2$, 
\[ \phi(\sD^2)\proj(\sD) = S\phi(\dD)S^* \]
\end{enumerate}
\end{lemma}
Note that $Z_+$ and $Z_-$ are non-negative self-adjoint operators, by \eqref{eq:Dsimilarity}. It is also useful to note that  $Z_+$ and $Z_-$ can be expressed as  spatially localized perturbations of the constant coefficient operator 
 \begin{align} Z_0&= -\D_x^2+\kappa_\infty^2; \quad \textrm{namely}, \label{Z0-def}\\
Z_{\pm} &=Z_0  +\vartheta_{\sharp}^2(\kappa^2(X)-\kappa_\infty^2) \mp v_{\rm D}\vartheta_{\sharp} \kappa^\prime(X) \,.
\label{Z-exp} \end{align}

\begin{proof}[Proof of Lemma \ref{lem:sim}]
Using the commutation relation $\sigma_3\sigma_1=-\sigma_1\sigma_3$,  we obtain
 \begin{align} 
\nonumber
\sD^{2} &=(i v_{\rm D} \sigma_{3} \partial_{X}+\vartheta_{\sharp} \kappa(X) \sigma_{1})^{2} \\ 
%&=-\sigma_{3}^{2} v_{\rm D}^{2} \partial_{X}^{2}+i v_{\rm D} \vartheta_{\sharp} \sigma_{1} \sigma_{3} \kappa(X) \partial_{X}+i \vartheta_{\sharp} v_{\rm D} \sigma_{3} \sigma_{1} \partial_{X} \kappa(X)+\vartheta_{\sharp}^{2} \sigma_{1}^{2} \kappa^{2}(X) \\ 
&= I(-v_{\rm D}^{2} \partial_{X}^{2}+\vartheta_{\sharp}^{2} \kappa^{2}(X))- v_{\rm D} \vartheta_{\sharp} \sigma_{2} \ \kappa^{\prime}(X) \,.
\nonumber %\label{eq:sD^2}
\end{align}
Hence, $\sD^2$ is diagonalizable using the eigenvectors  of $\sigma_2$. 
\end{proof}

Define $\projpm _+ \equiv \proj (\sD >0)$ and $\projpm_- \equiv \proj (\sD <0)$ so that $\proj  = \projpm _+ + \projpm _-$.

\begin{proposition}\label{Dto-dD}
\[\sD \projpm _{\pm} f = \pm S\dD^{\frac12}S^* \projpm _\pm  f \]
and, in particular,
\begin{equation}\label{eq:ups_similar}
 \frac{e^{-i\sD t}}{i(\sD -\omega  -i0)} P_{\pm} f = S \frac{e^{\mp i\dD^{\frac12} t}}{i(\pm \dD ^{\frac12} -\omega  -i0)} S^* P_{\pm} f \, .
\end{equation}
\end{proposition}
\begin{proof}[Proof of Proposition \ref{Dto-dD}]
Since $\dD$ is similar to the positive semi-definite $\sD^2$ (positive definite on its continuous part), we have
 $ \sD \projpm _{\pm} f = \pm \left( S\dD S^*\right)^{\frac12} \projpm _{\pm} f \, .$
For simplicity, we restrict our attention to $f\in {\rm Range}(\projpm _+)$.  The case where $f\in {\rm Range}(\projpm _-)$ negative case follows analogously. 

Recall that for any positive-semidefinite operator $A$ we have the following formula\footnote{ This formula is obtained by using functional calculus and calculating the integral $\int_0^{\infty} z^{-1/2}(z+a)^{-1} \, dz$ for a scalar $a>0$.}
$$ A^{\frac12} = \pi A \int\limits_{0}^{\infty} z^{-\frac12} \left( zI+A\right)^{-1} \, dz  \, .$$
Therefore 
\begin{align*}
 \left(S\dD S^*\right)^{\frac12} &=  S\dD S^*\ \pi\ \int\limits_{0}^{\infty} z^{-\frac12} \left( zI+S\dD S^*\right)^{-1} \, dz  \\&= 
  S\dD S^*\  \pi \int\limits_{0}^{\infty} z^{-\frac12} \left(S( zI+\dD )S^*\right)^{-1} \, dz \\
 &= S\dD S^*\ S\left[ \pi \int\limits_{0}^{\infty} z^{-\frac12} \left( zI+\dD \right)^{-1} \, dz \right] S^* \\
  &= S \left[  \pi \dD \int\limits_{0}^{\infty} z^{-\frac12} \left( zI+\dD \right)^{-1} \, dz \right] S^* = S\dD^{\frac12}S^* \, .
\end{align*}
Hence,  $\sD \projpm _{\pm} f = \pm S\dD^{\frac12}S^* \projpm _\pm$. The equation \eqref{eq:ups_similar} now follows.
\end{proof}
%
%Using \eqref{Ups-def}, \eqref{Ups-bnd1} and Proposition \ref{Dto-dD} we have 
%\begin{align*}
%|\Upsilon_1 (s)| &\leq   \left\| \langle x\rangle^{-r} S \frac{e^{- i\dD^{\frac12} t}}{i( \dD ^{\frac12} -\omega +i0)} S^* \projpm _{+} f \right\|_{L^2}\\
% &\quad  +\left\| \langle x\rangle^{-r} S \frac{e^{+ i\dD^{\frac12} t}}{i(- \dD ^{\frac12} -\omega +i0)} S^* \projpm _{-} f \right\|_{L^2} \, .
%\end{align*}
%We bound the first term; the second can be treated similarly. 
% let $g:\mathbb{R}\to \mathbb{R}_+$ be a smooth function which is $1$ near $\omega$, and $0$ away from it, and let $g^c \equiv =1-g$. Let $G$ and $G^c$ be the spectral projections with respect to $\sD$ associated with $g$ and $g^c$, respectively. Then

\subsection{Proof of Proposition \ref{thm:ld-est}, time-decay estimate \eqref{eq:norm2decay}}

Let $J$ be a non-empty open interval in $\R$ containing zero  and let $\chi_{_J}(y)$ denote a smoothed out characteristic function with support in $J$. The support of $J$ will be fixed below. We write $\R=J+J^c$ and hence $1=\chi_{_J}(y)+\chi_{_{J^c}}(y)$. Therefore,
\begin{align}
& \Big\| \langle x\rangle^{-r} S \frac{e^{- i\dD^{\frac12} t}}{i( \dD ^{\frac12} -\omega  -i0)}  S^* \projpm _{+} f \Big\|_{L^2}
   \le \ \textrm{Term I} + \textrm{ Term II}\, , \qquad {\rm where}\nonumber\\
&\qquad \textrm{Term I}\ \equiv  \left\| \langle x\rangle^{-r} S \frac{e^{- i\dD^{\frac12} t}}{i( \dD ^{\frac12} -\omega  -i0)} 
\chi_{_J}(|\dD^{\frac12}-\omega|)S^* \projpm_{+} f \right\|_{L^2}\label{I-def}\\
&\qquad    \textrm{Term II} \equiv\ \left\| \langle x\rangle^{-r} S \frac{e^{- i\dD^{\frac12} t}}{i( \dD ^{\frac12} -\omega  -i0)} \chi_{_{J^c}}(|\dD^{\frac12}-\omega|)S^* \projpm_{+} f \right\|_{L^2} \,. \label{II-def}
 \end{align}
We next estimate the expressions $\textrm{Term I}$ and $\textrm{Term II}$. In particular, 
we show that for any $N$, there exists $r=r(N)$ such that 
\begin{align}
 \textrm{Term I} &\le t^{-N} \|\langle x\rangle^N f\|_{L^2}, \label{TI-bd}\\
  \textrm{Term II} &\le t^{-\frac14} \|f\|_{W^{1,\frac43}} \label{TII-bd}\,.
 \end{align}
We first bound  $\textrm{Term II}$ and then $\textrm{Term I}$.

\subsection{Time-decay estimates for $\textrm{Term II}$, given by \eqref{II-def}}\label{sec:decay_II}
\ {\ }
\bigskip

Let $q^{-1}+(p^\prime)^{-1}=1/2$. Then, 
\begin{align*}
\textrm{Term II} =&\left\| \langle x\rangle^{-r} S \frac{e^{- i\dD^{\frac12} t}}{i( \dD ^{\frac12} -\omega  -i0)} \chi_{_{J^c}}(|\dD^{\frac12}-\omega|)S^* P_{+} f \right\|_{L^2} \\
&\quad \le \|\langle x\rangle^{-r} \|_{L^q} \  
\left\| S \frac{e^{- i\dD^{\frac12} t}}{i( \dD ^{\frac12} -\omega  -i0)} \chi_{_{J^c}}(|\dD^{\frac12}-\omega|)S^* P_{+} f \right\|_{L^{p^\prime}}.
\end{align*}
Next, we prove a time-decay estimate for the latter factor.

For $j=+,-$, introduce wave operators $W_j$ which intertwine each $Z_j$ on its continuous spectral part \cite{weder2000inverse, yajima1995wk}, i.e., 
$ Z_j \proj (Z_j)=W_j Z_0 W_j^*$, and define the block diagonal operator $W={\rm diag}(W_+,W_-)$. 
For any Borel measurable function, $\phi$, we have 
\begin{equation} \phi(Z_\pm)\proj (Z_\pm)= W_\pm\phi(Z_0)W_\pm^* ,\label{wvop-sclr}
\end{equation}
and hence for matrix operator-valued Borel functions
\[ \phi(\dD)\proj (\dD)= W\phi(Z_0\sigma_0)W^* ,\]
where $\sigma_0$ is the $2\times2$ identity matrix. Therefore,
\begin{equation}
\phi(\sD^2)P_c(\sD)= S \phi(\dD) \proj (\dD) S^* = SW\phi(\dD_0)(SW)^* , 
\label{wvop-mat}\end{equation}
where $\dD_0=\sigma_0Z_0$. In particular, $(SW)(SW)^*=\proj (\sD)$.

 As shown in \cite{weder2000inverse, yajima1995wk}
the wave operators $W_j$, and hence $W$, are bounded in $W^{k,p}(\R^n)$. This then allows us to 
reduce decay estimates for $e^{- i\dD^{\frac12} t}\proj (\dD)$ to those for $e^{- iZ_0^{\frac12} t}$. 
Therefore, we have with $p^{-1}+(p^\prime)^{-1}=1$:
\begin{align*}
\textrm{Term II} &\lesssim
\left\| S W \frac{e^{- i\dD_0^{\frac12} t}}{i( \dD_0 ^{\frac12} -\omega -i0)} \chi_{_{J^c}}(|\dD_0^{\frac12}-\omega|)(SW)^* f \right\|_{L^{p^\prime}}\\
&\lesssim 
\left\| \frac{e^{- i\dD_0^{\frac12} t}}{i( \dD_0 ^{\frac12} -\omega  -i0)} \chi_{_{J^c}}(|\dD_0^{\frac12}-\omega|)(SW)^* f \right\|_{L^{p^\prime}}\\
&\lesssim 
\left\| e^{- i\dD_0^{\frac12} t}\right\|_{ L^{p^\prime}\leftarrow W^{1,p}}\
\left\|  ( \dD_0 ^{\frac12} -\omega)^{-1} \chi_{_{J^c}}(|\dD_0^{\frac12}-\omega|)(SW)^* f \right\|_{W^{1,p}}\\
&\lesssim  
\left\| e^{- i\dD_0^{\frac12} t}\right\|_{L^{p^\prime}\leftarrow W^{1,p}}\
\left\| \dD_0^{1\over2}\ ( \dD_0 ^{\frac12} -\omega)^{-1}\chi_{_{J^c}}(|\dD_0^{\frac12}-\omega|)(SW)^* f \right\|_{L^p}\\
&\lesssim  
\left\| e^{- i\dD_0^{\frac12} t}\right\|_{L^{p^\prime}\leftarrow W^{1,p}}\
\left\| \dD_0^{1\over2}\ ( \dD_0 ^{\frac12} -\omega)^{-1}\chi_{_{J^c}}(|\dD_0^{\frac12}-\omega|)\right\|_{L^p\leftarrow L^p}\ \left\|(SW)^* f \right\|_{L^p}.
\end{align*}
The second factor just above is bounded because it can be expressed in terms of a Fourier multiplier on $L^p$
 and the third factor is controlled by $\|f\|_{L^p}$ by boundness of wave operators. Hence,
 \[ \textrm{Term II} \lesssim \left\| e^{- i\dD_0^{\frac12} t}\right\|_{W^{1,p}\to L^{p^\prime}} \cdot  \|f\|_{L^p} \lesssim 
 \left\| e^{- iZ_0^{\frac12} t}\right\|_{W^{1,p}\to L^{p^\prime}} \cdot \|f\|_{L^p}  \,. \]
Dispersive time-decay estimates for the  1D Klein-Gordon equation yield: 
 \[ \|e^{- iZ_0^{\frac12} t} f\|_{L^\infty}\le t^{-1/2}\|(I-\D_x^2)^{1\over4}f\|_{L^1},\quad t\gg1;\] % \cite{brenner1985scattering}.
for general results, see, e.g., \cite{Egorova:2016}. 
 Together with the unitarity of the Klein-Gordon flow $\|e^{- iZ_0^{\frac12} t}f\|_{L^2}=\|f\|_{L^2}$ we have, using interpolation, that
 \begin{align*}
 \|e^{- iZ_0^{\frac12} t}f\|_{L^{p^\prime}} &\lesssim  \|e^{- iZ_0^{\frac12} t}f\|_{L^2}^{\frac{2}{p'}} \, \cdot \, \|e^{- iZ_0^{\frac12} t}f\|_{L^{\infty}}^{1-\frac{2}{p'}} \\
 &\lesssim  \|f\|_{L^2}^{\frac{2}{p'}} \, \cdot \, \left(t^{-\frac12}\|(I-\partial_x^2)^{\frac14}f\|_{L^1}\right)^{1-\frac{2}{p'}}\\
 &\lesssim t^{-\frac12 +\frac{1}{p'}} \,\|f\|_{L^2}^{\frac{2}{p'}} \, \cdot \, \|f\|_{W^{\frac12 , 1}}^{1-\frac{2}{p'}}\,,
 \end{align*}
  where $p^{-1}+(p^\prime)^{-1}=1$. Finally, let us fix $p=4/3$ and $p^\prime = 4$. Then, we have that 
  \[ \textrm{Term II}\le t^{-\frac14} \, \|f\|_{L^{\frac43}} \,  \cdot \, \|f\|_{L^2}^{\frac12} \, \cdot \, \|f\|_{W^{\frac12 , 1}}^{\frac12}.\]
   
%
%
%\begin{theorem}\label{thm:brenner}
%Let $p,p'$ satisfy $1/p + 1/p' =1$ and $s,s'$ satisfy $$3(1/2 - 1/p'))<1+s-s' \, ,$$ then
%$$\|e^{-iZ_0 ^{\frac12}t}F\|_{W^{s',p'}} \leq t^{-\varrho} \|F\|_{W^{s,p}} \, , \qquad \varrho =  \frac12 - \frac{1}{p'} \, ,$$
%for $t>1$.
%\end{theorem}

\subsection{Time decay estimate for $\textrm{Term I}$ given by \eqref{I-def}}

Recall that
\[
\textrm{Term I}\ \equiv  \left\| \langle x\rangle^{-r} S \frac{e^{- i\dD^{\frac12} t}}{i( \dD ^{\frac12} -\omega  -i0)} 
\chi_{_J}(|\dD^{\frac12}-\omega|)S^* P_{+} f \right\|_{L^2} \,.
\]
Noting that, by integrating the right hand side of \eqref{eq:mint} for $\varepsilon >0$ and taking the upper limit of the integral to $\tau \to \infty$,
\begin{align}
&\frac{e^{ -i\dD^{\frac12} t}}{i(\dD^{\frac12} -\omega  -i0)} \chi_J (|\dD^{\frac12} -\omega |)S^*P_+ \nonumber\\
&\quad  = - e^{-i\omega t}\ \lim_{\varepsilon\to0^+} \int\limits_{t}^{\infty} 
 e^{-i(\dD^{\frac12}-\omega-i\varepsilon)\tau} \chi_J(|\dD^{\frac12}-\omega |)S^*P_+ \,  d\tau \, ,\label{mint}
\end{align}
it suffices to show sufficient decay of an appropriate operator norm of the integrand of \eqref{mint}. Since we have $\dD={\rm diag}(Z_+,Z_-)$, time-decay bounds of the integrand in \eqref{mint} can be reduced to $e^{-iZ_\pm t}P_c(Z_\pm)$.
The Mourre approach to scattering estimates, as applied in \cite[Section 2]{soffer1999resonances}, yields
\begin{equation}
\|\langle x \rangle^{-r} S e^{-i\dD^{\frac12}\tau}\chi_J(|\dD^{1/2}-\omega|)S^*P_+f\|_{L^2} \lesssim \left(\langle t \rangle^{-r} + \langle t \rangle^{-N/2}  \right)\|\langle x \rangle^{N/2} f\|_{L^2} \, , \qquad N\in \mathbb{N} \, 
\label{mou-bd}
\end{equation}
where $r$ and $N$ can be taken large. The desired bound on the expression $\textrm{Term I}$ (see \eqref{I-def}) now follows from estimating the integrand of  \eqref{mint} using  \eqref{mou-bd} and integrating.

With the estimates \eqref{TI-bd} and \eqref{TII-bd} on $\textrm{Term I}$ and $\textrm{Term II}$ now proved,
 the proof of Proposition \ref{thm:ld-est} is now complete.

%\textcolor{blue}{  {\bf TODO:} add references, eg Miller-Soffer-W. paper re: rad. damping; decay estimates Erdogan-Green, Yajima,\dots}

% \textcolor{blue} {Give a quick recap of what we did above  and  perspective 
%\begin{enumerate}
%\item Analytical problem: Cannot phase out the time-dependent perturbation since $\sigma_1\D_x$ does not commute with $\sD$
% $\beta A(T)\D_x$ cannot be treated perturbatively. So we need decay estimates for the propagation associated
%with $\slashed{D}(T)$, e.g. decay estimates on the monodromy operator on its continuous spectral subspace.
%Hopefully one can use what we know about the corresponding questions for $\exp{-i\sD t}$
%\item Question to investigate: Suppose we now take the expansion that we have constructed, \eqref{t2G}-\eqref{t2alpha}, define 
%\[(G,\alpha)= \textrm{truncated expansion\ $+ g^3(\eta_G,\eta_\alpha)$}.\]
%Can we prove that 
%\[ \sup_{0\le t\le \hat{T}\beta^{-2}} g^3|(\eta_G,\eta_\alpha)|_{{\rm some\ norm}} = o(1)\quad \textrm{as $\beta\to0$?} \]
%Look at coupled system for the remainder $(\eta_G,\eta_\alpha)$ and estimate. 
%\item Other problems, future directions
%\end{enumerate}
%}

\appendix

\section{Derivation of the effective Dirac equation}\label{ap:eff_pf}
Our proof proceeds in two parts. We first formally derive the effective Dirac equation and the corrector equations using multiple scales methods. Then, using energy estimates, we bound the corrector term for large but finite times.

Introduce $X=\varepsilon x$ and $T=\varepsilon t$, the slow space and time variables, respectively. Formally, we view $\psi$ as dependent on all four variables, i.e., we seek $\Psi(t,T,x,X)$ such that $\psi(t,x)$, the true solution of the Schr{\"o}dinger equation \eqref{eq:lsA}, is well approximated by $\Psi(t, T, x, X)\Big|_{T=\varepsilon t, X=\varepsilon X}$. Hence, $\partial_x\mapsto\partial_x + \varepsilon\partial_X$ and $\partial_t\mapsto\partial_t + \varepsilon\partial_T$, and \eqref{eq:lsA} is rewritten into
\begin{subequations}\label{eq:lsa_ms}
\begin{equation}
i(\partial_t -H_0)\Psi(t,x,T,X) = \varepsilon H_1\Psi + \varepsilon^2 H_2 \Psi \, ,
\end{equation}
where $H_0= -\partial_x^2 +V(x)$ as usual, and
\begin{equation}
H_1\equiv -i\partial_T -2\partial_x \partial_X +\kappa(X)W(x)+ 2i A(T)\partial_x \, , \quad H_2 \equiv -\partial_X^2 +2i\beta A(T)\partial_X \, .
\end{equation}
\end{subequations}
Since we are looking for wavepackets which are spectrally centered at the Dirac point, the boundary conditions (or function space) in which we solve \eqref{eq:lsa_ms} is $k_{\rm D}=~\pi$-pseudo periodicity in $x$ and $L^2_X(\R )$ in $X$.

We seek $\Psi$ as an expansion in orders of $\varepsilon$,
\begin{equation}
    \psi=\psi^{(0)}(t,T,x,X)+\varepsilon\psi^{(1)}(t,T,x,X)+\varepsilon^2\psi^{(2)}(t,T,x,X)+ \eta^{\varepsilon}(t,x)\, .
    \label{eq:psi_ms}
\end{equation}
Substituting \eqref{eq:psi_ms} into \eqref{eq:lsa_ms} and collecting terms by orders of $\varepsilon$, at the $\varepsilon^0$ order we obtain the initial value problem
\begin{subequations}
\label{eq:O(1)}    
\begin{equation}
    \big(i\partial_t- H_0\big)\psi^{(0)}=0  \, ,
    \end{equation}
    \begin{equation}\label{eq:O(1)_ic}
    \psi^{(0)}(0,0,x,X)=\alpha_{*,1}(X)\Phi_1(x;k_{\rm D})+ \alpha_{*,2}(X)\Phi_2(x;k_{\rm D})\, ,
    \end{equation}
\end{subequations}
with $\alpha_{\star} \in L^2_X(\R;\C ^2)$ as defined in \eqref{eq:D0_def}.\footnote{Admittedly, the initial condition $\psi_{\star}$ is only well-approximated by $\alpha_{\star}^{\top}\Phi$, but since we are only formally exapnding the initial value problem and disregard small error term, we can take this approximation in \eqref{eq:O(1)_ic}.} As noted in Section~\ref{sec:bulk}, the $L^2_{k_{\rm D}}$-nullspace of $E_{\rm D}I-H_{\rm bulk}$ is spanned by the Bloch modes $\{\Phi_1,\Phi_2\}$, and so the solution of \eqref{eq:O(1)} is given by
\begin{equation}
    \psi^{(0)} = e^{-iE_{\rm D} t}\sum_{j=1}^2 \alpha_j(T,X)\Phi_j(x) \,,
\label{eq:O(1)_sol}
\end{equation}
where $\alpha_1(T,X)$ and $\alpha_2(T,X)$ are rapidly-decaying functions yet to be determined with $\alpha_j(0,X) =\alpha_{*,j}(X)$ for $j=1,2$.

Proceeding to order $\varepsilon$, we obtain 
\begin{equation}
    \big(i\partial_t-H_{0}\big)\psi^{(1)} = \big(-i\partial_T -2\partial_x\partial_X + \kappa(X)W(x)+2i A(T)\,\partial_x\big) \psi^{(0)}\, .
\label{eq:O(delta)}    
\end{equation}
Since $\psi^{(0)}$ oscillates with frequency $E_{\rm D}$, it will be convenient to extract the fast oscillatory behavior of $\psi^{(1)}$ by defining 
\[\psi^{(1)}(t,T,x,X)=e^{-iE_{\rm D}t}\widetilde{\psi}^{(1)}(T,x,X)\,.\]
Substituting the above ansatz and \eqref{eq:O(1)_sol} into \eqref{eq:O(delta)}, we obtain
\begin{align}
\nonumber
   \big(E_{\rm D}I-H_0\big)&\widetilde{\psi}^{(1)}(T,x,X)=-\sum_{j=1}^2 i\partial_T\alpha_j(T,X)\,\Phi_j +i\sum_{j=1}^2 \partial_X\alpha_j\cdot2i\partial_x\Phi_j \\
   &+ \sum_{j=1}^2\kappa(X)\alpha_j\,W(x)\Phi_j+\sum_{j=1}^2  A(T)\alpha_j\cdot2i\partial_x\Phi_j \,.
\label{O(delta)_subbed}    
\end{align}
The solvability of \eqref{O(delta)_subbed} for $\widetilde{\psi}^{(1)}\in L^2_{k_{\rm D}}$ requires the $L^2_{k_{\rm D}}$-orthogonality of its right-hand side to $\Phi_1$ and $\Phi_2$. In \cite[Propsition 2.2]{drouot2020defect}, it is shown that
\begin{subequations}\label{eq:sharp_coefs}
\begin{gather}
	\begin{pmatrix}
	\big\langle\Phi_1,2i\partial_x\Phi_1\big\rangle & \big\langle\Phi_2,2i\partial_x\Phi_1\big\rangle\\
	\big\langle\Phi_1,2i\partial_x\Phi_2\big\rangle & \big\langle\Phi_2,2i\partial_x\Phi_2\big\rangle
	\end{pmatrix} = v_{\rm D} \sigma_3 \,, \\
	\begin{pmatrix}
	\big\langle\Phi_1,W\Phi_1\big\rangle & \big\langle\Phi_2,W\Phi_1\big\rangle\\
	\big\langle\Phi_1,W\Phi_2\big\rangle & \big\langle\Phi_2,W\Phi_2\big\rangle\\
	\end{pmatrix} = \vartheta_\sharp \sigma_1\, ,\label{eq:vartheta}
\end{gather}
\end{subequations}
where $v_{\rm D}, \vartheta_\sharp\neq0$. Thus, the solvability conditions reduce to
\begin{align*}
i\partial_T\alpha_1 &= \big\langle\Phi_1,2i\partial_x\Phi_1\big\rangle(i\partial_X+\beta A(T))\alpha_1+\big\langle\Phi_1,W\Phi_2\big\rangle \kappa(X)\alpha_2 \, ,\\
i\partial_T\alpha_2 &= \big\langle\Phi_2,2i\partial_x\Phi_2\big\rangle(i\partial_X+\beta A(T))\alpha_2+\big\langle\Phi_2,W\Phi_1\big\rangle \kappa(X)\alpha_1 \, ,
\end{align*}
from which we obtain the Dirac equation
\begin{equation}\label{eq:Dirac_appendix}
    i\partial_T \alpha(T,X) = \left( iv_{\rm D}\sigma_3\partial_X+\vartheta_\sharp\kappa(X)\sigma_1+v_{\rm D} A(T)\sigma_3\right)\alpha \, ,
\end{equation}
with $\alpha(0,X) = \alpha_{\star}(X)$, as first introduced in \eqref{eq:diracA}.

Let $\pi^{\perp}$ denote the projection onto the orthogonal complement of ${\rm span}\{\Phi _1, \Phi_2 \}$ in $L^2 _{k_{\rm D}}$:
\[ \pi^\perp = I\ -\  \sum_{j=1}^2 \left\langle\Phi_j,\cdot\right\rangle \Phi_j\ =\ \sum_{j\ge3}\left\langle\Phi_j,\cdot\right\rangle \Phi_j;\] 
for convenience we indexed the $L^2_{k_{\rm D}}$ eigenpairs of $H^0$ such that
${\rm span}\{\Phi _1, \Phi_2 \}^\perp={\rm span}\{\Phi_j : j\ge3\}$.
  If $(\alpha_1,\alpha_2)$ is constrained to satisfy \eqref{eq:Dirac_appendix}, then  $\pi^{\perp} \left( H_1\Psi_0 \right) = H_1\psi^{(0)}$ and so:
\begin{equation}\label{eq:psi1_exp}
\widetilde{\psi}^{(1)}(t,x,T,X) = (E_D I - H^0)^{-1} H_1\psi^{(0)} +\varepsilon \sum\limits_{j=1,2}\beta _j(T,X) \Phi_j (x) \, ,
\end{equation}
where $\varepsilon\beta_j$ are decaying functions of $X$ which are to be determined at the next order equation (order $\varepsilon^2)$), and finally
$\psi^{(1)} =~\widetilde{\psi}^{(1)} e^{-iE_D t}$.

Turning next to the $\varepsilon^2$ order equations, we get 
\begin{equation}\label{eq:o2}
(i\partial_t - H_0)\psi^{(2)} = H_1 \psi^{(1)} + H_2 \psi^{(0)} \, , \qquad \psi^{(2)}(0,x,0,X) = 0 \, ,
\end{equation}
we again write $\psi^{(2)}(t,x,T,X) = \widetilde{\psi}^{(2)}(x,T,X) e^{-iE_Dt}$. In analogy with our first order analysis, the condition for solvability condition of \eqref{eq:o2} in $L^2_{k_{\rm D}}$ is that $H_1 (\widetilde{\psi}^{(1)}+\sum_j \beta _j \Phi_j )$ is orthogonal to $\Phi_1$ and~$\Phi_2$. In a manner analogous  to the derivation of \eqref{eq:diracA}, we obtain a system of {\em forced} Dirac equations for $\beta \equiv (\beta_1,\beta_2)^{\top}$:
\begin{equation}\label{eq:dirac2}
i\partial_T \beta(T,X) -  \slashed{D}_A(T)\beta(T,X) = F_2(T,X) \, , 
\end{equation}
where $F_2=\left(F_{2,1} , F_{2,2} \right)^\top$, and  for $j=1,2$:
\begin{equation}\label{eq:F2_def} \qquad F_{2,j} = \langle \Phi_j, H_1 (E_D-H_0)^{-1} H_1 \psi^{(0)} \rangle_{L^2(\Omega)} \, .
\end{equation}
We note that $F_2$ is independent of $\beta$, and is therefore a forcing term in \eqref{eq:dirac2}. 
Corresponding to any solution of the initial value problem for \eqref{eq:dirac2} in $C(\R;L^2(\R ;\C ^2))$,  we have that 
\begin{equation}\label{eq:psi_2}
\psi^{(2)} (t,x,T,X) = e^{-iE_Dt} (E_D-H_0)^{-1}\pi^{\perp}\widetilde{\psi}^{(1)}\, .
\end{equation}

\begin{remark}
In writing $H_1 \widetilde{\psi}^{(1)}$ and $H_1(E_D-H_0)^{-1}H_1\psi^{(0)}$ above, we apply the operator $\partial _X$ to $\kappa(X)$. This is the reason we require that the domain-wall function $\kappa$ have bounded derivatives of all orders - a sufficient, but perhaps not necessary condition. Such derivatives will be applied further in the subsequent sections without further notice.
\end{remark}
Finally, to close the multiple scales expansion \eqref{eq:psi_ms}, the corrector $\eta^\varepsilon$ is required to satisfy 
\begin{equation}\label{eq:corr_eqn}
\left(i\partial_t -H_0 -2i\varepsilon A(\varepsilon t)\partial_x -\kappa(\varepsilon x) W(x)\right)\eta^{\varepsilon}(t,x) = \varepsilon^3 \mathscr{F}^{\varepsilon}(t,x) \, ,
\end{equation}
where 
\begin{equation}\label{eq:correctorF}
\mathscr{F}^{\varepsilon} \equiv \left[ H_1\psi^{(2)}+H_2\psi^{(1)} + \varepsilon H_2\psi^{(2)} \right]\Big|_{T=\varepsilon t \, , \, X= \varepsilon x } \, .
\end{equation}

\subsection{Bounding the corrector, $\eta^\varepsilon(t,x)$}\label{sec:eta}
So far, \eqref{eq:psi_ms} is a formal multiple scale expansion. To estimate the quality of the approximation of $\psi^{\varepsilon}(t,x)$ by $\psi ^{(0)}$, we need to estimate the $L^2$ norms of $\psi^{(1)}$, $\psi ^{(2)}$, and the corrector $\eta ^{\varepsilon}$.

By self-adjointness of $H_0 +2i\varepsilon A(\varepsilon t)\partial_x +\kappa(\varepsilon x) W(x)$ on the left hand side operator in \eqref{eq:corr_eqn}, we have that
$ \partial_t \|\eta^{\varepsilon}(t)\|^2 = 2\varepsilon^3 {\rm Re}\ \langle \eta , \mathscr{F}^\varepsilon(t,\cdot) \rangle$.
This implies, by the Cauchy-Schwarz inequality, that
$\partial_t \|\eta^{\varepsilon}(t)\| \le  \varepsilon^3  \|\mathscr{F}^\varepsilon(t,\cdot)\|$. Therefore, for all $t\ge0$:
\begin{equation}\label{eq:eta_bdF}
\|\eta ^{\varepsilon}(t,\cdot) \|_{L^2(\R ^2)} \leq t\varepsilon ^3 \sup\limits_{s\in [0,t]}\|\mathscr{F}^\varepsilon(s, \cdot)\|_{L^2(\R ^2)} \, ,
\end{equation}
where $\mathscr{F}^\varepsilon$ is given by \eqref{eq:correctorF}. By bounding the norm of $\mathscr{F}^{\varepsilon}$ we obtain the following proposition:

\begin{proposition}\label{lem:bding_ab}
For all $t>0$, we have the following $L^2 (\mathbb{R}^2)$ bounds
\begin{align}
\|\eta^{\varepsilon}(t, \cdot)\|_{_{L^2 (\mathbb{R}_x^2)}} &\lesssim t\varepsilon^3 \sup\limits_{s\in [0,t]}\left(\|\alpha(\varepsilon s, \cdot) \|_{_{H^3(\mathbb{R}^2)}}  + \|\beta(\varepsilon s, \cdot)\|_{_{H^2(\mathbb{R}^2)}}\right) \, . \label{eq:eta_ab_bd}
\end{align}
Here, $\alpha = (\alpha _1 , \alpha _2 )^{\top}$ and $\beta = (\beta _1, \beta _2 ) ^{\top}$ are solutions of
 the homogeneous and inhomogeneous Dirac equations  \eqref{eq:Dirac_appendix} and \eqref{eq:dirac2}, respectively.
The implicit constant in \eqref{eq:eta_ab_bd} depends only on the functions $V,W,\kappa$, and $A$. 
\end{proposition}

%
%\begin{proposition}\label{lem:bding_ab}
%For all $t>0$, 
%\begin{equation}
%\|\eta(t, \cdot)\|_2 \lesssim t\varepsilon^3 \sup\limits_{S\in [0,T]}(\|\alpha(S, \cdot) \|_{H^3} + \|\beta(S, \cdot)\|_{H^2}) \, . \label{eq:eta_ab_bd}\end{equation}
%\end{proposition}
%
\begin{proof} We shall use the following convention. If $G(t,T,x,X)$ is a multi-scale function, then we write $G_\varepsilon(t,x)\equiv G(t,T,x,X)\Big|_{T=\varepsilon t, X=\varepsilon x}$. The proof of Proposition \ref{lem:bding_ab} makes use of the following bounds:

\begin{align}
\|\psi^{(1)}_{\varepsilon}(t,\cdot )\|_2 &\lesssim \|\alpha(\varepsilon t,\cdot) \|_{H^1} + \|\beta(\varepsilon t,\cdot) \|_2  \, , \label{eq:psi1_bd_ab} \\
\|\psi_{\varepsilon}^{(2)}(t,\cdot ) \|_2 &\lesssim \|\alpha(\varepsilon t, \cdot) \|_{H^1} + \|\beta(\varepsilon t, \cdot) \|_2  \, , \label{eq:psi2_bd_ab} 
\\
\|H_1 \psi_{\varepsilon}^{(2)}(t,\cdot )\|_2 &\lesssim \|\alpha(\varepsilon t, \cdot)\|_{H^2} + \|\beta(\varepsilon t, \cdot)\|_{H^1} \, ,\label{eq:H1psi2_bd_ab}
\\
\|H_2 \psi_{\varepsilon}^{(1)}(t,\cdot ) \|_2 &\lesssim\|\alpha(\varepsilon t, \cdot)\|_{H^3} + \|\beta(\varepsilon t, \cdot)\|_{H^2} \, ,\label{eq:H2psi1_bd_ab}
\\
\|H_2 \psi^{(2)} _{\varepsilon}(t,\cdot )\|_2 &\lesssim \|\alpha(\varepsilon t, \cdot)\|_{H^3} + \|\beta(\varepsilon t, \cdot)\|_{H^2} \, ,\label{eq:H2psi2_bd_ab}\\
\|\mathscr{F}^{\varepsilon}(\varepsilon t,\cdot)\|_2 &\lesssim \|\alpha (\varepsilon t, \cdot)\|_{H^3} + \|\beta(\varepsilon t, \cdot)\|_{H^2}  \, . \label{eq:F_ab_bd}
\end{align}
It will be useful to decompose $\widetilde{\psi}^{(1)}$ into two separate terms and bound each of these elements separately
$$\widetilde{\Psi}_{11}=(E_D-H_0)^{-1}\pi^{\perp}H_1\psi^{(0)} \, , \qquad \widetilde{\Psi}_{12} = \sum_j \beta _j \Phi_j \,  .$$
 We start with bounding $\|\tilde{\Psi}_{1,1}\|_2$. 
By definition
\begin{align*}
-H_1 \psi ^{(0)} = \varepsilon^{\frac12} \sumj 
\Big( &i\partial_T \alpha_j(T,X)\Phi_j +2\partial _x \Phi_j \cdot \partial _X \alpha_j(T,X) \\ 
&-2i\alpha _j(T,X) A(T)  \cdot \partial_x \Phi_j -\kappa(X)W(x)\alpha_j(T,X) \Phi_j\Big)
\end{align*}
where $ {\Phi} = (\Phi_1, \Phi_2 )^{\top}$. Since $(E_D  -H_0)^{-1}$ operates on $L^2_{k_{\rm D}}$, we can represent  $(E_D  -H_0)^{-1} \pi^\perp H_1 \psi^{(0)}$ in the basis of the Bloch modes $\{\Phi_b\}_{b\geq3}$. Using that, by the derivation of \eqref{eq:Dirac_appendix},
 we enforced that  $\pi^\perp H_1\Psi_0=H_1\Psi_0$, we have
\begin{align*}
&\left[\widetilde{\Psi}_{1,1}\right]_\varepsilon(t,x)=\widetilde{\Psi}_{1,1}(t,\varepsilon t,x,\varepsilon x)\\
&=(E_D-H_0)^{-1}H_1 \psi^{(0)}\Big|_{T=\varepsilon t, X=\varepsilon x}
\\ 
&=\varepsilon^{\frac12}\sum\limits_{b\geq 3} \sumj \frac{\Phi_b }{E_b -E_D}\Big\langle \Phi_b , \\
& \qquad \qquad -i\partial_T \alpha_j(T,X)\Phi_j -2\partial _x \Phi_j \cdot \partial _X \alpha_j(T,X) +i\alpha_j(T,X) A(T) \cdot \partial_{x} \Phi_j +\kappa(X)W(x)\Phi_j \Big\rangle_{L^2(\Omega)} \\
&=\varepsilon ^{\frac12}\sum\limits_{b\geq 3} \frac{\Phi_b }{E_b -E_D}
 \Big[-\langle \Phi_b , \Phi _j \rangle_{L^2(\Omega)}\cdot  \slashed{D}(T) \alpha(T,X)\\
 &\qquad\qquad \qquad  +\sumj\langle \Phi_b ,\partial_{x} \Phi_j \rangle_{L^2(\Omega)}\cdot (2\partial _X \alpha(T,X) -2 iA(T) \alpha_j(T,X) ) \Big]_{T=\varepsilon t, X=\varepsilon x} \, . \\
\end{align*}
We estimate $[\widetilde{\Psi}_{1,1}]_{\varepsilon}$ in $L^2(\mathbb{R})$ using that, since $\Phi_b \in L^{\infty}$, then $\|\Phi_b(\cdot) f(\cdot) \|_2 \lesssim \|f\|_2$ for any $f\in L^2$. Hence,
\begin{align}
\|[\widetilde{\Psi}_{1,1}]_{\varepsilon}(t,\cdot) \|_{L^2(\mathbb{R})} & \lesssim  \varepsilon^{\frac12}  \sum\limits_{b\geq 3}\sumj \frac{\|\Phi_b \|_{L^{\infty}} }{|E_b -E_D|} \big|\langle \Phi_b , \Phi_j \rangle_{_{L^2(\Omega)}}\big|\ \varepsilon^{-\frac12}\|  \partial_X\alpha_j (\varepsilon t,\cdot)\|_{L^2(\R _X)}\nonumber\\
&\qquad\quad + \varepsilon^{\frac12}  \sum\limits_{b\geq 3} \frac{\|\Phi_b\|_{L^{\infty}} }{|E_b -E_D|}| \big|\langle \Phi_b , \partial_x  \Phi_j\rangle _{_{L^2(\Omega)}}\big|\ \varepsilon^{-\frac12}\|\alpha_j(\varepsilon t,\cdot)\|_{H^1(\R _X)}\nonumber  \\ 
&\lesssim  \|\alpha(\varepsilon t, \cdot) \|_{H^1(\R _X)} \sumj \sum\limits_{b\geq 3}  \frac{\|\Phi_b\|_{\infty} }{|E_b -E_D|}\left( 
\big|\langle \Phi_b , \Phi_j \rangle_{_{L^2(\Omega)}}\big| + |\langle \Phi_b , \partial_x  \Phi_j\rangle _{_{L^2(\Omega)}}| \right) \, .
\label{Psi11-est}\end{align}
 By the Sobolev inequality   and the relation $\partial_x^2 \Phi_b = (V-E_b)\Phi_b$, we have (with $\Omega =[0,1]$) that any $b\ge1$: 
$$\|\Phi_b \|_{L^{\infty}(\Omega)} \lesssim \|\Phi _b \| _{H^2(\Omega ) } \lesssim \|\Phi_b \|_{L^2(\Omega)} + \|\partial _x ^2 \Phi_b \| _{L^2(\Omega)} =  \|\Phi_b \|_{L^2(\Omega)} + \|(V-E_b)\Phi_b \| _{L^2(\Omega)}\, .$$
Furthermore, since $V$ is bounded  
$\|\Phi_b \|_{L^{\infty}(\Omega)} \lesssim E_b \|\Phi_b\|_{L^2(\Omega)} \lesssim |b|^2 \, $, 
where we have used  that $\|\Phi _b\|=1$, and that $E_b \sim b^2$ as $b\to \infty$, by Weyl asymptotics  in one dimension. Hence, the factor $\|\Phi _b \|_{\infty} / |E_b- E_D|$ in \eqref{Psi11-est} is uniformly bounded for all $b$.
Therefore,  bounding $\|[\widetilde{\Psi}_{1,1}]_{\varepsilon} \|_{L^2(\mathbb{R})}$
 reduces to showing, for $j=1,2$:  
 \[\sum_{b\ge3} \big|\langle \Phi _b , \Phi _j \rangle_{_{L^2(\Omega)}}\big| +\big| \langle \Phi_b, \partial_{x} \Phi_j  \rangle _{_{L^2(\Omega)}} \big| < \infty\]
 We claim that both summands decay rapidly with $b$. Indeed, by the self-adjointedness of $H^0$, for $r=0,1$:
 \[
 \langle \Phi_b, \partial^r_{x} \Phi_j  \rangle = E_b^{-2}\langle (H_0)^2\Phi_b, \partial^r_{x} \Phi_j  \rangle
 =
E_b^{-2}\langle \Phi_b, (H_0)^2\partial^r_{x} \Phi_j  \rangle \, ,
 \]
 and therefore
 \[  \big|\langle \Phi_b, \partial^r_{x} \Phi_j  \rangle\big| \lesssim b^{-2} \| (H_0)^2\partial^r_{x} \Phi_j \| 
 \lesssim b^{-2} \|\Phi_j\|_{_{H^{4+r}}}, \qquad  r=0,1 \, , \]
 which is sufficient to ensure summability. It follows that 
\begin{equation}\label{eq:psi11_bd}
\|[\widetilde{\Psi}_{1,1}]_{\varepsilon}(t,\cdot )\|_{L^2(\mathbb{R}_x)} \lesssim \|\alpha(\varepsilon t,\cdot) \|_{H^1(\R_X)} \, .
\end{equation}
Together with the bound $\|[\widetilde{\Psi}_{1,2}]_{\varepsilon}(t,\cdot ) \|_2 \lesssim \varepsilon \|\beta (\varepsilon t,\varepsilon \cdot) \|_{L^2(\mathbb{R}_x^2)} \lesssim \|\beta(\varepsilon t,\cdot) \|_{L^2(\mathbb{R}_X)}$, we obtain \eqref{eq:psi1_bd_ab}. 
 
 The upper bounds \eqref{eq:psi2_bd_ab}--\eqref{eq:H2psi2_bd_ab} proceed in a similar fashion. The upper bound \eqref{eq:F_ab_bd} follows directly from the triangle inequality and $\mathscr{F} = H_1\Psi_2 + H_2\Psi_1 + \varepsilon H_2 \Psi_2$. Finally, we prove \eqref{eq:eta_ab_bd} by combining \eqref{eq:correctorF},  \eqref{eq:eta_bdF}, and \eqref{eq:F_ab_bd}.

\end{proof}

Proposition \ref{lem:bding_ab} provides upper bounds for $\eta^\varepsilon$, the expansion corrector, in terms of the Sobolev norms of $\alpha(T,X)$ and $\beta(T,X)$, which satisfy the Dirac equations \eqref{eq:Dirac_appendix} and \eqref{eq:dirac2}, respectively.  We now turn to estimating these norms. 
\begin{lemma}\label{lem:dirac_bds}
Let $\alpha$ satisfy and $\beta$ denote solutions of homogeneous and non-homogeneous Dirac equations
 \eqref{eq:diracA} and \eqref{eq:dirac2}. As initial data we take
 $\alpha(0,\cdot)=\alpha_0 \in H^4(\mathbb{R};\mathbb{C}^2)$ and $ \beta(0,x) \equiv 0$.
Then, for all $T>0$ 
\begin{align}
\|\alpha(T,\cdot ) \|_{H^s} &\lesssim T\|\alpha(0,\cdot) \|_{H^s} \, , \qquad &s>0 \, , \label{eq:alpha_hs_bd} \\
\|\beta(T,\cdot )\|_{H^s} &\lesssim T^2 \|\alpha(0, \cdot)\|_{H^{s+2}} \, , \qquad &s\geq 0 \, . \label{eq:beta_hs_bd}
\end{align}
\end{lemma}
\begin{proof} 
For $s>0$, we apply $\partial_X^s$ to \eqref{eq:Dirac_appendix} to get 
$$
i\partial_T \left(\partial_X ^s \alpha \right) = \sDT \partial^s _X \alpha + \sigma_1 \kappa ^{(s)} (X) \alpha \, .
$$
Next, we this equation from the left by the raw vector $ (\partial_X^s\bar{\alpha} _1,\partial_X^s \bar{\alpha}_2)$, subtract the complex conjugate equation, and integrate over $\R _X$. By the self-adjointednees of $\sDT$, we get
\begin{equation}\label{eq:alpha_s_bd_pre}
i\partial_T \|\partial_X^s \alpha \|_2^2 = \int\limits_{\R} \left[ \left( \partial_X^s \alpha(T,X) \right)^*\sigma_1 \kappa ^{(s)} (X) \alpha (T,X)  \,  - \,  {\rm c.c.}\right] \, dX  \, .
\end{equation}
We bound each term in the integrand on the right-hand side separately, e.g.,
\begin{align*}
\int\limits_{\R} {\partial_X^s}\bar{\alpha}_1(T,X)\kappa^{(s)}(X) \alpha _2 (T,X) \, dX &\leq \|\partial^s_X \alpha (T,\cdot \|_2 \cdot \|\kappa^{(s)} \|_{\infty} \|\alpha_2(T, \cdot) \|_2 \, .
\end{align*}
and so, by substituting back into \eqref{eq:alpha_s_bd_pre}, we get that 
\begin{equation}\label{eq:ODE_norm_alpha}
i\partial_T \|\partial_X^s \alpha \|_2 \leq \|\kappa^{(s)} \|_{\infty}  \|\alpha(T,\cdot ) \|_2 \, .
\end{equation}
Now, since the Dirac equation \eqref{eq:Dirac_appendix} is unitary, $ \|\alpha(T,\cdot ) \|_2 =  \|\alpha_0 \|_2 $, and so by integrating the differential inequality in \eqref{eq:ODE_norm_alpha} we get the upper bound in \eqref{eq:alpha_hs_bd}.

 To obtain \eqref{eq:beta_hs_bd}, we have from \eqref{eq:dirac2}, in $L^2(\mathbb{R};\C^2)$,  that
\begin{equation}
i\partial_T \|\partial_X^s \beta ( T, \cdot) \|_2 ^2 = 2~{\rm Im}\langle \beta ( T, \cdot) , \partial_X^{2s}F_2 \rangle =(-1)^s \cdot 2~{\rm Im}(\langle \partial_X^{s} \beta( T, \cdot)  , \partial_X^s F_2 \rangle) \, ,
\end{equation} 
where $F_2$ was defined in \eqref{eq:dirac_beta} and \eqref{eq:F2_def}. Therefore, by the Cauchy-Schwarz inequality, 
$\partial_T \|\partial_X^s \beta ( T, \cdot) \|_2 \leq \|F_2 (T, \cdot)\|_{H^s (\R ^2)} \, .
$
Finally, we bound the $H^s(\mathbb{R};\C^2)$ norm of $F_2$ .
Since, by definition
$$F_{2,j} = \langle \Phi_j, H_1 (E_D-H_0)^{-1}H_1\psi^{(0)} \rangle_{L^2_x(\Omega)} \, ,$$
we, as in Lemma \ref{lem:bding_ab}, using \eqref{eq:alpha_hs_bd} that  $\|F_2 (T, \cdot)\|_{H^s_x} \lesssim \|\alpha  (T, \cdot)\|_{H^{s+2}_X} = T\|\alpha_0 \|_{H^{s+2}_X}$. 
Thus, $$\|\partial_X^s \beta \|_{L^2(\R ^2)} \lesssim \int\limits_{0}^{\top} \|F_2 (T', \cdot ) \|_{H^{s+2}_X} \, dT' \lesssim T^2\|\alpha_0 \|_{H^{s+2}_X},$$
which proves \eqref{eq:beta_hs_bd}.
\end{proof}

We now complete the proof of Theorem \ref{thm:valid}. With the notation and definitions introduced above
our solution of the Schr{\"o}dinger equation \eqref{eq:lsa_gen} is:
\[ \psi^\varepsilon(t, x) = \psi^{(0)}_{\varepsilon} (t,x)  + \varepsilon\psi^{(1)}_{\varepsilon} (t,x) + \varepsilon^2 \psi ^{(2)}_{\varepsilon}(t,x) + \eta^\varepsilon(t,x). \]
We shall estimate the size of the corrector to the leading  order (effective Dirac) approximation:
$$\psi^\varepsilon(t, x) - \psi^{(0)}_{\varepsilon} (t,x)  = \varepsilon\psi^{(1)}_{\varepsilon} (t,x) + \varepsilon^2 \psi ^{(2)}_{\varepsilon}(t,x) + \eta^\varepsilon(t,x) \, .$$
Using Lemmas \ref{lem:bding_ab} and \ref{lem:dirac_bds} we have that
\begin{align*}
\|\psi^\varepsilon&(t,x)- \psi^{(0)}_{\varepsilon} (t,\cdot)\|_2  \lesssim \varepsilon\|\psi^{(1)}_{\varepsilon} (t,\cdot)\|_2 + \varepsilon^2 \|\psi_{\varepsilon}^{(2)} (t,\cdot )\|_2 + \|\eta^{\varepsilon}(t,\cdot)\|_2 \\
&\lesssim \varepsilon  \underbrace{(\|\alpha (\varepsilon t, \cdot)\|_{H^1} + \|\beta(\varepsilon t,\cdot)\|_2)}_{ \psi^{(1)}} + \varepsilon^2\underbrace{(\|\alpha(\varepsilon t, \cdot)\|_{H^1} +\|\beta(\varepsilon t,\cdot)\|_2)}_{ \psi^{(2)}} +\underbrace{t\varepsilon^3 (\|\alpha(\varepsilon t, \cdot) \|_{H^3} + \|\beta(\varepsilon t,\cdot)\|_{H^2})}_{\eta} \\
&\lesssim \varepsilon \left( \varepsilon t\|\alpha_0\|_{H^1} + \varepsilon^2 t^2 \|\alpha_0 \|_{H^2} \right) +\varepsilon^2 \left( \varepsilon t\|\alpha_0 \|_{H^1} + \varepsilon^2 t^2 \|\alpha_0 \|_{H^2} \right) + t\varepsilon^3 \left(\varepsilon t\|\alpha_0 \|_{H^3} +t^2\varepsilon^2 \|\alpha _0\|_{H^4} \right) \\
&\lesssim \varepsilon^2 t \|\alpha_0 \|_{H^1} + \varepsilon^3 t^2 \|\alpha_0 \|_{H^2} +t^2\varepsilon ^4 \|\alpha_0 \|_{H^3} + t^3\varepsilon ^5 \|\alpha_0 \|_{H^4} \, .
\end{align*}
Therefore, for any $\rho>0$ and $\varepsilon$ sufficiently small,  $\sup_{0\leq t\lesssim \varepsilon^{-(3/2-\rho)}} \|\psi^\varepsilon(t,x)- \Psi_0 (t,\cdot)\|_2 \lesssim \varepsilon^{\rho}$. This completes the proof of Theorem \ref{thm:valid}.

\section{The potentials $U^{(\ell)}_\varepsilon$ and their effective Dirac Hamiltonians}\label{ap:potentials}

To construct $U^{(1)}_{\varepsilon}$, we choose $V(x) =\cos (4\pi x)$, $W(x) = \cos (2\pi x)$, $\kappa(X) = \tanh(X)$, $A(T)=  \cos (\omega T)$, and $\varepsilon=\frac{1}{2}$. The initial value problem \eqref{eq:lsA} becomes
\begin{subequations}\label{eq:lsa_cos}
\begin{align}
	i\partial_t\psi &=\left[-\partial^2_x+\cos(4\pi x)+\tfrac{1}{2}\tanh(\tfrac{1}{2} x)\cos(2\pi x)+i\beta\cos(\tfrac{1}{2}\omega t)\partial_x\right]\psi \, ,
\\
\psi(0,x) &= \psi_\star(x) \, ,
\end{align}
\end{subequations}
where $ \beta, \omega >0$. 
The potential $U^{(2)}_{\varepsilon}$ models a chain of isolated and dimerized square wells, a more realistic model for an array of coiled optical waveguides as in Fig.\ \ref{fig:coiling}. Let  $\Theta (x;a)$ be a square function of radius $1\gg a>0$
$$	\Theta(x;a) \equiv 
	\begin{cases} 
	1, \quad & |x| \leq a \, ,\\
	0, \quad &|x| > a \, .
	\end{cases} 
$$
and define 
\begin{equation}
	Q_\pm(x;a) \equiv \sum_{z\in\mathbb{Z}} \Theta(x+z;a) \pm \Theta(x+z+\tfrac{1}{2};a).
\label{eq:Q_def}
\end{equation}
We choose $V(x) =-5Q_+(x;\tfrac{1}{20})$, $W(x) =-5Q_-(x;\tfrac{1}{20})$, $\kappa=\kappa_{\rm pw}$ where
\begin{equation}\label{eq:piecewise_kappa}
	\kappa_{\rm pw}(y) \equiv 
	\begin{cases} 
	0 \, , \quad & 0 \leq y < y_0 \, , \\
	\tfrac{1}{2}\, , \quad & y_0 \leq y < 2y_0 \, , \\
	1\, , \quad & y \geq 2y_0 \, , \\
	-\kappa_{\rm pw}(-y) \, , \quad & y<0 \, , 
	\end{cases}  \qquad  y_0 \equiv  \tanh^{-1}(\tfrac{1}{2}) \approx 0.549\, .
	\end{equation}
	Hence, the initial value problem \eqref{eq:lsA} becomes
\begin{equation}
	i\partial_t\psi(x,t)=\left[-\partial^2_x-5Q_+(x;\tfrac{1}{20})-\tfrac{5}{2}\kappa_{\rm pw}(\tfrac{1}{2} x)Q_-(x;\tfrac{1}{20})+i\beta\cos(\tfrac{1}{2}\omega t)\partial_x\right]\psi \, .
\label{eq:lsa_well}
\end{equation}
The potential $U^{(3)}_{\varepsilon}$ is similar to $U^{(2)}_{\varepsilon}$, and the only change is that we now set $\kappa (X)={\rm sgn}(X+1/2)$, the signum function, where the shift is introduced so as to avoid changes in the middle of one of the potential wells.
\begin{equation}
	i\partial_t\psi(x,t)=\left[-\partial^2_x-5Q_+(x;\tfrac{1}{20})-\tfrac{5}{2}{\rm sgn}(x+\tfrac14)Q_-(x;\tfrac{1}{20})+i\beta\cos(\tfrac{1}{2}\omega t)\partial_x\right]\psi \, .
\label{eq:lsa_sgn}
\end{equation}

The corresponding Dirac Hamiltonians are denoted by $\sD^{(\ell)}$, where  for $\ell=1$
\begin{equation}\label{eq:diracA_cos}
\eqref{eq:diracA} ~~{\rm with} \qquad v_{\rm D} =2\pi \, , \qquad \vartheta_{\sharp}=\frac12 \, , \qquad \kappa(X) = \tanh (X) ~,
\end{equation}
for $\ell =2$
\begin{equation}\label{eq:diracA_well}
\eqref{eq:diracA} ~~{\rm with} \qquad v_{\rm D} \approx 6.45 \, , \qquad \vartheta_{\sharp} \approx 1.03\, , \qquad \kappa(X) = \kappa_{\rm pw}(X) \, ,~~{\rm see}~~ \eqref{eq:piecewise_kappa}  ~.
\end{equation}
and for $\ell =3$
\begin{equation}\label{eq:diracA_well}
\eqref{eq:diracA} ~~{\rm with} \qquad v_{\rm D} \approx 6.45 \, , \qquad \vartheta_{\sharp}\approx 1.03\, , \qquad \kappa(X) = {\rm sgn}(X) \, ,~~{\rm see}~~ \eqref{eq:piecewise_kappa}  ~.
\end{equation}

\section{Numerical methods}\label{ap:numMethods}

To solve the initial value problems \eqref{eq:lsA} and \eqref{eq:diracA}, we use three-points stencil finite-difference discretization of the Hamiltonians, and the Crank-Nicholson method for time integration \cite{iserles2009first}. As for boundary conditions, we observe that for a computational domain $[-L, L]$ with $L\gg 1$ sufficiently large, both periodic and vanishing boundary conditions yield seemingly identical results. The defect modes $\psi_{\star}$ and $\alpha_{\star}$ were both computed using direct diagonlization of the corresponding discretized Hamiltonians.

To derive the Dirac equation \eqref{eq:diracA} from the Schr{\"o}dinger equation \eqref{eq:lsA}, we compute the parameters $v_{\rm D}$ and $\vartheta_{\sharp}$ in the following way: first, we construct a second-order finite difference matrix approximation of $H_{0, k}$, the Hamiltonian with $k$ quasi-periodic boundary conditions, see \eqref{eq:Hbulk_k}. The Bloch modes $\Phi_1$ and $\Phi_2$ are then the lower-most eigenvectors with $k=k_{\rm D}=\pi$. Second, the inner-products in  \eqref{eq:lambda} and \eqref{eq:vartheta} are computed using a trapezoid quadrature rule. 

Finally, whenever curve fitting is required, we use the SciPy \texttt{curve\_fit} nonlinear least squares framework.

\bibliographystyle{amsplain}
\bibliography{citation}

\providecommand{\bysame}{\leavevmode\hbox to3em{\hrulefill}\thinspace}
\providecommand{\MR}{\relax\ifhmode\unskip\space\fi MR }
\providecommand{\MRhref}[2]{%
  \href{http://www.ams.org/mathscinet-getitem?mr=#1}{#2}
}
\providecommand{\href}[2]{#2}
\begin{thebibliography}{10}

\bibitem{ablowitz2015adiabatic}
Mark~J Ablowitz, Christopher~W Curtis, and Yi-Ping Ma, \emph{Adiabatic dynamics
  of edge waves in photonic graphene}, 2D Materials \textbf{2} (2015), no.~2,
  024003.

\bibitem{ablowitz2017tight}
MJ~Ablowitz and Cole JT, \emph{Tight-binding methods for general longitudinally
  driven photonics lattices: Edge states and solitons}, Phys. Rev. Lett.
  \textbf{96} (2017), 043868.

\bibitem{agmon1989perturbation}
Shmuel Agmon, Ira Herbst, and Erik Skibsted, \emph{Perturbation of embedded
  eigenvalues in the generalized n-body problem}, Communications in
  mathematical physics \textbf{122} (1989), no.~3, 411--438.

\bibitem{asboth2014chiral}
J{\'a}nos~K Asb{\'o}th, B~Tarasinski, and Pierre Delplace, \emph{Chiral
  symmetry and bulk-boundary correspondence in periodically driven
  one-dimensional systems}, Physical Review B \textbf{90} (2014), no.~12,
  125143.

\bibitem{bal2021multiscale}
Guillaume Bal and Daniel Massatt, \emph{Multiscale invariants of floquet
  topological insulators}, arXiv preprint arXiv:2101.06330 (2021).

\bibitem{bellec2017non}
Matthieu Bellec, Claire Michel, Haisu Zhang, Stelios Tzortzakis, and Pierre
  Delplace, \emph{Non-diffracting states in one-dimensional floquet photonic
  topological insulators}, EPL (Europhysics Letters) \textbf{119} (2017),
  no.~1, 14003.

\bibitem{cayssol2013floquet}
J{\'e}r{\^o}me Cayssol, Bal{\'a}zs D{\'o}ra, Ferenc Simon, and Roderich
  Moessner, \emph{Floquet topological insulators}, physica status solidi
  (RRL)--Rapid Research Letters \textbf{7} (2013), no.~1-2, 101--108.

\bibitem{costin2001resonance}
Ovidiu Costin and Avraham Soffer, \emph{Resonance theory for schr{\"o}dinger
  operators}, Communications in Mathematical Physics \textbf{224} (2001),
  no.~1, 133--152.

\bibitem{dal2015floquet}
Virginia Dal~Lago, M~Atala, and LEF~Foa Torres, \emph{Floquet topological
  transitions in a driven one-dimensional topological insulator}, Physical
  Review A \textbf{92} (2015), no.~2, 023624.

\bibitem{Drouot:21}
A.~Drouot, \emph{The bulk-edge correspondence for continuous dislocated
  systems}, Annales de l'Institut Fourier (2021).

\bibitem{drouot2020defect}
Alexis Drouot, Charles~L Fefferman, and Michael~I Weinstein, \emph{Defect modes
  for dislocated periodic media}, Communications in Mathematical Physics
  \textbf{377} (2020), no.~3, 1637--1680.

\bibitem{eastham73spectral}
MSP Eastham, \emph{The {S}pectral {T}heory of {P}eriodic {D}ifferential
  {E}quations}, Scottish Academic Press, 1973.

\bibitem{Egorova:2016}
I.~Egorova, E.~Kopylova, V.~Marchenko, and G.~Teschl, \emph{Dispersion
  estimates for one-dimensional {S}chr\"odinger and {K}lein-{G}ordon equations
  revisited}, Russian Mathematical Surveys 71 \textbf{71} (2016), no.~3.

\bibitem{FLTW14}
Charles~L Fefferman, James~P Lee-Thorp, and Michael~I Weinstein,
  \emph{Topologically protected states in one-dimensional continuous systems
  and dirac points}, Proceedings of the National Academy of Sciences
  \textbf{111} (2014), no.~24, 8759--8763.

\bibitem{FW14}
Charles~L Fefferman and Michael~I Weinstein, \emph{Wave packets in honeycomb
  structures and two-dimensional dirac equations}, Communications in
  Mathematical Physics \textbf{326} (2014), no.~1, 251--286.

\bibitem{FLTW17}
C.L. Fefferman, J.~P. Lee-Thorp, and M.~I. Weinstein, \emph{Topologically
  protected states in one-dimensional systems}, Memoirs of the American
  Mathematical Society \textbf{247} (2017), no.~1173, 1--132.

\bibitem{fleury2016floquet}
Romain Fleury, Alexander~B Khanikaev, and Andrea Alu, \emph{Floquet topological
  insulators for sound}, Nature communications \textbf{7} (2016), no.~1, 1--11.

\bibitem{graf2018bulk}
Gian~Michele Graf and Cl{\'e}ment Tauber, \emph{Bulk--edge correspondence for
  two-dimensional floquet topological insulators}, Annales Henri Poincar{\'e},
  vol.~19, Springer, 2018, pp.~709--741.

\bibitem{guglielmon2018photonic}
Jonathan Guglielmon, Sheng Huang, Kevin~P Chen, and Mikael~C Rechtsman,
  \emph{Photonic realization of a transition to a strongly driven floquet
  topological phase}, Physical Review A \textbf{97} (2018), no.~3, 031801.

\bibitem{iserles2009first}
Arieh Iserles, \emph{A first course in the numerical analysis of differential
  equations}, no.~44, Cambridge university press, 2009.

\bibitem{ivanov2021topological}
Sergey~K Ivanov, Yaroslav~V Kartashov, Matthias Heinrich, Alexander Szameit,
  Lluis Torner, and Vladimir~V Konotop, \emph{Topological dipole floquet
  solitons}, Physical Review A \textbf{103} (2021), no.~5, 053507.

\bibitem{ivanov2020edge}
Sergey~K Ivanov, Yaroslav~V Kartashov, Lukas~J Maczewsky, Alexander Szameit,
  and Vladimir~V Konotop, \emph{Edge solitons in lieb topological floquet
  insulator}, Optics letters \textbf{45} (2020), no.~6, 1459--1462.

\bibitem{jotzu2014experimental}
Gregor Jotzu, Michael Messer, R{\'e}mi Desbuquois, Martin Lebrat, Thomas
  Uehlinger, Daniel Greif, and Tilman Esslinger, \emph{Experimental realization
  of the topological haldane model with ultracold fermions}, Nature
  \textbf{515} (2014), no.~7526, 237--240.

\bibitem{jurgensen2021quantized}
Marius J{\"u}rgensen, Sebabrata Mukherjee, and Mikael~C Rechtsman,
  \emph{Quantized nonlinear thouless pumping}, Nature \textbf{596} (2021),
  no.~7870, 63--67.

\bibitem{kirr2003metastable}
E~Kirr and MI~Weinstein, \emph{Metastable states in parametrically excited
  multimode hamiltonian systems}, Communications in mathematical physics
  \textbf{236} (2003), no.~2, 335--372.

\bibitem{KW03}
\bysame, \emph{Metastable states in parametrically excited multimode
  hamiltonian systems}, Communications in mathematical physics \textbf{236}
  (2003), no.~2, 335--372.

\bibitem{Kuchment:16}
P.A. Kuchment, \emph{An overview of periodic elliptic operators}, Bull. Amer.
  Math. Soc. \textbf{53} (2016), no.~3, 343--414.

\bibitem{LT-etal:16}
J.~P. Lee-Thorp, I.~Vuki{\'c}evi{\'c}, X.~Xu, J.~Yang, C.~L. Fefferman, C.~W.
  Wong, and M.~I. Weinstein, \emph{Photonic realization of topologically
  protected bound states in domain-wall waveguide arrays}, Phys. Rev. A
  \textbf{93} (2016), 033822.

\bibitem{lu2020dirac}
Jianfeng Lu, Alexander~B Watson, and Michael~I Weinstein, \emph{Dirac operators
  and domain walls}, SIAM Journal on Mathematical Analysis \textbf{52} (2020),
  no.~2, 1115--1145.

\bibitem{miller2000metastability}
Peter~D Miller, A~Soffer, and Michael~I Weinstein, \emph{Metastability of
  breather modes of time-dependent potentials}, Nonlinearity \textbf{13}
  (2000), no.~3, 507.

\bibitem{mukherjee2020observation}
Sebabrata Mukherjee and Mikael~C Rechtsman, \emph{Observation of floquet
  solitons in a topological bandgap}, Science \textbf{368} (2020), no.~6493,
  856--859.

\bibitem{ozawa2019topological}
Tomoki Ozawa, Hannah~M Price, Alberto Amo, Nathan Goldman, Mohammad Hafezi,
  Ling Lu, Mikael~C Rechtsman, David Schuster, Jonathan Simon, Oded Zilberberg,
  et~al., \emph{Topological photonics}, Reviews of Modern Physics \textbf{91}
  (2019), no.~1, 015006.

\bibitem{plotnik2013observation}
Y~Plotnik, M~C Rechtsman, D~Song, M~Heinrich, J~M Zeuner, S~Nolte, Y~Lumer,
  N~Malkova, J~Xu, A~Szameit, Z~Chen, and M~Segev, \emph{Observation of
  unconventional edge states in `photonic graphene'}, Nature materials
  \textbf{13} (2014), 57--62.

\bibitem{rechtsman2013photonic}
Mikael~C Rechtsman, Julia~M Zeuner, Yonatan Plotnik, Yaakov Lumer, Daniel
  Podolsky, Felix Dreisow, Stefan Nolte, Mordechai Segev, and Alexander
  Szameit, \emph{Photonic floquet topological insulators}, Nature \textbf{496}
  (2013), no.~7444, 196--200.

\bibitem{RS4}
M.~Reed and B.~Simon, \emph{Methods of modern mathematical physics: Analysis of
  operators, volume iv}, Academic Press, 1978.

\bibitem{roy2017periodic}
Rahul Roy and Fenner Harper, \emph{Periodic table for floquet topological
  insulators}, Physical Review B \textbf{96} (2017), no.~15, 155118.

\bibitem{rudner2013anomalous}
Mark~S Rudner, Netanel~H Lindner, Erez Berg, and Michael Levin, \emph{Anomalous
  edge states and the bulk-edge correspondence for periodically driven
  two-dimensional systems}, Physical Review X \textbf{3} (2013), no.~3, 031005.

\bibitem{sadel2017topological}
Christian Sadel and Hermann Schulz-Baldes, \emph{Topological boundary
  invariants for floquet systems and quantum walks}, Mathematical Physics,
  Analysis and Geometry \textbf{20} (2017), no.~4, 1--16.

\bibitem{SW21}
Amir Sagiv and Michael~I Weinstein, \emph{Effective gaps in continuous floquet
  hamiltonians}, arXiv preprint arXiv:2105.00958 (2021).

\bibitem{sakurai2020modern}
JJ~Sakurai and Jim Napolitano, \emph{Modern quantum mechanics}, Cambridge
  University Press, 2020.

\bibitem{SW:98}
A.~Soffer and M.I. Weinstein, \emph{Time dependent resonance theory},
  Geometrical and Functional Analysis \textbf{8} (1998), 1086--1128.

\bibitem{soffer1999resonances}
A~Soffer and Michael~I Weinstein, \emph{Resonances, radiation damping and
  instabilitym in hamiltonian nonlinear wave equations}, Inventiones
  mathematicae \textbf{136} (1999), no.~1, 9--74.

\bibitem{wang2013observation}
YH~Wang, Hadar Steinberg, Pablo Jarillo-Herrero, and Nuh Gedik,
  \emph{Observation of floquet-bloch states on the surface of a topological
  insulator}, Science \textbf{342} (2013), no.~6157, 453--457.

\bibitem{weder2000inverse}
Ricardo Weder, \emph{Inverse scattering on the line for the nonlinear
  klein--gordon equation with a potential}, Journal of mathematical analysis
  and applications \textbf{252} (2000), no.~1, 102--123.

\bibitem{yajima1995wk}
Kenji Yajima, \emph{The wk, p-continuity of wave operators for schr{\"o}dinger
  operators}, Journal of the Mathematical Society of Japan \textbf{47} (1995),
  no.~3, 551--581.

\end{thebibliography}

\end{document}